\documentclass[11pt,reqno,a4paper]{amsart}
\usepackage[dvipsnames, table]{xcolor}
\usepackage{amsthm}
\theoremstyle{plain}
\usepackage{amssymb}
\usepackage{marvosym}
\usepackage{bm}
\usepackage{mathrsfs}
\usepackage{enumerate}  
\usepackage{mathtools}
\usepackage{tikz-cd}
\usepackage{tikz}
\usepackage{graphicx}
\usepackage{float}
\usepackage[titletoc,title]{appendix}
\usepackage{marginnote}
\usepackage{todonotes}
\usepackage{etoolbox}
\usepackage[all]{xy}
\makeatletter
\patchcmd{\Ginclude@eps}{"#1"}{#1}{}{}
\makeatother
\usepackage[outdir=./]{epstopdf}
\usepackage[numbers]{natbib}
\usepackage[utf8]{inputenc}
\usepackage[english]{babel}
\usepackage{changepage}
\usepackage{makecell}
\usepackage{enumitem}
\usepackage{comment}

\usepackage[colorlinks,citecolor=blue]{hyperref}

\definecolor{lightblue}{HTML}{1F88CD}
\definecolor{lightgrey}{HTML}{727272}
\definecolor{lightblue2}{HTML}{009EC1}
\definecolor{mypink}{HTML}{FD00B0}
\definecolor{lightred}{HTML}{ff4d4d}

\usepackage{hyperref}
\hypersetup{
	colorlinks=true,
    linkcolor={OliveGreen},
    citecolor={blue},
	urlcolor={black}
}

\newtheorem*{theorem*}{Theorem}
\newtheorem{theorem}{Theorem}[section]
\newtheorem{corollary}[theorem]{Corollary}
\newtheorem{lemma}[theorem]{Lemma}

\newtheorem{proposition}[theorem]{Proposition}
\theoremstyle{definition}
\newtheorem{example}[theorem]{Example}
\theoremstyle{definition}
\newtheorem{definition}[theorem]{Definition}
\theoremstyle{definition}
\newtheorem{remark}[theorem]{Remark}
\theoremstyle{definition}

\theoremstyle{definition}

\theoremstyle{definition}

\theoremstyle{definition}
\newtheorem{question}[theorem]{Question}
\theoremstyle{definition}

\theoremstyle{definition}
\newtheorem{question!}[theorem]{Question!}
\theoremstyle{definition}

\makeatletter
\newcommand*\sbt{\mathpalette\sbt@{.75}}
\newcommand*\sbt@[2]{\mathbin{\vcenter{\hbox{\scalebox{#2}{$\m@th#1\bullet$}}}}}
\makeatother

\usepackage[OT2,T1]{fontenc}
\DeclareSymbolFont{cyrletters}{OT2}{wncyr}{m}{n}
\DeclareMathSymbol{\Sha}{\mathalpha}{cyrletters}{"58}

\newcommand{\ra}{\rightarrow}
\newcommand{\xra}{\xrightarrow}

\newcommand{\wt}{\widetilde}

\newcommand{\sst}{\subset}

\newcommand{\bR}{\bm{\mathrm{R}}}
\newcommand{\bL}{\bm{\mathrm{L}}}

\newcommand{\D}{\mathrm{D}}

\newcommand{\KK}{\mathrm{K}}

\newcommand{\ZZ}{\mathbb{Z}}
\newcommand{\NN}{\mathbb{N}}
\newcommand{\QQ}{\mathbb{Q}}

\newcommand{\CC}{\mathbb{C}}
\newcommand{\PP}{\mathbb{P}}

\newcommand{\NS}{\mathrm{NS}}

\newcommand{\prim}{\mathrm{prim}}

\newcommand{\ch}{\mathrm{ch}}

\newcommand{\td}{\mathrm{td}}

\newcommand{\ttop}{\mathrm{top}}

\newcommand{\Hdg}{\mathrm{Hdg}}
\newcommand{\HH}{\mathrm{HH}}
\newcommand{\HT}{\mathrm{HT}}

\newcommand{\Ktop}{\mathrm{K}_0^{\mathrm{top}}}
\newcommand{\pr}{\mathrm{pr}}

\newcommand{\Db}{\mathrm{D}^{\mathrm{b}}}
\newcommand{\Hilb}{\mathrm{Hilb}}
\newcommand{\Knum}{\mathrm{K}_{\mathrm{num}}}

\renewcommand{\Re}{\operatorname{Re}}
\renewcommand{\Im}{\operatorname{Im}}

\DeclareMathOperator{\im}{im}

\DeclareMathOperator{\Coh}{\mathrm{Coh}}

\DeclareMathOperator{\Sym}{Sym}
\DeclareMathOperator{\Ext}{Ext}

\DeclareMathOperator{\Hom}{Hom}

\DeclareMathOperator{\Pic}{Pic}

\DeclareMathOperator{\Stab}{Stab}
\DeclareMathOperator{\Gr}{Gr}

\newcommand{\GL}{\widetilde{\mathrm{GL}}^+_2(\mathbb{R})}
\newcommand{\cX}{\mathcal{X}}
\newcommand{\cY}{\mathcal{Y}}

\newcommand{\cA}{\mathcal{A}}
\newcommand{\cE}{\mathcal{E}}
\newcommand{\cU}{\mathcal{U}}
\newcommand{\cH}{\mathcal{H}}
\newcommand{\cB}{\mathcal{B}}
\newcommand{\cS}{\mathcal{S}}
\newcommand{\cK}{\mathcal{K}}
\newcommand{\cI}{\mathcal{I}}
\newcommand{\cT}{\mathcal{T}}

\newcommand{\Ku}{\mathcal{K}u}

\newcommand{\cD}{\mathcal{D}}
\newcommand{\cL}{\mathcal{L}}
\newcommand{\cN}{\mathcal{N}}
\newcommand{\cM}{\mathcal{M}}

\newcommand{\CH}{\mathrm{CH}}
\newcommand{\At}{\mathrm{At}}
\newcommand{\Dqc}{\mathrm{D}_{\mathrm{qc}}}
\DeclareMathOperator{\cF}{\mathcal{F}}

\DeclareMathOperator{\oh}{\mathcal{O}}
\newcommand{\rH}{\mathrm{H}}

\newcommand{\hf}[1]{\textcolor{purple}{#1}}
\usetikzlibrary{decorations.pathmorphing}
\begin{document}

\title[Atomic sheaves via Bridgeland moduli spaces]{Atomic sheaves on hyper-K\"ahler manifolds via Bridgeland moduli spaces}


\address{Shanghai Center for Mathematical Sciences, Fudan University, Jiangwan Campus, 2005 Songhu Road, Shanghai, 200438, P. R. China}
\email{hfguo@fudan.edu.cn}

\address{School of Mathematical Sciences, Zhejiang University, Hangzhou, Zhejiang Province 310058, P. R. China}
\email{jasonlzy0617@gmail.com}
%\urladdr{sites.google.com/view/zhiyuliu}

\author{Hanfei Guo, Zhiyu Liu}
\address{}
\email{}

\maketitle

\begin{abstract}
In this paper, we provide new examples of 1-obstructed and atomic sheaves on an infinite series of locally complete families of projective hyper-K\"ahler manifolds. More precisely, 

(1) we prove that the fixed loci of the natural anti-symplectic involutions on the moduli spaces of stable objects in the Kuznetsov component $\Ku(X)$ of a Gushel--Mukai fourfold $X$ are 1-obstructed Lagrangian submanifolds,

(2) we construct a family of immersed atomic Lagrangian submanifolds on each moduli space of stable objects in $\Ku(X)$, and

(3) we construct non-rigid projectively hyperholomorphic twisted bundles on any hyper-K\"ahler manifold of $\mathrm{K3^{[n]}}$-type for infinitely many $n$.

Additionally, we discuss examples of atomic Lagrangian submanifolds satisfying $b_1=20$ in a family of hyper-K\"ahler manifolds of $\mathrm{K3^{[2]}}$-type, as well as atomic sheaves supported on non-atomic Lagrangians.

% We also discuss examples of atomic Lagrangian submanifolds satisfying $b_1 =20$ in a family of hyper-K\"ahler manifolds of $\mathrm{K3^{[2]}}$-type. Additionally, we explore some atomic sheaves that are supported on non-atomic Lagrangians.
%We prove that the fixed loci of the natural anti-symplectic involutions on the moduli spaces of stable objects in the Kuznetsov components of Gushel--Mukai fourfolds are 1-obstructed Lagrangians. Subsequently, we construct a family of atomic sheaves on each moduli space of stable objects. We also discuss examples of atomic Lagrangian submanifolds satisfying $b_1=20$ in a hyper-K\"ahler manifold of $\mathrm{K3^{[2]}}$-type. 

%The proofs are based on the recent two constructions of Lagrangian subvarieties of moduli spaces of stable objects in K3 categories in \cite{ppzEnriques2023, FGLZ}. Our results provide new examples of 1-obstructed and atomic sheaves on an infinite series of families of polarized hyper-K\"ahler manifolds.

\end{abstract}

{
\hypersetup{linkcolor=blue}
\setcounter{tocdepth}{1}
\tableofcontents
}

\section{Introduction}
Hyper-K\"ahler manifolds are higher-dimensional analogues of K3 surfaces. 
% constituting one of the building blocks of K\"ahler manifolds with the trivial first Chern class by Beauville--Bogomolov decomposition theorem.
However, despite their well-studied structural properties, constructing concrete examples remains highly challenging. Currently, up to deformation, all known hyper-K\"ahler manifolds are arising from the moduli spaces of semistable sheaves on symplectic surfaces.

To construct examples of new deformation types, people have been exploring suitable objects on higher-dimensional hyper-K\"ahler manifolds and investigating the moduli spaces they form. A natural idea is to search for alternatives with properties similar to simple coherent sheaves on K3 surfaces.

Let $X$ be a projective hyper-K\"ahler manifold. One effort in this direction is the theory of modular sheaves introduced by O'Grady in \cite{ogrady:modular}, whose discriminant satisfies a certain numerical equation. Then, it is shown in \cite{ogrady:modular} that modular sheaves admit a wall-chamber decomposition in the ample cone and provide potential candidates for hyperholomorphic sheaves in the sense of \cite{Ver96hyperholomorphic}. Later, in \cite{markman:rank-1-obstruction}, the author studies a class of objects of $\Db(X)$, called \emph{1-obstructed objects} (cf.~Section \ref{subsubsec-1-obs}), that can infinitesimally deform along a codimension one subspace of the second Hochschild cohomology~$\HH^2(X)$. At the same time, for $F\in\Db(X)$, rather than $F$ itself, Beckmann investigates when the Mukai vector of $F$ can infinitesimally deform along~$\HH^2(X)$ in codimension one in \cite{beckmann:atomic-object}. If this is the case, such $F$ is called an \emph{atomic object} (cf.~Definition \ref{def-atomic}). By \cite[Theorem 6.13]{markman:rank-1-obstruction}, the extended Mukai vectors of atomic objects are as simple as the Mukai vectors of coherent sheaves on a K3 surface. Moreover, according to \cite[Proposition 1.5]{beckmann:atomic-object}, atomic torsion-free sheaves are modular. 

The aforementioned works have built rich theories and suggested that the first step in constructing hyper-K\"ahler manifolds of new deformation types is to find atomic or $1$-obstructed sheaves on higher dimensional hyper-K\"ahler manifolds with interesting moduli spaces. Among these objects, examples with the potential to achieve this goal are atomic sheaves supported on Lagrangians. 

More precisely, for a projective hyper-K\"ahler manifold $X$, 
we say an immersed Lagrangian submanifold is $1$-obstructed or atomic if the push-forward of its structure sheaf is (cf.~Definition~\ref{def-atomic-lag}). Assume that there exists a suitable connected atomic Lagrangian submanifold $L$ in~$X$ with the first Betti number $b_1(L)>0$ and consider the component $\cM$ of the moduli space of stable sheaves containing $\oh_L$. Then $\cM$ contains a smooth open dense subscheme $\cM^0$ that has a holomorphic symplectic structure. As discussed in \cite[Section 7.4]{markman:rank-1-obstruction}, a suitable compactification $\overline{\cM}$ of $\cM^0$ may yield a projective hyper-K\"ahler manifold with $b_2(\overline{\cM})\geq b_2(X)+1$ and $\dim \overline{\cM}=b_1(L)$.

In practice, starting from an atomic Lagrangian and applying derived equivalences and deformations, Bottini \cite{bottini:towards-OG10} constructs a stable atomic vector bundle $F_0$ on any hyper-K\"ahler fourfold of $\mathrm{K3^{[2]}}$-type, whose deformation space is of 10-dimensional. Furthermore, on a specific hyper-K\"ahler fourfold, the component of the moduli space of stable sheaves containing~$F_0$ is shown to be birational to a hyper-K\"ahler manifold of OG10-type. This result illuminates the potential for constructing hyper-K\"ahler manifolds of new deformation types within the aforementioned framework\footnote{After our paper appeared, Bottini successfully realizes OG-type manifolds as a moduli space of stable atomic bundles on $\mathrm{K3^{[2]}}$-type manifolds in \cite{bottini2024grady}. A way to construct (possibly singular) hyper-K\"ahler varieties is also proposed in \cite{LLX25}, which is applied to the Jacobian fibration in \cite{IM08cubic}.}.

In this paper, we aim to find out more non-rigid atomic Lagrangians in $\mathrm{K3^{[n]}}$-type hyper-K\"ahler manifolds. Based on the above discussion, if the $b_1$ of an atomic Lagrangian is neither $10$ nor $0$, then a smooth hyper-K\"ahler compactification $\overline{\cM}$ obtained through this method, if it exists, is not deformation equivalent to any known hyper-K\"ahler manifolds.

% If $X$ is of $\mathrm{K3}^{[n]}$-type, then $b_2(X)=23$. So if there is an atomic Lagrangian submanifold $L\subset X$ with $b_1(L)>0$, then it may lead to a hyper-K\"ahler manifold $\overline{\cM}$ with $b_2(\overline{\cM})\geq 24$ and $\dim \overline{\cM}=b_1(L)$ as discussed above. As the only known hyper-K\"ahler manifolds satisfy $b_2\geq 24$ are of OG10-type, to construct an example of the new deformation type using this method, the first step is to answer the following question:

% \begin{question}
% Is there a projective hyper-K\"ahler manifold $X$ of $\mathrm{K3^{[n]}}$-type such that there exists a connected atomic Lagrangian submanifold $L\subset X$ with
% \[\dim \rH^1(L, \CC)\notin \{0, 10\}?\]
% \end{question}

% In this paper, we give an affirmative answer to this question in Corollary \ref{cor-intro}. Our argument is based on a categorical method, which also allows us to construct infinitely many new examples of atomic and $1$-obstructed sheaves.

%\[\text{Find a hyper-K\"ahler manifold }X\text{ of }\mathrm{K3^{[n]}}\text{-type, such that there exists a atomic Lagrangian submanifold }L\subset X\text{ with }b_1(L)>0\text{ but }b_1(L)\neq 10\]

\subsection{Atomic Lagrangians in Bridgeland moduli spaces}
We focus on the hyper-K\"ahler manifolds arising as Bridgeland moduli spaces of stable objects in the Kuznetsov components of Gushel--Mukai (GM) fourfolds or sixfolds.

Given a GM variety $X$ of dimension $n\geq 3$, which is a quadric section of the projective cone over the Grassmannian $\mathrm{Gr}(2,5)$, there is a semiorthogonal decomposition
\[\Db(X)=\langle\Ku(X),\oh_X,\cU^{\vee}_X,\dots,\oh_X((n-3)H),\cU_X^{\vee}((n-3)H)\rangle,\]
where $\cU_X$ is the pull-back of the tautological subbundle on $\Gr(2, 5)$ and $\oh_X(H)$ is the ample generator of $\Pic(X)$. We refer to Section \ref{sec-pre} for a more detailed introduction.

% One of the most typical examples of hyper-K\"ahler manifolds are moduli spaces of stable sheaves on K3 surfaces (cf.~\cite{mukai:moduli-K3-I}). More recently, people consider categories that behave homologically as actual K3 surfaces, known as K3 categories, and construct projective hyper-K\"ahler manifolds of $\mathrm{K3}^{[n]}$-type from them as moduli spaces of stable objects with respect to suitable stability conditions. 

% An important example of K3 categories is the \emph{Kuznetsov component} 
%  $\Ku(X)$ of a Gushel--Mukai (GM) fourfold $X$. Given a GM variety of dimension $n\geq 3$, which is a quadric section of the projective cone over the Grassmannian $\mathrm{Gr}(2,5)$, there is a semiorthogonal decomposition
% \[\Db(X)=\langle\Ku(X),\oh_X,\cU^{\vee}_X,\dots,\oh_X((n-3)H),\cU_X^{\vee}((n-3)H)\rangle,\]
% where $\cU_X$ is the pull-back of the tautological subbundle on $\Gr(2, 5)$ and $\oh_X(H)$ is the ample generator of $\Pic(X)$. We refer to Section \ref{sec-pre} for a more detailed introduction.

\begin{itemize}
    \item When $n=4$ or $6$, there is a family $\Stab^{\circ}(\Ku(X))$ of stability conditions on $\Ku(X)$ constructed in \cite{perry2019stability}. The numerical K-group $\Knum(\Ku(X))$ contains a canonical rank two lattice $-A_1^{\oplus 2}$, generated by classes $\Lambda_1$ and $\Lambda_2$.

    \item When $n=3$ or $5$, there is a family of stability conditions on $\Ku(X)$ constructed in \cite{bayer2017stability}, now called Serre-invariant stability conditions. The lattice $\Knum(\Ku(X))$ is a rank two lattice, generated by classes $\lambda_1$ and $\lambda_2$.
\end{itemize}

\medskip

For a stability condition $\sigma_X$ on $\Ku(X)$, we denote by $M^X_{\sigma_X}(a,b)$ the Bridgeland moduli space that parametrizes S-equivalence classes of $\sigma_X$-semistable objects in $\Ku(X)$ of class $a\Lambda_1+b\Lambda_2$ (when $\dim X=4$ or $6$) or $a\lambda_1+b\lambda_2$ (when $\dim X=3$ or $5$). If $X$ is a GM fourfold or sixfold, by \cite[Theorem 1.5]{perry2019stability}, $M^X_{\sigma_X}(a,b)$ is a projective hyper-K\"ahler manifold of $\mathrm{K3^{[n]}}$-type when $\sigma_X\in \Stab^{\circ}(\Ku(X))$ and $\gcd(a,b)=1$ (see also Theorem \ref{thm-fourfold-moduli}).

\bigskip

Recently, there are two constructions of Lagrangian subvarieties in $M^X_{\sigma_X}(a,b)$ for a GM fourfold or sixfold $X$ when $\sigma_X\in \Stab^{\circ}(\Ku(X))$ is generic and $\gcd(a,b)=1$, demonstrated in \cite{ppzEnriques2023} and \cite{FGLZ}:

\begin{enumerate}
    \item There exists a natural anti-symplectic involution 
    \[\tau_X\colon M^X_{\sigma_X}(a,b)\to M^X_{\sigma_X}(a,b),\]
and its fixed locus $\mathrm{Fix}(\tau_X)$ is a Lagrangian submanifold by \cite{ppzEnriques2023} (cf.~Theorem \ref{thm-ppz}).

    \item \begin{enumerate}
        \item When $X$ is a general GM fourfold or a very general GM sixfold, for a general hyperplane section $Y\subset X$ and a Serre-invariant stability condition $\sigma_Y$ on $\Ku(Y)$, there is a finite unramified morphism
    \[r\colon  M^Y_{\sigma_Y}(a,b)\to M^X_{\sigma_X}(a,b)\]
    constructed in \cite{FGLZ} (cf.~Theorem \ref{thm-FGLZ-moduli}), whose image is a Lagrangian subvariety (possibly singular). 

    \item When $X$ is a very general GM fourfold and $X\hookrightarrow Y$ is an embedding such that $Y$ is a general GM fivefold, then for any Serre-invariant stability condition $\sigma_Y$ on $\Ku(Y)$, there is a finite unramified morphism
    \[r\colon  M^Y_{\sigma_Y}(a,b)\to M^X_{\sigma_X}(a,b)\]
    constructed in \cite{FGLZ} (cf.~Theorem \ref{thm-FGLZ-moduli}), whose image is a Lagrangian subvariety.
    \end{enumerate}
    
\end{enumerate}

The first main theorem of this paper shows that the above two constructions provide atomic sheaves supported on Lagrangians for each $M^X_{\sigma_X}(a,b)$. 

\begin{theorem}\label{thm-main}
Let $X$ be a GM fourfold or sixfold and $a,b$ be a pair of coprime integers. Fix a stability condition $\sigma_X\in \Stab^{\circ}(\Ku(X))$.

\begin{enumerate}[label=(\roman*)]
    \item \emph{(Theorem \ref{thm-fixed-loci-1-obs})} The fixed locus of the natural anti-symplectic involution
\[\mathrm{Fix}(\tau_X)\subset M^X_{\sigma_X}(a,b)\]
is an $1$-obstructed Lagrangian submanifold.

\item \emph{(Theorem \ref{thm-pushforward-atomic}, \ref{thm-push-atomic-5fold})} In the situation (2)(a) or (2)(b) above, the finite unramified morphism
\[r\colon M^Y_{\sigma_Y}(a,b)\to M^X_{\sigma_X}(a,b)\]
realizes $M^Y_{\sigma_Y}(a,b)$ as an immersed atomic Lagrangian submanifold of $M^X_{\sigma_X}(a,b)$.
%satisfies that $r_*\oh_{M^Y_{\sigma_Y}(a,b)}$ is an atomic sheaf on $M^X_{\sigma_X}(a,b)$, i.e.~$M^Y_{\sigma_Y}(a,b)$ is an immersed atomic Lagrangian of $M^X_{\sigma_X}(a,b)$.
%If $X$ is general and $Y\subset X$ is a smooth hyperplane section, then for any Serre-invariant stability condition $\sigma_Y$ on $\Ku(Y)$, there is a morphism \[r\colon M^Y_{\sigma_Y}(a,b)\to M^X_{\sigma_X}(a,b),\]
%such that $r_*\oh_{M^Y_{\sigma_Y}(a,b)}$ is an atomic sheaf on $M^X_{\sigma_X}(a,b)$.
\end{enumerate}
\end{theorem}

In the situation (2)(a), as the hyperplane section $Y\subset X$ can vary in the linear series $|\oh_X(H)|$, we obtain an $8$-dimensional family of immersed atomic Lagrangians on each $M^X_{\sigma_X}(a,b)$ when $\dim X=4$, and a $10$-dimensional family when $\dim X=6$. On the other hand, in the situation (2)(b), we also get a $10$-dimensional family\footnote{We expect that all these families are not isotrivial, as in known explicit examples.} of immersed atomic Lagrangians on each $M^X_{\sigma_X}(a,b)$ when $Y$ varies in the moduli of GM fivefolds containing $X$ (cf.~\cite[Lemma 5.4]{iliev2011fano}).

Currently, on a general projective hyper-K\"ahler manifold $X$, examples of $1$-obstructed sheaves are $\oh_X$ and skyscraper sheaves, up to derived equivalences. On the other hand, atomic Lagrangians only exist in specific examples, relying on the calculation of their Chern classes (cf. \cite[Section 3]{markman:rank-1-obstruction}, \cite[Section 8]{beckmann:atomic-object}). Our Theorem~\ref{thm-main} provides the first example of atomic sheaves supported on Lagrangians on an infinite series of $20$-dimensional locally complete families of projective hyper-K\"ahler manifolds of $\mathrm{K3^{[n]}}$-type, without utilizing the explicit geometry of the Lagrangians.

Subsequently, based on~\cite{bottini:towards-OG10,bottini23thesis,bottini2024grady}, starting from the atomic Lagrangians in Theorem \ref{thm-main}, we obtain projectively hyperholomorphic twisted bundles on a series of hyper-K\"ahler manifolds.

\begin{theorem}\label{thm-bundle}
Let $X$ be a hyper-K\"ahler manifold of $\mathrm{K3^{[n]}}$-type such that $n=a^2+b^2+1$ for a pair of coprime integers $a,b$. Then there exist non-rigid, stable, atomic, projectively hyperholomorphic twisted bundles on $X$.
\end{theorem}

\subsection{A non-rigid atomic Lagrangian submanifold}

In Section \ref{sec-new-example}, we apply Theorem \ref{thm-main} to an explicit family of hyper-K\"ahler fourfolds of $\mathrm{K3}^{[2]}$-type, called \emph{double dual EPW sextics}. It is first constructed in \cite{o2006irreducible} and systematically studied by \cite{o2006dual,o2010double,ogrady:period-double,ogrady:moduli-double-epw,iliev2011fano,debarre2020double,debarre2019gushel}. 

For a general GM fourfold $X$, there is an associated double dual EPW sextic $\wt{\mathsf{Y}}_{A(X)^{\perp}}$. Moreover, for a general hyperplane section $Y\subset X$ or a general GM fivefold $Y$ containing $X$, there is a smooth connected surface $\wt{\mathsf{Y}}_{A(Y)^{\perp}}^{\geq 2}$ with a finite unramified morphism 
$$i\colon\wt{\mathsf{Y}}_{A(Y)^{\perp}}^{\geq 2}\to \wt{\mathsf{Y}}_{A(X)^{\perp}},$$ which is shown to be Lagrangian (cf.~\cite[Proposition 5.2, 5.6]{iliev2011fano}). Moreover, these Lagrangians cover $\wt{\mathsf{Y}}_{A(X)^{\perp}}$ when $X$ is fixed and $Y$ varies. See Section \ref{subsec-double-epw} for a review of these constructions.

It has been demonstrated in \cite{JLLZ2021gushelmukai,GLZ2021conics} and Remark~\ref{rmk-5fold-epw} that both of these varieties can be described in terms of Bridgeland moduli spaces. Furthermore, it is calculated in \cite[Proposition 2.5]{debarre:GM-jacobian} and \cite[Proposition 0.5]{Log12} that $\rH^1(\wt{\mathsf{Y}}_{A(Y)^{\perp}}^{\geq 2}, \CC)=\CC^{20}$. By applying Theorem \ref{thm-main}, we establish:

\begin{theorem}[{Theorem \ref{thm-double-epw}, \ref{thm-double-epw-5fold}}]
Let $X$ be a general GM fourfold and either $Y\hookrightarrow X$ be a general hyperplane section or $Y\supset X$ be a general GM fivefold. Then $$i\colon\wt{\mathsf{Y}}_{A(Y)^{\perp}}^{\geq 2}\to \wt{\mathsf{Y}}_{A(X)^{\perp}}$$ is an immersed atomic Lagrangian submanifold, whose first Betti number $b_1$ equals to $20$.
\end{theorem}

See also \cite{ferretti:special-subvarieties} for a related study. As a corollary, the same argument as in Theorem \ref{thm-bundle} implies the following.

\begin{corollary}\label{cor-epw-bundle}
Let $X$ be a hyper-K\"ahler manifold of $\mathrm{K3^{[2]}}$-type. Then there exist stable, atomic, projectively hyperholomorphic twisted bundles on $X$ with a $20$-dimensional deformation space.
\end{corollary}

%In particular, such atomic Lagrangian submanifolds are parametrized by an open subset of $\PP^8$. Moreover, $\wt{\mathsf{Y}}_{A(X)^{\perp}}$ is covered by them (cf.~\cite{iliev2011fano}).

% Since $\rH^1(\wt{\mathsf{Y}}_{A(Y)^{\perp}}^{\geq 2}, \CC)=\CC^{20}$ by \cite[Proposition 2.5]{debarre:GM-jacobian} and \cite[Proposition 0.5]{Log12}, the above theorem implies the following corollary.

% \begin{corollary}\label{cor-intro}
% There exists a unirational locally complete $20$-dimensional family of projective hyper-K\"ahler fourfolds of $\mathrm{K3}^{\mathrm{[2]}}$-type such that each member $X$ in this family contains an $8$-dimensional family of connected atomic Lagrangian submanifolds $L$ with
% \[\rH^1(L, \CC)=\CC^{20}.\]
% \end{corollary}

As explained above (see also Section \ref{subsec-jacobian}), a hyper-K\"ahler compactification of an open dense subscheme $\cM^0$ of the component of the moduli space of stable sheaves on $\wt{\mathsf{Y}}_{A(X)^{\perp}}$ containing 
$$i_*\oh_{\wt{\mathsf{Y}}^{\geq 2}_{A(Y)^{\perp}}}\in \Coh(\wt{\mathsf{Y}}_{A(X)^{\perp}})$$ may yield a $20$-dimensional projective hyper-K\"ahler manifold of a new deformation type. Moreover, the support morphism of $\cM^0$ to the Hilbert scheme of Lagrangians in $\wt{\mathsf{Y}}_{A(X)^{\perp}}$ is closely related to the relative Jacobian fibration of GM fivefolds containing $X$. We hope that the techniques and ideas in \cite{LSV,bottini:towards-OG10,bottini23thesis,bottini2024grady} can be applied to this case to construct a suitable compactification of $\cM^0$. See Section \ref{subsec-jacobian} for a more detailed discussion. %We will address this in future work and anticipate that the modular construction outlined in \cite{bottini:towards-OG10} can be applied to this scenario.

%\zy{Rewrite....Explain the relation to relative Jacobians...}

% to construct such a compactification.

\subsection{Other examples}
Besides double dual EPW sextics, we also discuss other locally complete families of hyper-K\"ahler manifolds arising from GM fourfolds and cubic fourfolds. 

For a GM fourfold, there is another associated family of hyper-K\"ahler sixfolds, called double EPW cubes, constructed in \cite{IKKR19}. In \cite{FGLZ}, the authors construct Lagrangian submanifolds in a very general double EPW cube via the Hilbert schemes of twisted cubics on GM threefolds. In Corollary \ref{cor-double-epw-cube}, we prove that these Lagrangians are atomic. 

On the other hand, for a general cubic fourfold $X$, there are two families of hyper-K\"ahler manifolds associated with $X$: the Fano variety of lines $F(X)$ \cite{beauville:fano-variety-cubic-4fold} and the LLSvS eightfold~$\mathsf{Z}_X$ \cite{LLSvS17}. The hyper-K\"ahler fourfold $F(X)$ admits two different families of Lagrangian submanifolds. One of them is the Fano surface of lines in a general hyperplane section of $X$, which is atomic (cf.~\cite[Section 3]{markman:rank-1-obstruction}, \cite[Section 8]{beckmann:atomic-object}). Another family of Lagrangians is constructed from the Fano surface of planes in a general cubic fivefold containing $X$ (cf.~\cite{IM08cubic}). In Proposition \ref{prop-fanoplane-not-atomic}, we demonstrate that these Lagrangians are not atomic, but there exist atomic sheaves supported on them. For the LLSvS eightfold $\mathsf{Z}_X$, the Lagrangian submanifolds are constructed in \cite{shinder2017geometry}. Analogously, in Proposition \ref{prop-llsvs-not-atomic}, we prove that while these Lagrangians are not atomic, there exist atomic sheaves supported on them. As in Theorem \ref{thm-bundle} and Corollary \ref{cor-epw-bundle}, we also use these examples to construct stable atomic twisted bundles (cf.~Corollary \ref{cor-llsvs-bundle} and \ref{cor-fano-bundle}).

In \cite[Conjecture A.1]{FGLZ}, the authors speculate a version of the construction of Lagrangians in the moduli spaces of stable objects in the Kuznetsov components of cubic fourfolds, analogous to the case of GM fourfolds. Recently, this conjecture has been proved for very general cubic fourfolds by \cite[Theorem 8.2]{LLPZ:higher-dim-moduli}. We expect that these Lagrangians are not atomic in general, but one can find atomic sheaves supported on them, as illustrated by the examples in Section \ref{sec-llsvs}.

%\zy{Double EPW cube, LLSvS, Fano variety of lines...}

%\subsection*{Related work}

\subsection*{Plan of the paper}

In Section \ref{sec-pre}, we provide an overview of Kuznetsov components of GM varieties and stability conditions on them. In Section \ref{sec-moduli}, we focus on the properties of moduli spaces of stable objects in Kuznetsov components. Notably, we establish Proposition \ref{prop-restrict-BM-divisor}, which relates the natural divisor classes (called Bayer--Macr\`i divisors) on moduli spaces of semistable objects in Kuznetsov components of GM varieties of different dimension. In Section \ref{sec-atomic-object}, we revisit the theory of atomic and $1$-obstructed objects. Specifically, we generalize the results in \cite[Section 7.1]{beckmann:atomic-object} from connected Lagrangian submanifolds to non-connected immersed Lagrangians (cf.~Theorem \ref{atomic-criterion-thm}). 

In Section \ref{sec-atomic-sheaves}, we first show that the fixed locus of the natural anti-symplectic involution on a moduli space of stable objects in the Kuznetsov component of a GM fourfold or sixfold is $1$-obstructed (cf.~Theorem \ref{thm-fixed-loci-1-obs}). Then, using this, we prove in Corollary \ref{cor-c1-moduli} that the first Chern class of a moduli space of stable objects in the Kuznetsov component of a GM threefold or fivefold is proportional to the Bayer--Macr\`i divisor on it. This enables us to construct families of immersed atomic Lagrangians in each moduli space of stable objects in the Kuznetsov component of GM fourfolds or sixfolds (cf.~Theorem \ref{thm-pushforward-atomic} and Theorem \ref{thm-push-atomic-5fold}). Finally, we prove Theorem \ref{thm-bundle} using these atomic Lagrangians.

In Section \ref{sec-new-example}, we focus on double dual EPW sextics associated with general GM fourfolds and show that the natural Lagrangian submanifolds in them are non-rigid and atomic (cf.~Theorem~\ref{thm-double-epw} and Theorem~\ref{thm-double-epw-5fold}). Additionally, in Section \ref{subsec-cube}, we explore atomic Lagrangians within a family of projective hyper-K\"ahler sixfolds, known as double EPW cubes. 

In Section \ref{sec-discussion}, we first explain the relationship between the compactification of deformation spaces of atomic Lagrangians in double dual EPW sextics and the relative Jacobians of GM fivefolds (Section \ref{subsec-jacobian}). Subsequently, in Section \ref{sec-llsvs}, we provide two examples of atomic sheaves supported on non-atomic Lagrangian submanifolds constructed from cubic hypersurfaces (cf.~Proposition \ref{prop-llsvs-not-atomic} and \ref{prop-fanoplane-not-atomic}).

\subsection*{Notation and conventions} \leavevmode
\begin{itemize}
\item All schemes in our paper are over $\CC$. For a complex number $z\in \CC$, its imaginary part is denoted by $\Im(z)$. A variety is a pure-dimensional reduced separated scheme of finite type over $\CC$.

\item If $X$ is a projective hyper-K\"ahler manifold, we say a closed subvariety $L\subset X$ is a Lagrangian subvariety if the holomorphic two-form is zero when restricted to the smooth locus of $L$. We say $i\colon L\to X$ is an \emph{immersed Lagrangian submanifold} if $L$ is a smooth projective variety, $i$ is finite and unramified, and the image of $i$ is a Lagrangian subvariety of $X$.

\item For a smooth projective variety $X$, we denote by $\rH^{p,q}_{F}(X):=\rH^{p+q}(X, F)\cap \rH^{p,q}(X)$ for $F=\QQ$ or $\mathbb{R}$. For a cohomology class $\gamma$, we denote by $[\gamma]_k$ the component of $\gamma$ of degree~$k$.

\item For a triangulated category $\cD$, its Grothendieck group and numerical Grothendieck group are denoted by $\mathrm{K}(\cD)$ and $\Knum(\cD):=\KK(\cD)/\ker(\chi)$, respectively.

\item For a scheme $X$, we denote by $\Dqc(X)$ the unbounded derived category of $\oh_X$-modules with quasi-coherent cohomology. The full triangulated subcategory $\Db(X)\subset \Dqc(X)$ consists of pseudo-coherent complexes with bounded cohomology. If $X$ is noetherian,~$\Db(X)$ coincides with the bounded derived category of coherent sheaves on $X$.

\item For any smooth projective variety $X$ and an object $E\in \Db(X)$, we denote by $v(E)$ the Mukai vector of $E$.

\item Let $f\colon X\to Y$ be a morphism between schemes. All operations $f^*, f_*,$ and $\otimes$ in this paper are derived. We denote by $X_y:=f^{-1}(y)$ the fiber over a point $y\in Y$. Given an object $E\in \Dqc(X)$, we denote by $E_y$ the (derived) pull-back $i^*_y E$, where $i_y\colon X_y\hookrightarrow X$ is the inclusion.

\end{itemize}

\subsection*{Acknowledgements}

It is our pleasure to thank Emanuele Macr\`i and Giovanni Mongardi for many insightful comments and suggesting Proposition \ref{prop-general-criterion}, Proposition~\ref{prop-llsvs-not-atomic}(2), and Proposition~\ref{prop-fanoplane-not-atomic}(2). We would also like to thank Alessio Bottini, Sasha Kuznetsov, Chunyi Li, Alex Perry, Laura Pertusi, and Xiaolei Zhao for many useful discussions. We are especially grateful to Qizheng Yin for many helpful conversations and numerous suggestions. HG would like to express sincere gratitude to Zhiyuan Li for his academic support and encouragement. ZL would like to thank the Institute for Advanced Study in Mathematics at Zhejiang University for financial support and wonderful research environment. We thank
the anonymous referee for a careful reading as well as a list of useful suggestions that improved the exposition of the paper. ZL was partially supported by NSFC Grant 123B2002.

\section{Preliminaries}\label{sec-pre}

In this section, we review some basic definitions and properties of Kuznetsov components of GM varieties, as well as stability conditions on Kuznetsov components.

\subsection{Kuznetsov components}\label{subsec-ku}

%\zy{Ku, projection functor, relative Ku...}

Recall that a Gushel--Mukai (GM) variety $X$ of dimension $n$ is a smooth intersection $$X=\mathrm{Cone}(\Gr(2,5))\cap Q,$$ 
where $\mathrm{Cone}(\Gr(2,5))\subset \PP^{10}$ is the projective cone over the \text{Plücker} embedded Grassmannian $\Gr(2,5)\subset \PP^9$, and $Q\subset \PP^{n+4}$ is a quadric hypersurface. Then $n\leq 6$ and we have a natural morphism $\gamma_X\colon X\to \Gr(2,5)$. We say $X$ is \emph{ordinary} if $\gamma_X$ is a closed embedding, and \emph{special} if~$\gamma_X$ is a double covering onto its image. 

\begin{definition}
Let $X$ be a GM variety of dimension $n=4$ or $6$. We say $X$ is \emph{Hodge-special} if 
\[\mathrm{H}^{\frac{n}{2}, \frac{n}{2}}(X)\cap \mathrm{H}_{\prim}^{n}(X, \QQ) \neq 0,\]
where $\mathrm{H}_{\prim}^n(X, R)$ is defined as the orthogonal complement of $$\gamma_X^*\mathrm{H}^n(\Gr(2,5), R)\subset \mathrm{H}^n(X, R)$$ with respect to the intersection form for any commutative ring $R$.
\end{definition}

By \cite[Corollary 4.6]{debarre2015special}, $X$ is non-Hodge-special when $X$ is very general among all ordinary GM varieties of dimension $n$ or very general among all special GM varieties of dimension $n$.

The semiorthogonal decomposition of $\Db(X)$ for a GM variety $X$ of dimension $n\geq 3$ is given by 
\[\Db(X)=\langle\Ku(X),\oh_X,\cU^{\vee}_X,\cdots,\oh_X((n-3)H),\cU_X^{\vee}((n-3)H)\rangle,\]
where $\cU_X$ is the pull-back of the tautological subbundle of $\Gr(2,5)$ via $\gamma_X$. We refer to $\Ku(X)$ as the Kuznetsov component of $X$. 

When $n$ is even, $\Ku(X)$ is a K3 category, i.e.~$S_{\Ku(X)}=[2]$. When $n$ is odd, $\Ku(X)$ is an Enriques category, i.e.~$S_{\Ku(X)}=T_X\circ [2]$, where $T_X$ is a non-trivial auto-equivalence of $\Ku(X)$ satisfying $T^2_X=\mathrm{id}_{\Ku(X)}$.

When $n\geq 4$, we define the projection functor $$\mathrm{pr}_X:=\bR_{\cU_X}\bR_{\oh_X(-H)}\bL_{\oh_X}\bL_{\cU^{\vee}_X}\cdots \bL_{\oh_X((n-4)H)}\bL_{\cU^{\vee}_X((n-4)H)}\colon \D^b(X)\to \Ku(X)$$
where $\bR$ and $\bL$ are mutation functors (cf.~\cite[Definition 3.3]{BLMNPS21}).

%By \cite[Proposition 2.6]{kuznetsov2018derived}, the Serre functor $S_{\Ku(X)}=[2]$ when $\dim X$ is even, and $S_{\Ku(X)}=T_3\circ [2]$ for an involution $T_3$ of $\Ku(X)$ when $\dim X$ is odd.  

When $n=3$, according to the proof of \cite[Proposition 3.9]{kuznetsov2009derived}, $\Knum(\Ku(X))$ is a rank two lattice generated by $\lambda_1$ and $\lambda_2$, where
\begin{equation}\label{lambda}
    \ch(\lambda_1)=-1+\frac{1}{5}H^2,\quad \ch(\lambda_2)=2-H+\frac{1}{12}H^3,
\end{equation}
with the Euler pairing
\begin{equation}\label{eq-matrix-odd}
\left[               
\begin{array}{cc}   
-1 & 0 \\  
0 & -1\\
\end{array}
\right].
\end{equation} 

When $n=4$, there is a rank two sublattice in~$\Knum(\Ku(X))$ generated by $\Lambda_1$ and $\Lambda_2$, where
\begin{equation}\label{Lambda}
    \ch(\Lambda_1)=-2+(H^2-\gamma^*_X \sigma_2)-\frac{1}{20}H^4,
\quad \ch(\Lambda_2)=4-2H+\frac{1}{6}H^3,
\end{equation}
whose Euler pairing is 
\begin{equation}\label{eq-matrix-even}
\left[               
\begin{array}{cc}   
-2 & 0 \\  
0 & -2\\
\end{array}
\right],
\end{equation}
where $\gamma^*_X \sigma_2$ is the pull-back of the Schubert cycle $\sigma_2\in \mathrm{H}^4(\Gr(2, 5), \ZZ)$. When $X$ is non-Hodge-special, by \cite[Proposition 2.25]{kuznetsov2018derived}, we have $\Knum(\Ku(X))=\langle\Lambda_1, \Lambda_2\rangle$.

When $n=5$, by \cite[Corollary 6.5]{kuznetsov2019categorical}, we can find a smooth GM threefold $X'$ with an equivalence $\Ku(X')\simeq \Ku(X)$. Hence, $\Knum(\Ku(X))$ is also a rank two lattice with the Euler pairing \eqref{eq-matrix-odd}. 

When $n=6$, by \cite[Corollary 6.5]{kuznetsov2019categorical} again, we can find a smooth GM fourfold $X'$ with an equivalence $\Ku(X')\simeq \Ku(X)$. Hence, $\Knum(\Ku(X))$ contains a rank two lattice whose Euler pairing is the same as \eqref{eq-matrix-even}, which is the whole numerical Grothendieck group when $X$ is non-Hodge-special by \cite[Corollary 4.6]{debarre2015special} and \cite[Proposition 2.25]{kuznetsov2018derived}. 

We will use the following lemma. See e.g.~\cite[Lemma 4.11]{FGLZ}.

\begin{lemma}\label{lem-compute-class}
Let $X$ be a GM variety of dimension $n$ and $j\colon Y\hookrightarrow X$ be a smooth hyperplane section.
\begin{enumerate}
    \item If $n=4$, then $\pr_X(j_*\lambda_i)=\Lambda_i$ and $j^*\Lambda_i=2\lambda_i$ for each $i=1,2$.
    
    \item If $n=5$, we \emph{define} $\lambda_i$ to be the unique numerical class satisfying $j^*\lambda_i=\Lambda_i$ for each $i=1,2$. Then $\pr_X(j_*\Lambda_i)=2\lambda_i$ for each $i=1,2$ and $\Knum(\Ku(X))=\langle \lambda_1, \lambda_2\rangle$ with the Euler pairing \eqref{eq-matrix-odd}.
    
    \item If $n=6$, we \emph{define} $\Lambda_i:=\pr_X(j_*\lambda_i)$. Then the restriction of the Euler pairing to the sublattice $\langle \Lambda_1, \Lambda_2\rangle \subset \Knum(\Ku(X))$ is given by \eqref{eq-matrix-even}.
\end{enumerate}

\end{lemma}

We will also use the following relative version of Kuznetsov components. We say a smooth projective morphism $\pi\colon \cX\to S$ is \emph{a family of GM varieties of dimension $n$} if there exists a relative ample line bundle $\oh_{\cX}(1)$ such that $\cX_s$ is a smooth GM variety of dimension $n$ and~$\oh_{\cX_s}(1)$ is the ample generator of $\Pic(\cX_s)$ for each $s\in S$ (cf.~\cite[Definition 3.1]{debarre:GM-moduli}). By \cite[Lemma 5.9]{bayer2022kuznetsov}, when $S$ is noetherian and $n\geq 3$, there is a $S$-linear semiorthogonal decomposition 
\[\Db(\cX)=\langle \Ku(\cX), \pi^*(\Db(S))\otimes \oh_{\cX}, \pi^*(\Db(S))\otimes \cU^{\vee}_{\cX},\dots\]
\[\dots,\pi^*(\Db(S))\otimes \oh_{\cX}((n-3)H), \pi^*(\Db(S))\otimes \cU^{\vee}_{\cX}((n-3)H)\rangle\]
such that $\Ku(\cX)_s\simeq \Ku(\cX_s)$. We call $\Ku(\cX)$ \emph{the relative Kuznetsov component of $\cX$ over $S$}.

By \cite[Theorem 3.17, Lemma 3.25]{BLMNPS21} or \cite[Lemma 3.15]{perry:noncommutative-hpd}, there is also a relative projection functor
\[\pr_{\cX}\colon \Db(\cX)\to \Ku(\cX)\]
such that $(\pr_{\cX})_s=\pr_{\cX_s}$ for each $s\in S$.

\subsection{Hodge structures of Kuznetsov components}

%\zy{Knum, pull-back of Knum, Ktop, Mukai lattice, $A_1^{\oplus 2}$, Mukai vector, Hprim,...}

Following \cite[Section 2]{thomas:cubic-4fold}, for any GM variety $X$ of dimension $n=4$ or $6$, there is a sublattice
\[\wt{\rH}(\Ku(X), \ZZ):=\Ktop(\Ku(X))\subset \Ktop(X)\]
of the topological K-theory of $X$, where $\Ktop(\Ku(X))$ is the zero-degree homotopy group of the topological K-theory spectra defined in \cite{blanc:topo-K-noncommutative}. We denote by $\wt{\rH}(\Ku(X), F):=\wt{\rH}(\Ku(X), \ZZ)\otimes F$, where $F=\QQ, \mathbb{R}$, or $\CC$.

By \cite[Theorem 2.1]{thomas:cubic-4fold} and \cite[Proposition 3.4]{debarre2019gushel}, the Mukai vector induces an embedding
\[v\colon \wt{\rH}(\Ku(X), \ZZ)\hookrightarrow \Ktop(X)\hookrightarrow \rH^*(X, \QQ).\]
We define
\[\wt{\rH}_{\Hdg}(\Ku(X), \ZZ):=v^{-1}\big(\bigoplus_{p\geq 0} \rH^{p,p}_{\QQ}(X) \big).\]
The lattice $\wt{\rH}(\Ku(X), \ZZ)$ is equipped with a non-degenerate symmetric pairing $(-,-):=-\chi(-,-)$ and there exists a canonical homomorphism
\[\KK(\Ku(X))\to \wt{\rH}(\Ku(X), \ZZ),\quad [E]\mapsto [E]_{\ttop}.\]
This induces an embedding
\[\Knum(\Ku(X))\hookrightarrow \wt{\rH}_{\Hdg}(\Ku(X), \ZZ),\]
which is compatible with the pairing on both sides, up to a sign. Moreover, it is an isomorphism by \cite[Proposition 8.2]{perry2020integral}.

We denote by $A_1^{\oplus 2}\subset \wt{\rH}_{\Hdg}(\Ku(X), \ZZ)$ the image of the sublattice $\langle \Lambda_1, \Lambda_2\rangle\subset \Knum(\Ku(X))$ under the isomorphism above.

In the following, we will not distinguish the lattices $\Knum(\Ku(X))$ and $\wt{\rH}_{\Hdg}(\Ku(X), \ZZ)$.

Following \cite{thomas:cubic-4fold}, we can define a weight-two Hodge structure on $\wt{\rH}(\Ku(X), \ZZ)$ as
\[\wt{\rH}^{2,0}(\Ku(X), \CC):=v^{-1}(\rH^{n-1,1}(X)),\quad \wt{\rH}^{0,2}(\Ku(X), \CC):=v^{-1}(\rH^{1,n-1}(X)),\]
and
\[\wt{\rH}^{1,1}(\Ku(X), \CC):=v^{-1}\big(\bigoplus_{p\geq 0} \rH^{p,p}(X) \big).\]

\begin{lemma}[{\cite[Proposition 3.1]{pertusi2019double}}]\label{lem-iso-Hprime}
Let $X$ be a GM variety of dimension $n=4$ or $6$. Then the Chern character map induces a Hodge isometry
\[\ch\colon (A_1^{\oplus 2})^{\perp}\xra{\cong} \rH^n_{\prim}(X, \ZZ),\]
where the orthogonal is taken inside $\wt{\rH}(\Ku(X), \ZZ)$.
\end{lemma}

\begin{proof}
When $n=4$, this follows from \cite[Proposition 3.1]{pertusi2019double}. When $n=6$, using \eqref{Lambda} and the relations in Lemma \ref{lem-compute-class}, a similar argument as in \cite[Proposition 3.1]{pertusi2019double} yields the result.
\end{proof}

\subsection{Stability conditions on Kuznetsov components}

In this subsection, we review some basic notions and facts about stability conditions.

Let $\cD$ be a triangulated category and $\KK(\cD)$ be its K-group. Fix a surjective morphism to a finite rank lattice $v \colon \KK(\cD) \twoheadrightarrow \Lambda$. 

\begin{definition}\label{def-stability-condition}
A \emph{stability condition} on $\cD$ is a pair $\sigma = (\cA_{\sigma}, Z_{\sigma})$, where $\cA_{\sigma}$ is the heart of a bounded t-structure on $\cD$ and $Z_{\sigma} \colon \Lambda \ra \CC$ is a group homomorphism such that 
\begin{enumerate}
    \item for any $E \in \cA_{\sigma}$, we have $\Im Z_{\sigma}(v(E)) \geq 0$ and if $\Im Z_{\sigma}(v(E)) = 0$, $\Re Z_{\sigma}(v(E)) < 0$. From now on, we write $Z_{\sigma}(E)$ rather than $Z_{\sigma}(v(E))$.
\end{enumerate}
For any object $E \in \cA_{\sigma}$, we define the slope function $\mu_{\sigma}(-)$ as
\[
\mu_\sigma(E) := \begin{cases}  - \frac{\Re Z_{\sigma}(E)}{\Im Z_{\sigma}(E)}, & \Im Z_{\sigma}(E) > 0 \\
+ \infty , & \text{else}.
\end{cases}
\]
An object $0 \neq E \in \cA_{\sigma}$ is called $\sigma$-(semi)stable if for any proper subobject $F \sst E$, we have $\mu_\sigma(F) (\leq) \mu_\sigma(E/F)$. 
\begin{enumerate}[resume]
    \item Any object $E \in \cA_{\sigma}$ has a Harder--Narasimhan filtration in terms of $\sigma$-semistability defined above.
    \item There exists a quadratic form $Q$ on $\Lambda \otimes \mathbb{R}$ such that $Q|_{\ker Z_{\sigma}}$ is negative definite  and $Q(E) \geq 0$ for any $\sigma$-semistable object $E \in \cA_{\sigma}$.
\end{enumerate}
\end{definition}

%In this paper, we always take $\Lamba$ to be the numerical K-group $\Knum(\cD)$. In our paper, $\cD$ is a semiorthogonal component of $\Db(X)$ for a smooth projective variety $X$, so $\Knum(\cD)$ is a finite rank lattice. %Such stability conditions are called \emph{numerical stability conditions}.

Using techniques developed in \cite{bayer2017stability}, stability conditions are constructed on the Kuznetsov components for a series of Fano varieties. For GM varieties, the following result is proved in \cite{perry2019stability}.

\begin{theorem}[{\cite[Theorem 4.12]{perry2019stability}}]\label{blms-induce}
Let $X$ be a GM fourfold or sixfold. Then there exists a family of stability conditions $\Stab^{\circ}(\Ku(X))$ on $\Ku(X)$ with respect to the lattice $A_1^{\oplus 2}$.
\end{theorem}

In other words, for any object $F\in \Ku(X)$ and $\sigma=(\cA_{\sigma}, Z_{\sigma})\in \Stab^{\circ}(\Ku(X))$ such that $[F]_{\ttop}\in (A_1^{\oplus 2})^{\perp}$, we have $Z_{\sigma}(F)=0$.

\begin{remark}\label{rmk-stab-A1}
Let $X$ be a GM fourfold or sixfold and $j\colon Y\hookrightarrow X$ be a smooth hyperplane section. For $\sigma_X\in \Stab^{\circ}(\Ku(X))$ and a stability condition $\sigma_Y$ on $\Ku(Y)$ satisfying $$Z_{\sigma_X}(\Lambda_i)=Z_{\sigma_Y}(\lambda_i)$$ for each $i=1,2$, we claim that for $F\in \Ku(X)$, the equality holds:
\[2Z_{\sigma_Y}(j^*F)=Z_{\sigma_X}(F).\]
Indeed, in this case, we have $Z_{\sigma_Y}(j^*F)=Z_{\sigma_X}(\pr_X(j_*j^*F))$ by Lemma \ref{lem-compute-class}, then the claim follows from \cite[Lemma 4.10(1), Lemma 4.3(2)]{FGLZ}.

Similarly, if $X$ is a GM fivefold and $j\colon Y\hookrightarrow X$ is a smooth hyperplane section, then
\[Z_{\sigma_X}(\pr_X(j_*F))=2Z_{\sigma_Y}(F)\]
for any $F\in \Ku(Y)$.
\end{remark}

When $Y$ is a GM threefold, stability conditions on $\Ku(Y)$ are constructed in \cite{bayer2017stability}. It is proved in \cite{pertusiGM3fold} that they are \emph{Serre-invariant} (cf.~\cite[Definition 4.10]{JLLZ2021gushelmukai}). Furthermore, they all belong to the same $\GL$-orbit\footnote{The action of $\GL$ on stability conditions are defined in \cite[Lemma 8.2]{bridgeland:stability}}. Using the equivalence in \cite[Corollary 6.5]{kuznetsov2019categorical}, the same result holds for GM fivefolds.

\begin{theorem}[{\cite[Theorem A.10]{JLLZ2021gushelmukai}}]\label{thm-unique-fourfold}
Let $Y$ be a GM threefold or fivefold, then all Serre-invariant stability conditions on $\Ku(Y)$ are unique up to the $\GL$-action.
\end{theorem}

Roughly speaking, stability conditions in the same $\GL$-orbit share the same moduli stack of semistable objects. If we are only interested in moduli stacks or spaces of semistable objects, there is no harm in changing stability conditions in a single $\GL$-orbit. 

For non-Hodge-special GM varieties of even dimension, the following result is proved in \cite[Proposition 4.12]{FGLZ}.

\begin{theorem}\label{thm-unique-threefold}
Let $X$ be a GM fourfold or sixfold. If $X$ is non-Hodge-special, then stability conditions on $\Ku(X)$ are unique up to the $\GL$-action.
\end{theorem}

In general, we have:

\begin{theorem}[{\cite[Corollary 8.9]{ppzEnriques2023}}]\label{thm-stabo-orbit}
Let $X$ be a GM fourfold or sixfold. Then any two stability conditions in $\Stab^{\circ}(\Ku(X))$ are in the same $\GL$-orbit.
\end{theorem}

In our paper, we only consider stability conditions introduced in Theorem \ref{blms-induce} and \ref{thm-unique-fourfold}. Therefore, by Theorem \ref{thm-unique-threefold} and \ref{thm-stabo-orbit}, we can freely vary the stability conditions without changing the stability of objects.

\section{Bridgeland moduli spaces and Bayer--Macr\`i divisors}\label{sec-moduli}

In this section, we first recall some basic facts on the moduli spaces of semistable objects in Kuznetsov components. Then we summarize the results in \cite{FGLZ} and \cite{ppzEnriques2023} on constructing Lagrangian subvarieties of moduli spaces of stable objects in Kuznetsov components of GM fourfolds or sixfolds. Finally, in Proposition \ref{prop-restrict-BM-divisor}, we establish a relation between natural divisor classes on moduli spaces of semistable objects in Kuznetsov components of GM fourfolds or sixfolds and their hyperplane sections.

Let $X$ be a GM fourfold or sixfold. For any stability condition $\sigma_X\in \Stab^{\circ}(\Ku(X))$ and a pair of integers $a,b$, we denote by $\cM^X_{\sigma_X}(a,b)$ the moduli stack parametrizing families of geometrically $\sigma_X$-semistable objects of class $a\Lambda_1+b\Lambda_2$ (cf.~\cite[Definition 21.11]{BLMNPS21}). Similarly, for a GM threefold or fivefold $Y$ and a Serre-invariant stability condition $\sigma_Y$ on $\Ku(Y)$, let $\cM^Y_{\sigma_Y}(a,b)$ be the moduli stack parametrizing families of geometrically $\sigma_Y$-semistable objects of class $a\lambda_1+b\lambda_2$.

According to \cite[Theorem 21.24(3)]{BLMNPS21}, $\cM^X_{\sigma_X}(a,b)$ and $\cM^Y_{\sigma_Y}(a,b)$ are algebraic stacks of finite type over $\CC$, and admit proper good moduli spaces\footnote{See \cite{alp} for the definition of good moduli spaces.} $M^X_{\sigma_X}(a,b)$ and $M^Y_{\sigma_Y}(a,b)$, respectively. 

\subsection{Relative stability conditions}

In our paper, we need to consider the deformation of Bridgeland moduli spaces. For this purpose, we introduce the following notions of relative stability conditions and relative moduli spaces.

Roughly speaking, given a flat projective morphism $\cX\to S$ between quasi-projective schemes over $\CC$ and a suitable full $S$-linear triangulated subcategory $\cD\subset \Db(\cX)$, a \emph{relative stability condition on $\cD$ over $S$} is a collection 
\[\underline{\sigma}=(\underline{\sigma}_s=(Z_{\underline{\sigma}_s}, \cA_{\underline{\sigma}_s}))_{s\in S}\]
of stability conditions on $\cD_s$ for each $s\in S$ satisfying relative versions of properties in Definition \ref{def-stability-condition}, such as the central charge $Z_{\underline{\sigma}_s}$ being locally constant over $S$.

We will not delve into this definition in detail; for precise definitions, we refer the reader to \cite[Definition 21.15]{BLMNPS21}. One of the main features of relative stability conditions is the existence of well-behaved moduli stacks and good moduli spaces. In our paper, we will rely on the following two results.

\begin{theorem}\label{thm-relative-stability-threefold}
Let $\cY\to S$ be a smooth projective family of GM threefolds,  where $S$ is a connected quasi-projective scheme over $\CC$. Then there exists a relative stability condition $\underline{\sigma}$ on $\Ku(\cY)$ over~$S$ such that $\underline{\sigma}_s$ is a Serre-invariant stability condition on $\Ku(\cY_s)$ for each $s\in S$. Moreover, there exists a relative moduli space 
\[\pi\colon M^{\cY}_{\underline{\sigma}}(a,b)\to S\]
for each pair of integers $a,b$ such that $\pi$ is proper and satisfies
\[\pi^{-1}(s)\cong M^{\cY_s}_{\underline{\sigma}_s}(a,b)\]
for each $s\in S$.
\end{theorem}

\begin{proof}
The existence of $\underline{\sigma}$ follows from \cite[Proposition 26.1, Theorem 23.1]{BLMNPS21}. Then the description of the relative moduli space is given in \cite[Theorem 21.24]{BLMNPS21}.
\end{proof}

Recall that each ordinary GM fourfold $X$ has a canonical quadric, which is smooth when $X$ is general (cf.~\cite[Section 2.1, Remark 2.2]{perry2019stability}). More generally, for a family of ordinary GM fourfolds, there is a family of canonical quadrics (cf.~\cite[Section 5.1]{perry2019stability}).

\begin{theorem}\label{thm-relative-stability-fourfold}
Let $X, X'$ be ordinary GM fourfolds with smooth canonical quadrics and $a,b$ be a pair of coprime integers. Let $\sigma_X\in \Stab^{\circ}(\Ku(X))$ be a generic stability condition. Then there exists a smooth connected curve $C$, a family $\cX\to C$ of ordinary GM fourfolds, points $0,1\in C$, and a relative stability condition $\underline{\sigma}$ on $\Ku(\cX)$ over $C$ such that

\begin{enumerate}
    \item $\cX_0=X$ and $\cX_1=X'$,

    \item $\underline{\sigma}_c\in \Stab^{\circ}(\Ku(\cX_c))$ is generic for any $c\in C$,

    \item $M_{\sigma_X}^X(a,b)=M^X_{\underline{\sigma}_0}(a,b)$, and

    \item there exists a relative moduli space 
\[\pi\colon M^{\cX}_{\underline{\sigma}}(a,b)\to C\]
smooth and projective over $C$.
\end{enumerate}
\end{theorem}

\begin{proof}
This follows from \cite[Proposition 5.6]{perry2019stability}.
%The existence of $\underline{\sigma}$ follows from \cite[Proposition 5.3, Remark 5.4]{perry2019stability} and a similar argument as in the proof of \cite[Proposition 30.5]{BLMNPS21}. Subsequently, the existence of the relative moduli space and the properness of $\pi$ follow from \cite[Theorem 21.24]{BLMNPS21}. According to \cite[Theorem 31.1]{BLMNPS21} or \cite[Theorem 1.4]{perry2020integral}, when $\gcd(a,b)=1$ and $S$ is reduced, $\pi$ is smooth over the locus in $S$ such that $\underline{\sigma}_s$ is generic. Finally, the projectivity follows from \cite[Theorem 21.15]{BLMNPS21} and \cite[Corollary 3.4]{david:projective-criteria}.
%The existence of $\underline{\sigma}$ for a family of cubic fourfolds follows from \cite[Example 30.2, Proposition 30.5]{BLMNPS21}. The existence of $\underline{\sigma}$ for a family of GM fourfolds follows from \cite[Proposition 5.3, Remark 5.4]{perry2019stability} and the same argument as the first part of the proof of \cite[Proposition 30.5]{BLMNPS21}. Then the description of the relative moduli space is given in \cite[Theorem 21.24]{BLMNPS21}.
\end{proof}

\subsection{Basic properties of moduli spaces}

%\zy{moduli stack (geometrically stable), moduli space, theorem of FGLZ...}

In the following, we recapitulate additional properties of moduli spaces of semistable objects in Kuznetsov components.

\begin{theorem}[{\cite[Theorem 1.5]{perry2019stability}}]\label{thm-fourfold-moduli}
Let $X$ be a GM fourfold or sixfold. Let $a,b$ be a pair of coprime integers and $\sigma_X\in \Stab^{\circ}(\Ku(X))$ be a generic stability condition. Then:

\begin{enumerate}
    \item the moduli space $M^X_{\sigma_X}(a,b)$ is a projective hyper-K\"ahler manifold of $\mathrm{K3^{[n]}}$-type, and

    \item there is a natural Hodge isometry
    \[\lambda_{a,b}\colon (a\Lambda_1+b\Lambda_2)^{\perp}\xra{\cong} \rH^2(M^X_{\sigma_X}(a,b), \ZZ),\]
    where the orthogonal is taken inside $\wt{\rH}(\Ku(X), \ZZ)$.
\end{enumerate}
\end{theorem}

\begin{remark}\label{rmk-non-hodeg-special}
In particular, if $X$ is non-Hodge-special, then we have $\wt{\rH}_{\Hdg}(\Ku(X), \QQ)=A_1^{\oplus 2}$ and
$$\rH^{1,1}_{\ZZ}(M^X_{\sigma_X}(a,b))\cong (a\Lambda_1+b\Lambda_2)^{\perp}\cap A_1^{\oplus 2}\cong\ZZ\cdot (b\Lambda_1-a\Lambda_2).$$
\end{remark}

For moduli spaces of stable objects in $\Ku(Y)$, we have the following result.

\begin{lemma}\label{lem-structure-3fold-moduli}
Let $Y$ be a general\footnote{Here, the generality assumption means that $Y$ is in an open dense subset of the moduli space of ordinary GM threefolds or fivefolds.} GM threefold or fivefold and $\sigma_Y$ be a Serre-invariant stability condition on $\Ku(Y)$. Then $M_{\sigma_Y}^Y(a,b)$ is smooth and projective for any pair of coprime integers $a,b$.

%We further assume $Y$ is general when it is a GM threefold. Then for any Serre-invariant stability condition on $\Ku(Y)$ and pair of coprime integers $a,b$, the moduli space $M_{\sigma_Y}^Y(a,b)$ is a smooth projective variety.
\end{lemma}

\begin{proof}
This is \cite[Proposition 5.1]{FGLZ}.
\end{proof}

%For later use, we need the following lemma.

%\begin{lemma}[{\cite[Proposition 5.2]{FGLZ}}]\label{lem-smooth-moduli-exist}
%Let $X$ be a non-Hodge-special GM fourfold and $a,b$ be a pair of coprime integers. Then, there exists an open dense subset $U_{a,b}\subset |\oh_{X}(H)|$ such that each $[Y]\in U_{a,b}$ is smooth and the moduli space $M^Y_{\sigma_Y}(a,b)$ is smooth and projective for any Serre-invariant stability condition $\sigma_Y$ on $\Ku(Y)$.
%\end{lemma}

\subsection{Lagrangian subvarieties in moduli spaces}\label{subsec-lag-moduli}

Next, we recall the results in \cite{FGLZ} and \cite{ppzEnriques2023}, which enable us to construct Lagrangian subvarieties of $M_{\sigma_X}^X(a,b)$ via two different methods.

For a GM threefold or fivefold $Y$, there is an involution $T_Y$ on $\Ku(Y)$ (cf.~\cite[Lemma 4.3]{FGLZ}), which acts trivially on $\Knum(\Ku(Y))$. Therefore, by Theorem \ref{thm-unique-threefold}, $T_Y$ induces an involution on the moduli space
\begin{equation}\label{eq-involution-moduli-3fold}
    \tau_Y\colon M_{\sigma_Y}^Y(a,b)\to M_{\sigma_Y}^Y(a,b)
\end{equation}
for any integers $a,b$, where $\sigma_Y$ is a Serre-invariant stability condition on $\Ku(Y)$.

For a GM fourfold or sixfold $X$, there is also an involution $T_X$ on $\Ku(X)$ as in \cite[Lemma 4.3]{FGLZ}, which acts trivially on the sublattice $A_1^{\oplus 2}\subset \wt{\rH}(\Ku(X), \ZZ)$. The following result is proved in \cite{ppzEnriques2023}:

\begin{theorem}[{\cite[Theorem 1.6]{ppzEnriques2023}}]\label{thm-ppz}
Let $X$ be a GM fourfold or sixfold. Given a pair of coprime integers $a,b$ and a generic $\sigma_X\in \Stab^{\circ}(\Ku(X))$, the involution $T_X$ on $\Ku(X)$ induces an anti-symplectic involution
\[\tau_X\colon M_{\sigma_X}^X(a,b)\to M_{\sigma_X}^X(a,b)\]
such that the fixed locus $\mathrm{Fix}(\tau_X)$ is a Lagrangian submanifold.
\end{theorem}

Therefore, one can construct a Lagrangian submanifold $\mathrm{Fix}(\tau_X)$ for each moduli space $M_{\sigma_X}^X(a,b)$. However, as shown in \cite[Corollary 6.8]{FGLZ}, $\mathrm{Fix}(\tau_X)$ is a Lagrangian constant cycle subvariety, hence it is expected to be rigid. 

On the other hand, for a GM variety $X$ of dimension $n\geq 4$ and a smooth hyperplane section $j\colon Y\hookrightarrow X$, there are natural adjoint functors between $\Ku(Y)$ and $\Ku(X)$:
\[j^*\colon \Ku(X)\to \Ku(Y),~\text{and}~\pr_X\circ j_*\colon \Ku(Y)\to \Ku(X).\]
It is shown in \cite{FGLZ} that these functors relate semistable objects in $\Ku(Y)$ and $\Ku(X)$, enabling the construction of a family of Lagrangian subvarieties.

\begin{theorem}[{\cite[Theorem 5.4]{FGLZ}}] \label{thm-FGLZ-moduli}
Given a pair of integers $a,b$, let $X$ be a GM variety of dimension $n\geq 4$ and $j \colon Y \hookrightarrow X$ be a smooth hyperplane section. 

\begin{enumerate}
    \item Assume $X$ is general when $n=4$ and non-Hodge-special when $n=6$. Let $\sigma_Y$ be a Serre-invariant stability condition on $\Ku(Y)$ and $\sigma_X\in \Stab^{\circ}(\Ku(X))$ generic. Then the functor $\pr_X\circ j_*$ induces a finite morphism
\[r\colon M^Y_{\sigma_Y}(a,b) \to M^X_{\sigma_X}(a,b).\]

\item Assume $Y$ is non-Hodge-special when $n=5$. Let $\sigma_X$ be a Serre-invariant stability condition on $\Ku(X)$ and $\sigma_Y\in \Stab^{\circ}(\Ku(Y))$. Then the functor $j^*$ induces a finite morphism
\[r\colon M^X_{\sigma_X}(a,b) \to M^Y_{\sigma_Y}(a,b).\]
\end{enumerate}

Moreover, if $\gcd(a,b)=1$, then 
\begin{itemize}
    \item $r$ is generically unramified and its image is a Lagrangian subvariety in $M^X_{\sigma_X}(a,b)$. $r$ is unramified when $Y$ is a general hyperplane section, and

    \item the restriction $r|_M$ of $r$ to any irreducible component $M$ of $M^Y_{\sigma_Y}(a,b)$ is either a birational map onto its image or a finite morphism of degree $2$ onto its image and \'etale on the smooth locus of $M$.
 
\end{itemize}
\end{theorem}

The following lemma is an easy consequence of the results in \cite{LLPZ:higher-dim-moduli}. The proof is similar to that of \cite[Theorem 8.2]{LLPZ:higher-dim-moduli}.

\begin{lemma}\label{lem-birational-comp}
Let $X$ be a general GM fourfold and $j\colon Y\hookrightarrow X$ be a general hyperplane section. Then, for any pair of coprime integers $a,b$ and Serre-invariant stability condition $\sigma_Y$ on $\Ku(Y)$, there exists a $\sigma_Y$-stable object $E\in \Ku(Y)$ of class $a\lambda_1+b\lambda_2$ such that $T_X(\pr_X(j_*E))\neq \pr_X(j_*E)$.
\end{lemma}

\begin{proof}
We do induction on $a^2+b^2$. If $a^2+b^2\leq 2$, the statement follows from \cite{JLLZ2021gushelmukai} and \cite[Appendix B]{FGLZ}. Assume that the statement holds whenever $a^2+b^2<n$ and $n\geq 3$. Now, we will prove the statement for $v:=a\lambda_1+b\lambda_2$ with $a^2+b^2=n$.

Let $v_-$ and $v_+$ be the decomposition of $v$ as in \cite[Proposition 2.13]{LLPZ:higher-dim-moduli}. By the induction hypothesis, we can find $\sigma_Y$-stable objects $E_-,E_+\in \cA_{\sigma_Y}$ such that $T_X(\pr_X(j_*E_{\pm}))\neq \pr_X(j_*E_{\pm})$. We can assume that $\phi_{\sigma_Y}(E_-)<\phi_{\sigma_Y}(E_+)$. Then, using \cite[Lemma 2.12]{LLPZ:higher-dim-moduli}, we can find a $\sigma_Y$-stable object $E\in \cA_{\sigma_Y}$ with $[E]=v$ such that there is an exact sequence $0\to E_-\to E\to E_+\to 0$ in $\cA_{\sigma_Y}$. 

We claim that $T_X(\pr_X(j_*E))\neq \pr_X(j_*E)$. Indeed, as $Y$ is general, we have $T_Y(E)\neq E$ and $T_Y(E_{\pm})\neq E_{\pm}$. If $T_X(\pr_X(j_*E))= \pr_X(j_*E)$, then by \cite[Lemma 5.2]{FGLZ} we get $j^*\pr_X(j_*E)\cong E\oplus T_Y(E)$. On the other hand, we have the following exact sequence in $\cA_{\sigma_Y}$
\[0\to j^*\pr_X(j_*E_-)\to j^*\pr_X(j_*E)\to j^*\pr_X(j_*E_+)\to 0.\]
This implies that the exact sequence in 
\cite[Lemma 4.10(1)]{FGLZ} associated with $j^*\pr_X(j_*E_{\pm})$ splits, which contradicts $T_X(\pr_X(j_*E_{\pm}))\neq \pr_X(j_*E_{\pm})$. Thus, the result follows.
\end{proof}

\subsection{Bayer--Macr{\`\i} divisors}

In this subsection, we review the construction of Bayer--Macr\`i divisors on moduli spaces introduced in \cite{bayer:projectivity} and prove Proposition \ref{prop-restrict-BM-divisor}, which is crucial for the subsequent sections.

%For a scheme $X$, we denote by $\Dqc(X)$ the unbounded derived category of $\oh_X$-modules with quasi-coherent cohomology. The full triangulated subcategory $\Db(X)\subset \Dqc(X)$ consists of pseudo-coherent complexes with bounded cohomology. If $X$ is noetherian, $\Db(X)$ coincides with the bounded derived category of coherent sheaves on $X$.

Let $Y$ be a GM threefold and $\sigma_Y$ be a Serre-invariant stability condition on $\Ku(Y)$. We only state the construction for the moduli space $M^Y_{\sigma_Y}(a,b)$ for simplicity. However, the construction for other moduli spaces is exactly the same.

Let $\cE$ be a universal family on $\cM_{\sigma_Y}^Y(a,b)\times X$. For a scheme $T$ and a morphism $T\to \cM_{\sigma_Y}^Y(a,b)$, we denote by $\cE_T$ the associated family of geometrically $\sigma_Y$-stable objects in $\Ku(Y)$. Then we define a real numerical Cartier class $\wt{\ell}_{\sigma_Y}\in \NS(\cM^Y_{\sigma_Y}(a,b))_{\mathbb{R}}$ by
\[\wt{\ell}_{\sigma_Y}.C:=\Im\big(\frac{Z_{\sigma_Y}(p_{Y*}\cE_C)}{-Z_{\sigma_Y}(a\lambda_1+b\lambda_2)}\big)\]
for any integral curve $C$ and morphism $C\to \cM_{\sigma_Y}^Y(a,b)$, where $p_{Y*}\colon \Ku(Y)_C\to \Ku(Y)$ is the push-forward functor and $\Ku(Y)_C\subset \Db(Y_C)$ is the base change of $\Ku(Y)$ (cf.~\cite[Theorem 3.17]{BLMNPS21}). 

According to \cite[Theorem 21.25]{BLMNPS21}, such $\wt{\ell}_{\sigma_Y}\in \NS(\cM^Y_{\sigma_Y}(a,b))_{\mathbb{R}}$ descends to a real numerical Cartier class $\ell_{\sigma_Y}\in \NS(M^Y_{\sigma_Y}(a,b))_{\mathbb{R}}$, which we call the \emph{Bayer--Macr\`i divisor} on $M^Y_{\sigma_Y}(a,b)$. When $\sigma_Y=(\cA_{\sigma_Y}, Z_{\sigma_Y})$ is \emph{rational}, i.e.~the image of $Z_{\sigma_Y}$ is contained in $\QQ[\mathfrak{i}]$, $\ell_{\sigma_Y}$ is a rational numerical Cartier class.

\begin{proposition}\label{prop-restrict-BM-divisor}
Let $X$ be a GM fourfold or sixfold and $j\colon Y\hookrightarrow X$ be a smooth hyperplane section. Let $\sigma_X\in \Stab^{\circ}(\Ku(X))$ and $\sigma_Y$ be a Serre-invariant stability condition on $\Ku(Y)$ such that $Z_{\sigma_X}(\Lambda_i)=Z_{\sigma_Y}(\lambda_i)$ for each $i=1,2$. Assume $a,b$ is a pair of integers such that $\pr_X\circ j_*$ induces a morphism
\[r\colon M^Y_{\sigma_Y}(a,b)\to M^X_{\sigma_X}(a,b).\]
Then we have
\begin{equation}\label{eq-bm}
    r^*\ell_{\sigma_X}=\ell_{\sigma_Y}\in \NS(M^Y_{\sigma_Y}(a,b))_{\mathbb{R}}.
\end{equation}
In particular, $\ell_{\sigma_Y}$ is an ample class if such a morphism $r$ exists.
\end{proposition}

\begin{proof}
We denote by $\wt{r}\colon \cM^Y_{\sigma_Y}(a,b)\to \cM^X_{\sigma_X}(a,b)$ the corresponding morphism between moduli stacks. Recall that $\wt{r}$ is induced as follows. For any morphism $T\to \cM^Y_{\sigma_Y}(a,b)$ corresponding to the family $\cE_T\in \cM^Y_{\sigma_Y}(T)$ of geometrically $\sigma_Y$-stable objects on $Y_T$, we can construct an object $\pr_{X_T}\circ (\mathrm{id}_T\times j)_* (\cE_T)$, where $\pr_{X_T}:=\pr_X\otimes \mathrm{id}_{\Dqc(T)}$ is the base change of $\pr_X$\footnote{As the composition $\Db(X)\xra{\pr_X} \Ku(X)\hookrightarrow \Db(X)$ is a Fourier--Mukai functor, we can pull-back its kernel to $\Dqc(X_T\times_T X_T)$, which defines $\pr_{X_T}\colon \Dqc(X_T)\to \Dqc(X_T)$.}. By assumption, this is a family of geometrically $\sigma_X$-semistable objects on $X_T$, thereby yielding an object in~$\cM^X_{\sigma_X}(T)$. From the construction, we see
\begin{equation}\label{eq-restrict-family-1}
    \cF|_{C\times X}=\pr_{X_C}\circ (\mathrm{id}_C\times j)_* (\cE_C)
\end{equation}
for any integral curve $C$ and morphism $C\to \cM_{\sigma_Y}^Y(a,b)$, where $\cE$ and $\cF$ are universal families on $\cM^Y_{\sigma_Y}(a,b)\times Y$ and $\cM^X_{\sigma_X}(a,b)\times X$, respectively.

By definition, we have
\[\wt{r}^*\wt{\ell}_{\sigma_X}.C=\Im\big(\frac{Z_{\sigma_X}(p_{X*}\cF|_{C\times X})}{-Z_{\sigma_X}(a\Lambda_1+b\Lambda_2)}\big),\]
where $p_{X*}\colon \Ku(X)_C\to \Ku(X)$ is the push-forward functor. Given that $Z_{\sigma_X}(\Lambda_i)=Z_{\sigma_Y}(\lambda_i)$ and $j^*\Lambda_i=2\lambda_i$ for each $i=1,2$, we see $$2Z_{\sigma_X}(p_{X*}\cF|_{C\times X})=Z_{\sigma_Y}(j^*p_{X*}\cF|_{C\times X})$$ by Remark \ref{rmk-stab-A1}. Then, since $p_{Y*}\cF|_{C\times Y}=j^*p_{X*}\cF|_{C\times X}$ by the theorem of base change, we obtain 
\[2\Im\big(\frac{Z_{\sigma_X}(p_{X*}\cF|_{C\times X})}{-Z_{\sigma_X}(a\Lambda_1+b\Lambda_2)}\big)=\Im\big(\frac{Z_{\sigma_Y}(p_{Y*}\cF|_{C\times Y})}{-Z_{\sigma_Y}(a\lambda_1+b\lambda_2)}\big).\] Therefore, we get 
\begin{equation}\label{eq-bm-divisor-1}
\wt{r}^*\wt{\ell}_{\sigma_X}.C=\frac{1}{2}\Im\big(\frac{Z_{\sigma_Y}(p_{Y*}\cF|_{C\times Y})}{-Z_{\sigma_Y}(a\lambda_1+b\lambda_2)}\big).
\end{equation}

As in Section~\ref{subsec-lag-moduli}, there is an involution $\tau_Y\colon M^Y_{\sigma_Y}(a,b)\to M^Y_{\sigma_Y}(a,b)$ induced by the involution~$T_Y$ on $\Ku(Y)$, corresponding to an involution of moduli stacks $\wt{\tau}_Y\colon \cM^Y_{\sigma_Y}(a,b)\to \cM^Y_{\sigma_Y}(a,b)$ (see \eqref{eq-involution-moduli-3fold}). Therefore,
\begin{equation}\label{eq-bm-divisor-2}
\wt{\tau}_Y^*\wt{\ell}_{\sigma_Y}.C=\Im\big(\frac{Z_{\sigma_Y}(p_{Y*}T_{Y_C}(\cE_C))}{-Z_{\sigma_Y}(a\lambda_1+b\lambda_2)}\big),
\end{equation}
for any morphism $C\to \cM^Y_{\sigma_Y}(a,b)$ from an integral curve $C$, where $\cE_C$ is the associated family and $T_{Y_C}:=T_Y\otimes \mathrm{id}_{\Db(C)}$ is the base change of $T_Y$\footnote{The functor $T_{Y_C}:=T_Y\otimes \mathrm{id}_{\Db(C)}\colon \Db(Y_C)\to \Db(Y_C)$ is defined by the pull-back of the kernel of the composition $\Db(Y)\xra{\pr_Y} \Ku(Y)\hookrightarrow \Db(Y)$ to $\Db(Y_C\times_C Y_C)$.}. Note that $\cE_C\in \Db(X_C)$ by \cite[Lemma 8.3(1)]{BLMNPS21}. From \cite[Lemma 2.41]{ku:hyperplane-section}, we obtain an isomorphism $p_{Y*}(T_{Y_C}(\cE_C))\cong T_Y(p_{Y*}(\cE_C))$. This implies 
\begin{equation}\label{eq-fix-bm-divisor}
\wt{\tau}_Y^*\wt{\ell}_{\sigma_Y}.C=\wt{\ell}_{\sigma_Y}.C
\end{equation}
as $T_Y$ acts trivially on $\Knum(\Ku(Y))$ by \cite[Lemma 4.10]{FGLZ}.

Consequently, to prove the proposition, it is enough to show
\[2\wt{r}^*\wt{\ell}_{\sigma_X}. C=\wt{\ell}_{\sigma_Y}.C+\wt{\tau}_Y^*\wt{\ell}_{\sigma_Y}.C.\]
We claim that there exists an exact triangle
\begin{equation}\label{eq-bm-divisor-triangle}
    T_{Y_C}(\cE_C)\to \cF|_{C\times Y}\to \cE_C,
\end{equation}
and the result follows from \eqref{eq-bm-divisor-1} and \eqref{eq-bm-divisor-2}. Indeed, from \eqref{eq-restrict-family-1}, we see
\[\cF|_{C\times Y}=(\mathrm{id}_C\times j)^*\circ \pr_{X_C}\circ (\mathrm{id}_C\times j)_* (\cE_C)=\big(\mathrm{id}_{\Db(C)}\otimes (j^*\circ \pr_X\circ j_*)\big)(\cE_C),\] then \eqref{eq-bm-divisor-triangle} follows from the base change of \cite[Equation (35)]{FGLZ}. This completes the proof of \eqref{eq-bm}. Finally, the ampleness of $\ell_{\sigma_Y}$ follows from the facts that $\ell_{\sigma_X}$ is ample (cf.~\cite[Theorem 3.1]{sacca:moduli-ku}) and the morphism $r$ is finite according to Theorem~\ref{thm-FGLZ-moduli}.
\end{proof}

\begin{remark}\label{rmk-bm-divisor-fix}
In the proof of Proposition~\ref{prop-restrict-BM-divisor}, we  also obtain $\tau_Y^*\ell_{\sigma_Y}=\ell_{\sigma_Y}\in \NS(M^Y_{\sigma_Y}(a,b))_{\mathbb{R}}$ in \eqref{eq-fix-bm-divisor} for any GM threefold or fivefold $Y$, Serre-invariant stability condition $\sigma_Y$ on $\Ku(Y)$, and integers $a,b$.
\end{remark}

\begin{remark}\label{rmk-bm-divisor-4fold}
Under the isometry 
\[\lambda_{a,b}\colon (a\Lambda_1+b\Lambda_2)^{\perp}\xra{\cong} \rH^2(M^X_{\sigma_X}(a,b), \mathbb{R}),\]
the divisor class $\ell_{\sigma_X}$ is proportional to $b\Lambda_1-a\Lambda_2$ for any $\sigma_X\in \Stab^{\circ}(\Ku(X))$. This can be deduced either by employing the calculations as in \cite[Lemma 9.2]{bayer:projectivity} or by using a deformation argument, given that  $(a\Lambda_1+b\Lambda_2)^{\perp}\cap (A_1^{\oplus 2})$ 
is generated by $b\Lambda_1-a\Lambda_2$ when $X$ is non-Hodge-special (cf.~Remark \ref{rmk-non-hodeg-special}).
\end{remark}

Similarly, using \cite[Equation (40)]{FGLZ} instead of \cite[Equation (35)]{FGLZ}, we have the following result for GM fivefolds.

\begin{proposition}\label{prop-restrict-BM-divisor-5fold}
Let $X$ be a GM fivefold and $j\colon Y\hookrightarrow X$ be a smooth hyperplane section. Let $\sigma_Y\in \Stab^{\circ}(\Ku(Y))$ and $\sigma_X$ be a Serre-invariant stability condition on $\Ku(X)$ such that $Z_{\sigma_Y}(\Lambda_i)=Z_{\sigma_X}(\lambda_i)$ for each $i=1,2$. Assume $a,b$ is a pair of integers such that $j^*$ induces a morphism
\[r\colon M^X_{\sigma_X}(a,b)\to M^Y_{\sigma_Y}(a,b).\]
Then we have
\begin{equation}\label{eq-bm-5fold}
    r^*\ell_{\sigma_Y}=\ell_{\sigma_X}\in \NS(M^X_{\sigma_X}(a,b))_{\mathbb{R}}.
\end{equation}
In particular, $\ell_{\sigma_X}$ is an ample class if such a morphism  $r$ exists.
\end{proposition}

\section{Atomic and 1-obstructed objects}\label{sec-atomic-object}

In this section, we provide a brief overview of atomic and 1-obstructed objects, along with the results relevant to our study. We recommend \cite{beckmann:atomic-object} and \cite{markman:rank-1-obstruction} for more details. Subsequently, we offer a criterion (Theorem~\ref{atomic-criterion-thm}) to verify the atomicity for sheaves supported on Lagrangians, which generalizes \cite[Theorem 1.8]{beckmann:atomic-object} to immersed Lagrangian subvarieties.

\subsection{Recollections}
We fix a projective hyper-K\"ahler manifold $X$ of dimension $2n>2$. 
\subsubsection{Extended Mukai lattice}

The second cohomology $\rH^2(X,\ZZ)$ is equipped with an integral primitive quadratic form $q_X$, called the \emph{Beauville--Bogomolov--Fujiki (BBF) form}, whose signature is $(3,b_2(X)-3)$. Motivated by the theory of K3 surfaces, one can associate $X$ with the \emph{extended Mukai lattice} (or \emph{rational LLV lattice})
$$\tilde{\rH}(X,\QQ):=\QQ\alpha\oplus \rH^2(X,\QQ)\oplus\QQ\beta,$$
which is first introduced by \cite{Taelman23}. This is a rational vector space with a quadratic form~$\tilde{q}$, whose restriction on $\rH^2(X,\QQ)$ is $q_X$, and $\alpha,\beta$ are isotropic vectors orthogonal to $\rH^2(X,\QQ)$ satisfying $\tilde{q}(\alpha,\beta)=-1$. The elements $\alpha$ and $\beta$ are of degree $-2$ and $2$, respectively, and vectors in $\rH^2(X,\QQ)$ are of degree 0.

% Moreover, $(\tilde{H}(X,\QQ), \tilde{q})$ has a Hodge structure by defining $\tilde{H}^{2,0}(X,\CC):=H^{2,0}(X)$.

We denote by $h$ the \emph{grading operator} defined by 
$$h|_{\rH^k(X,\QQ)}\colon \rH^k(X,\QQ)\rightarrow \rH^k(X,\QQ), \quad x\mapsto (k-2n)x.$$
For an element $\omega\in \rH^2(X,\QQ)$, we define an operator $e_{\omega}\in\mathrm{End(H^*(X,\QQ))}$ by cupping with $\omega$. We say $\omega$ satisfies the \emph{Hard Lefschetz property} if there is an associated $\mathfrak{sl}_2$-triple $(e_\omega, h, \Lambda_\omega)$ acting on $\rH^*(X,\QQ)$, i.e.~there is a Lie homomorphism
$$\rho_\omega\colon \mathfrak{sl}_2\rightarrow\mathrm{End}(\rH^*(X,\QQ)).$$

The \emph{Looijenga--Lunts--Verbitsky (LLV) algebra} $\mathfrak{g}(X)$ is introduced in \cite{LL97} and \cite{Ver96}, which is the subalgebra of $\mathrm{End}(\rH^*(X,\QQ))$ generated by $(e_\omega, h, \Lambda_\omega)$ for all $\omega\in \rH^2(X,\QQ)$ with the Hard Lefschetz property. On the other hand, there is an action of $\mathfrak{g}(X)$ on $\tilde{\rH}(X,\QQ)$ defined by
$$e_{\omega}(\alpha)=\omega,\quad e_{\omega}(\lambda)=q_X(\omega, \lambda)\beta,\quad e_{\omega}(\beta)=0,\quad\forall\lambda\in \rH^2(X,\QQ).$$

\subsubsection{Atomic objects}
Now we define the notion of atomic objects introduced in \cite{beckmann:atomic-object}, which were independently studied by \cite{markman:rank-1-obstruction}, where the author referred to them as cohomologically 1-obstructed objects.

\begin{definition}\label{def-atomic}
An object $E\in\Db(X)$ is called \emph{atomic} if there is a non-zero  $\tilde{v}(E)\in\tilde{\rH}(X,\QQ)$ such that $\mathrm{Ann}(v(E))=\mathrm{Ann}(\tilde{v}(E))\subset\mathfrak{g}(X)$, where $\mathrm{Ann}(-)$ denotes the annihilator Lie subalgebra.
\end{definition}

Such a vector $\tilde{v}(E)\in\tilde{\rH}(X,\QQ)$ is called an \emph{extended Mukai vector of $E$} (cf.~\cite[Definition 4.16]{Beckmann-derived} and \cite[Proposition 3.3]{beckmann:atomic-object}). In \cite[Section 6]{markman:rank-1-obstruction}, the $1$-dimensional subspace of $\tilde{\rH}(X,\QQ)$ spanned by $\tilde{v}(E)$ is called the \emph{LLV line of $E$}. Note that by definition, the atomicity of an object~$E$ is detected by the linear span of $v(E)$.

By definition, the property of $E$ being atomic is a condition on its Lie subalgebra $\mathrm{Ann}(v(E))$. Indeed, this is equivalent to saying that $\mathrm{Ann}(v(E))$ has the smallest possible positive codimension. 

\begin{proposition}[{\cite[Proposition 3.1]{beckmann:atomic-object}}]\label{ann-max-dim}
An object $E\in\Db(X)$ is atomic if and only if $\mathrm{Ann}(v(E))$ is of dimension $b_2(X)+1$.
\end{proposition}

As the action of the LLV algebra on the extended Mukai lattice behaves well under derived equivalences (cf.~\cite{Taelman23}) and deformations, it is demonstrated in \cite{beckmann:atomic-object} that being atomic is invariant under these two operations.

\begin{proposition}[{\cite[Proposition 3.10]{beckmann:atomic-object}}]\label{prop-atomic-lagrangian-deform}
Let $E\in\Db(X)$, then the property of being atomic for $E$ is invariant under derived equivalences. Moreover, given a smooth proper family $\mathcal{X}\rightarrow B$ of hyper-K\"ahler manifolds such that $B$ is connected and $\mathcal{E}\in \Db(\cX)$, then for two points $b, b'\in B$, $\mathcal{E}_b$ is atomic if and only if $\mathcal{E}_{b'}$ is. 
\end{proposition}

In particular, we have the following.

\begin{proposition}\label{prop-immersed-deform}
Given a smooth proper family $\mathcal{X}\rightarrow B$ of projective hyper-K\"ahler manifolds such that $B$ is connected. Let $\cY$ be an equidimensional scheme and $\cY\to B$ be a morphism such that fibers are equidimensional of the same dimension. Then for any proper morphism $f\colon \cY \to \cX$ and $\cF\in \Db(\cY)$, if there exists a point $0\in B$ such that $f_{0_*}\cF_0\in \Db(\cX_0)$ is an atomic object, then $f_{b_*}\cF_{b}\in \Db(\cX_{b})$ is atomic for any $b\in B$.
%and $j\colon \cY \to \cX$ is a morphism with $\cY$ $\cY\to B$ flat and $j_b\colon \cY_b\to \cX_b$ is an immersion Lagrangian submanifold
\end{proposition}

\begin{proof}
We may assume that $B$ is equidimensional, otherwise, we can apply the result to each irreducible component of $B$. For any $b\in B$, we consider the Cartesian diagram
% https://q.uiver.app/#q=WzAsNCxbMCwxLCJcXGNZIl0sWzEsMSwiXFxjWCJdLFsxLDAsIlxcY1hfYiJdLFswLDAsIlxcY1lfYiJdLFswLDEsImYiXSxbMywwLCIiLDAseyJzdHlsZSI6eyJ0YWlsIjp7Im5hbWUiOiJob29rIiwic2lkZSI6InRvcCJ9fX1dLFsyLDEsIiIsMix7InN0eWxlIjp7InRhaWwiOnsibmFtZSI6Imhvb2siLCJzaWRlIjoidG9wIn19fV0sWzMsMiwiZl9iIl1d
\[\begin{tikzcd}
	{\cY_b} & {\cX_b} \\
	\cY & \cX
	\arrow["{f_b}", from=1-1, to=1-2]
	\arrow[hook, from=1-1, to=2-1]
	\arrow[hook, from=1-2, to=2-2]
	\arrow["f", from=2-1, to=2-2]
\end{tikzcd}\]
By our assumption, $\cX_b$ is smooth, hence $\cX_b\hookrightarrow \cX$ is a locally complete intersection. Moreover, $\mathrm{codim}_{\cY} (\cY_b)=\mathrm{codim}_{\cX}(\cX_b)=\dim B$. Then using \cite[Corollary 2.27]{ku:hyperplane-section}, we obtain
\begin{equation}\label{eq-prop-sec4}
    (f_*\cF)|_{\cX_b}\cong f_{b_*}\cF_{b}.
\end{equation}
As $f$ is proper, we see $f_*\cF\in \Db(\cX)$. Then applying Proposition \ref{prop-atomic-lagrangian-deform} to $f_*\cF$ and using \eqref{eq-prop-sec4}, we see $f_{b_*}\cF_{b}$ is atomic for every $b\in B$.
\end{proof}

\subsubsection{1-obstructed objects}\label{subsubsec-1-obs}
For $E\in\Db(X)$, one of the significant results in \cite{beckmann:atomic-object} and \cite{markman:rank-1-obstruction} is to establish a relation between $E$ being atomic and obstructions of the Mukai vector $v(E)$ to stay Hodge-type. We first review two different obstruction maps. 

Recall that the group of degree 2 polyvector fields of $X$ is $$\HT^2(X):=\rH^2(X,\oh_X)\oplus \rH^1(X,\mathcal{T}_X)\oplus \rH^0(X,\wedge^2\mathcal{T}_X).$$
Then, for $x\in \rH^*(X, \QQ)$, we have a natural morphism 
$$\chi^{\mathrm{coh}}_{x}\colon \HT^2(X)\rightarrow \rH^*(X,\CC),\quad\mu\mapsto\mu\lrcorner x,$$
defined by contraction on vector fields. We call $\chi^{\mathrm{coh}}_{v(E)}$ the \emph{cohomological obstruction map} of an object $E\in \Db(X)$. The following theorem enables us to study atomic sheaves through their cohomological obstruction maps.

\begin{theorem}[{\cite[Theorem 1.2, Remark 4.1]{beckmann:atomic-object}}]\label{Beckmann-equiv-atomic}
For an element $x\in \bigoplus_{p\geq 0} \rH^{p,p}_{\QQ}(X)$, the map $\chi^{\mathrm{coh}}_{x}$ is of rank one if and only if $\mathrm{Ann}(x)=\mathrm{Ann}(\tilde{x})$ for a non-zero element $\tilde{x}\in \tilde{\rH}(X, \QQ)$.

In particular, an object $E\in\Db(X)$ is atomic if and only if $\chi^{\mathrm{coh}}_{v(E)}$ is of rank one.
\end{theorem}
According to \cite{Toda09}, $\chi^{\mathrm{coh}}_{v(E)}$ measures the obstruction to lifting the Mukai vector $v(E)$ along the first-order deformation corresponding to the elements of $\HT^2(X)$. Furthermore, the authors of \cite{beckmann:atomic-object} and \cite{markman:rank-1-obstruction} noted another class of geometric objects with stronger deformation properties. 

For a projective hyper-K\"ahler manifold $X$, we consider its second Hochschild cohomology group $\HH^2(X)$ as the space of natural transformations of $\Db(X)$ from the identity functor $\mathrm{id}$ to its shift $\mathrm{id}[2]$. Then, for any $E\in\Db(X)$, we have the natural morphism 
$$\chi_E\colon \HH^2(X)\rightarrow\Ext_X^2(E,E),\quad \mu\mapsto\mu_E$$ defined by evaluating $E$ at the natural transformation $\mu$, called the \emph{obstruction map} of $E$. We say $E\in\Db(X)$ is \emph{1-obstructed} if $\chi_E$ is of rank one, which is the main theme of \cite{markman:rank-1-obstruction}. In addition, the property of being $1$-obstructed also remains invariant under derived equivalences. 

%\begin{proposition}[{\cite{markman:rank-1-obstruction}}]\label{derived-1-obstructed}
%Let $Y$ be another projective hyper-K\"ahler manifold. If there is a derived equivalence $\Phi: \Db(X)\rightarrow\Db(Y)$, then for $F\in\Db(X)$, $F$ is one-obstructed if and only if $\Phi(F)$ is one-obstructed.
%\end{proposition}

By the Hochschild--Kostant--Rosenberg (HKR) isomorphism $I^{HKR}\colon \HH^2(X)\cong\HT^2(X)$, the obstruction $\chi_E$ parametrizes the obstruction to lifting $E$ along the first-order deformation given by  
the elements of $\HT^2(X)$. Moreover, we have the following commutative diagram according to \cite{Huang21}
% https://tikzcd.yichuanshen.de/#N4Igdg9gJgpgziAXAbVABwnAlgFyxMJZABgBpiBdUkANwEMAbAVxiRAEEQBfU9TXfIRQBGclVqMWbAELdeIDNjwEiZYePrNWiEAGFu4mFADm8IqABmAJwgBbJGRA4ISUSAZ0ARjAYAFfspC7jAWOCDUmlI6ADrRAMYAFlgA+gCicpY29oiOzkgATNQe3n4BgmxWWMYJYRGS2iAAkgB6wAASANIASlwZINZ2BdR5iG7FPv5K5TqV1bUSWmyxDFZxEFZgMFYGXEA
$$\begin{tikzcd}
\HH^2(X) \arrow[r, "\chi_E"] \arrow[d, "I^{HKR}"'] & \Ext_X^2(E,E) \\
\HT^2(X) \arrow[ru, "\lrcorner\exp{\mathrm{At}_{E}}"']                  &  
\end{tikzcd}$$
where $\mathrm{At}_{E}\in\Ext^1_X(E,E\otimes\Omega_X^1)$ is the Atiyah class of $E$.

As in \cite{beckmann:atomic-object} and \cite[Lemma 6.10]{markman:rank-1-obstruction}, for a sheaf $E\in\Coh(X)$, we have
$$I^{HKR}(\ker(\chi_E))\subset\ker(\chi^{\mathrm{coh}}_{v(E)})$$ 
and the equality holds if $E$ is 1-obstructed. In particular, combined with Theorem~\ref{Beckmann-equiv-atomic}, this implies that 
\begin{theorem}[{\cite[Theorem 1.3]{beckmann:atomic-object}}]\label{1-obstructed-atomic}
Let $X$ be a projective hyper-K\"ahler manifold. For a sheaf $E\in\Coh(X)$, if $E$ is 1-obstructed, then $E$ is atomic.
\end{theorem}

\begin{remark}
In fact, as in \cite{beckmann:atomic-object} and \cite{markman:rank-1-obstruction}, all results mentioned above also apply to objects in $\Db(X)$ under additional assumptions. Since the primary objects for study in this paper are atomic sheaves, we omit the related discussion for brevity.
\end{remark}

\subsection{Atomic sheaves supported on Lagrangian subvarieties}

By \cite[Theorem 6.24]{markman:rank-1-obstruction} and \cite[Proposition 3.11]{beckmann:atomic-object}, each torsion atomic sheaf is supported on points or Lagrangian subvarieties. We make the following definition. 

Recall that $i\colon L\to X$ is an immersed Lagrangian submanifold if $L$ is a smooth projective variety, $i$ is finite and unramified, and the image of $i$ is a Lagrangian subvariety of $X$.

\begin{definition}\label{def-atomic-lag}
Let $X$ be a projective hyper-K\"ahler manifold and $i\colon L\to X$ be an immersed Lagrangian submanifold. We say $L$ is an \emph{immersed atomic (resp.~$1$-obstructed) Lagrangian} if~$i_*\oh_L$ is an atomic (resp.~$1$-obstructed) sheaf on $X$.
\end{definition}

Note that compared to \cite[Definition 7.1]{beckmann:atomic-object}, we do not assume $L$ to be embedded into $X$ and connected. Moreover, by Proposition \ref{prop-immersed-deform}, the property of being atomic is invariant under deformations of immersed Lagrangian submanifolds. Furthermore, we have the following.

\begin{lemma}\label{lem-union-component-1obs}
Let $X$ be a projective hyper-K\"ahler manifold and $i\colon L\to X$ be an immersed Lagrangian submanifold. If $L\to X$ is atomic, then the restriction to each union of its connected components is also atomic. The same result holds for $1$-obstructedness provided $i$ is an embedding.
\end{lemma}

\begin{proof}
Let $M_1\subset L$ be a union of connected components. First, we assume that $L\to X$ is atomic. By Theorem \ref{atomic-criterion-thm}, we know that $i^*\colon \rH^2(X, \QQ)\to \rH^2(L, \QQ)$ is of rank one and $c_1(L)\in \im(i^*)$. As $i^*|_{M_1}\colon \rH^2(X, \QQ)\to \rH^2(M_1, \QQ)$ is non-trivial, we see it is of rank one as well. And from $\omega_{L}|_{M_1}=\omega_{M_1}$, we get $c_1(L)|_{M_1}=c_1(M_1)\in \im(i^*|_{M_1})$. Then the atomicity of $i|_{M_1}\colon M_1\to X$ follows from Theorem \ref{atomic-criterion-thm}.

Now assume that $L\to X$ is an embedding and $1$-obstructed. Let $M_2$ be the union of other components so that $M_1\cup M_2=L$ and $M_1\cap M_2=\varnothing$. For each $k$, we have $$\Ext^k_X(i_*\oh_{M_1}, i_*\oh_{M_2})=\Ext_X^k(i_*\oh_{M_2}, i_*\oh_{M_1})=0.$$
In particular, there is a decomposition
\[\Ext_X^2(i_*\oh_L, i_*\oh_L)=\Ext^2_X(i_*\oh_{M_1}, i_*\oh_{M_1})\oplus \Ext^2_X(i_*\oh_{M_2}, i_*\oh_{M_2})\]
satisfying $p_t\circ \chi_{i_*\oh_L}=\chi_{i_*\oh_{M_t}}$ for each $t=1,2$, where
\[p_t\colon \Ext_X^2(i_*\oh_L, i_*\oh_L)\to \Ext^2_X(i_*\oh_{M_t}, i_*\oh_{M_t})\]
is the projection. Since $\chi_{i_*\oh_L}$ is of rank one, it follows that $\chi_{i_*\oh_{M_t}}$ has rank at most one for each $t=1,2$. Then the result follows from the non-triviality of $\chi_{i_*\oh_{M_t}}$ as in the arguments in \cite[Lemma 3.9, Remark 3.10]{markman:rank-1-obstruction} and \cite[Section 7.2]{beckmann:atomic-object}.
\end{proof}

\subsection{A criterion}

Let $X$ be a projective hyper-K\"ahler manifold of dimension $2n$ and $L$ be a smooth projective variety. From now on, we assume that there is an immersed Lagrangian submanifold $i\colon L\to X$.

In the rest of this section, we aim to investigate when $i_{*}\mathcal{O}_{L}$ is an atomic sheaf. Recall that by \cite[Theorem 1.8]{beckmann:atomic-object}, the author demonstrates a criterion for the structure sheaf of a connected embedded Lagrangian submanifold to be atomic. Utilizing the same argument, we establish the following criterion in our scenario, which allows the atomic sheaf to be supported on a (possibly) non-connected singular Lagrangian. 

%In our definition, a variety is simply a pure-dimensional reduced separated scheme of finite type over $\CC$.

\begin{theorem}\label{atomic-criterion-thm}
Let $X$ be a projective hyper-K\"ahler manifold. Then an immersed Lagrangian submanifold $i\colon L\to X$ is atomic if and only if the map $i^*\colon \rH^2(X, \QQ)\to \rH^2(L, \QQ)$
has rank one and $c_1(L)\in \im(i^*)\subset \rH^2(L, \QQ)$.
\end{theorem}
%\hf{the converse part TBC...}

\begin{proof}
The proof employs the same strategy as the argument in \cite[Section 7]{beckmann:atomic-object}. For completeness, we sketch the main ideas here.

First, we assume that $i^*\colon \rH^2(X, \QQ)\to \rH^2(L, \QQ)$
has rank one and $c_1(L)\in \im(i^*)\subset \rH^2(L, \QQ)$. Since $c_1(L)\in \im(i^*)$, there exists $\lambda\in \rH^{1,1}_{\QQ}(X)$ such that $i^{*}\lambda=-c_1(L)/2$. Then the following Lemma~\ref{mukai-vector-class} implies that $$v(i_{*}\mathcal{O}_{L})\mathrm{exp}(\lambda)=i_*[L].$$ At the same time, such $\lambda$ yields an operator in $\mathfrak{g}(X)$ so that  $\mathrm{Ann}(v(i_{*}\mathcal{O}_{L})\mathrm{exp}(\lambda))$ and $\mathrm{Ann}(v(i_{*}\mathcal{O}_{L}))$ are adjoint to each other and have the same dimension. Then by Proposition~\ref{ann-max-dim} and Theorem~\ref{Beckmann-equiv-atomic}, $i_{*}\mathcal{O}_{L}$ is atomic if and only if the cohomological obstruction map $\chi^{\mathrm{coh}}_{i_*[L]}$ is of rank one, where $$\chi^{\mathrm{coh}}_{i_*[L]}\colon \HT^2(X)\rightarrow \rH^*(X,\Omega_X), \quad \mu\mapsto\mu\lrcorner i_*[L].$$ 
As $\dim \HT^2(X)=\dim \rH^2(X, \QQ)$, it suffices to show $\mathrm{ker}(\chi_{i_*[L]})\cong\mathrm{ker}(i^*)$.

To this end, we only need to generalize \cite[Proposition 7.2]{beckmann:atomic-object} to our setting. Indeed. the argument there applies analogously without changes. We denote $\mathrm{ker}(\chi^{\mathrm{coh}}_{i_*[L]})$ by $K$. Because the image of $i$ is Lagrangian, we see  $\rH^2(X,\oh_X)\subset K$. Moreover, by the existence of the $\mathfrak{sl}_2$-triple $(e_\omega, h_\omega, \Lambda_\omega)\subset\mathfrak{g}(X)_{\CC}$ (cf.~\cite[Section 2]{Taelman23}) where $\omega\in \rH^0(X,\Omega_X^2)$ is the holomorphic symplectic 2-form, the same calculation as in \cite[Proposition 7.2]{beckmann:atomic-object} gives $\rH^0(X,\wedge^2\mathcal{T}_X)\subset K$. 

Then, according to \cite[Lemma 7.3]{beckmann:atomic-object}, $K\cap \rH^1(X,\mathcal{T}_X)$ can be identified with the kernel of the following 
$$\rH^1(X,\Omega_X)\rightarrow \rH^*(X,\Omega_X^*),\quad \mu\mapsto\mu\wedge i_*[L].$$ Hence, applying Lemma \ref{lem-voisin}, we obtain $$K\cap \rH^1(X,\mathcal{T}_X)\cong\ker(i^*\colon \rH^{1,1}(X,\QQ)\rightarrow \rH^{1,1}(L,\QQ))$$
and $\mathrm{ker}(\chi_{i_*[L]})\cong\mathrm{ker}(i^*)$ follows. This finishes the proof of the atomicity of $i_*\oh_L$.

Conversely, if $i_*\oh_L$ is atomic, we can adapt the argument in the proof of \cite[Theorem 1.8]{beckmann:atomic-object} by substituting \cite[Lemma 7.4]{beckmann:atomic-object} with Lemma \ref{mukai-vector-class}, thereby obtaining the desired result.
\end{proof}

% Before proving Theorem~\ref{atomic-criterion-thm}, we present three lemmas that are crucial in the proof of Theorem~\ref{atomic-criterion-thm}.

From the proof of Theorem~\ref{atomic-criterion-thm}, we can directly deduce the following criterion.

\begin{proposition}\label{prop-general-criterion}
Let $X$ be a projective hyper-K\"ahler manifold and $i\colon L\to X$ be an immersed Lagrangian submanifold such that the map $i^*\colon \rH^2(X, \QQ)\to \rH^2(L, \QQ)$ has rank one. Then $i_*E$ is atomic for any $E\in \Db(L)$ satisfying $v(i_*E)\exp(\lambda)\in \QQ \cdot i_*[L]$ for some $\lambda\in\rH^{1,1}_{\QQ}(X)$.
\end{proposition}

% \begin{proof}
% As in the proof of Theorem~\ref{atomic-criterion-thm}, 
% \[\dim \mathrm{Ann}(v(i_*E))=\dim \mathrm{Ann}(i_*[L])=b_2(X)+1.\]
% Thus, the result follows from Proposition \ref{ann-max-dim}.
% \end{proof}

\begin{remark}\label{rmk-prop}
In particular, if $\frac{c_1(L)}{2}=-\gamma-i^*\lambda$ for $\gamma\in \rH^2(L, \QQ)$ and $\lambda\in\rH^{1,1}_{\QQ}(X)$, then the condition $v(i_*\cL)\exp(\lambda)=i_*[L]$ holds for any line bundle $\cL$ on $L$ with $c_1(\cL)=\gamma \in \rH^2(L, \QQ)$.
\end{remark}

Now we present lemmas that we used above.

\begin{lemma}[{\cite[Lemma 1.5]{voisin1992stabilite}}]\label{lem-voisin}
Let $X$ be a projective hyper-K\"ahler manifold of dimension $2n>2$ and $i\colon L\to X$ be an immersed Lagrangian submanifold. Then $\ker(i^*)\subset \rH^2(X, \QQ)$ is equal to the kernel of the composition map
 \[i_*[L]\cup -\colon \rH^2(X, \QQ)\xra{i^*}  \rH^2(L, \QQ) \xra{i_*} \rH^{2n+2}(X, \QQ).\]
\end{lemma}

\begin{proof}
The proof of \cite[Lemma 1.5]{voisin1992stabilite} adapts to our situation without changes. It suffices to prove this for cohomology with real coefficients. By the adjunction formula, we have an inclusion $\ker(i^*)\subset \ker(i_*[L]\cup -)$. To show the reverse inclusion, as in the proof of \cite[Lemma 1.5]{voisin1992stabilite}, we choose a K\"ahler class $\lambda\in \rH^2(X, \mathbb{R})$ such that $i^*\lambda$ is also a K\"ahler class. Now it suffices to prove that the intersection form $$q_{\lambda}(a,b):=\int_L (i^*\lambda)^{n-2}.a.b$$ is non-degenerate after restricting it to $\mathrm{im}(i^*)\subset \rH^{1,1}_{\mathbb{R}}(L)$. Then the result follows from the Hodge index theorem as in \cite[Lemma 1.5]{voisin1992stabilite} since $i^*\lambda$ is a K\"ahler class.
%It suffices to prove this for cohomology with real coefficients. Since $j^*\circ j_*=\mu_{[X]}$, clearly $ \mathrm{Ker}(j^*)\subset\mathrm{Ker}(\mu_{[X]})\subset H^2(Y,\mathbb{R})$; to show the reverse inclusion, choose a Kähler class $\lambda\in H^2(Y,\mathbb{R})$; the result is obvious if $n=\mathrm{dim}(X)=1$, so assume $n>2$, and define the following intersection form $q_\lambda$ on $H^2(X,\mathbb{R}): q_\lambda(a,b)=\int_X\lambda^{n-2}.a.b$. Now if $a$ and $b$ are elements of $H^2(Y,\mathbb{R})$, we have: $q_\lambda(j^*a, j^*b)=\int_Y\mu_{[X]}(a).b.\lambda^{n-2}$. If $\mu_{[X]}(a)=0$, then $\forall b\in H^2(Y,\mathbb{R})$, and $q_X(j^*a, j^*b) = 0 $, i.e., $ j^*a$ is in the kernel of the restriction of $q_\lambda$ to $\mathrm{Im}(j^*)$. To show that $j^*a=0$, it suffices to see that the restriction of $q_\lambda$ to $\mathrm{Im}(j^*)$ is non-degenerate; this follows from the index theorem. By hypothesis, the image of $j^*$ is contained in $ H^2(X,\mathbb{R})\cap H^{1,1}(X):=H^{1,1}_{\mathbb{R}}(X)$; on $H^{1,1}_{\mathbb{R}}(X)$, the form $q_\lambda$ is non-degenerate, with signature $(1,h^{1,1}-1)$, and more precisely, we have $ q_\lambda(j^*\lambda)>0$, $q_\lambda$ is negative definite on the orthogonal of $j^*(\lambda)$. Since $\mathrm{Im}(j^*)$ contains an element of self-intersection $>0$, it is easy to see that $q_X$ is non-degenerate on $\mathrm{Im}(j^*)$.
\end{proof}

%\begin{lemma}
%Let $X$ be a projective hyper-K\"ahler manifold and $L$ be a smooth projective variety with a finite unramified morphism $i\colon L\to X$ such that the image of $i$ is Lagrangian. We denote by $W$ the kernel of the contraction morphism 
%\end{lemma}

\begin{lemma}\label{mukai-vector-class}
Let $X$ be a projective hyper-K\"ahler manifold and $i\colon L\to X$ be an immersed Lagrangian submanifold. Then $v(i_{*}\mathcal{O}_{L})=i_*\mathrm{exp}(c_1(L)/2)$. 
\end{lemma}

\begin{proof}
Let $\omega$ be the holomorphic symplectic $2$-form of $X$. We denote by $\cN^{\vee}_{L/X}$ the $-1$ degree cohomology sheaf of the cotangent complex $\mathbb{L}_{L/X}$ of $i$. When $i$ is a closed embedding, $\cN^{\vee}_{L/X}$ is exactly the conormal bundle of $L$ in $X$. Since $i$ is a local complete intersection morphism by \cite[\href{https://stacks.math.columbia.edu/tag/0E9K}{Tag 0E9K}]{stacks-project}, $\cN^{\vee}_{L/X}$ is a bundle (cf.~\cite[\href{https://stacks.math.columbia.edu/tag/0FK2}{Tag 0FK2}]{stacks-project} and \cite[\href{https://stacks.math.columbia.edu/tag/068B}{Tag 068B}]{stacks-project}). Using the standard exact triangle of cotangent complexes by  \cite[\href{https://stacks.math.columbia.edu/tag/0E44}{Tag 0E44}]{stacks-project}, we obtain exact sequences
\[0\to \cT_L\to \cT_X|_L\to \cN_{L/X}\to 0\]
and
\[0\to \cN^{\vee}_{L/X}\to \Omega_X|_L\to \Omega_L\to 0.\]
As $\omega$ is a holomorphic symplectic $2$-form, it induces an isomorphism $\omega\colon \cT_X\xra{\cong} \Omega_X$. To adapt the argument in \cite[Lemma 7.4]{beckmann:atomic-object}, it suffices to prove that there exists an isomorphism $\cN_{L/X}\cong \Omega_L$. According to \cite[Remark 4.4]{huybrechts:book-cubic-hypersurface}, there is a natural map $$s\colon \cN_{L/X}\to \Omega_L, \quad$$
locally given by $v\mapsto \omega(v,-)$. This yields a commutative diagram
% https://q.uiver.app/#q=WzAsMTAsWzAsMCwiMCJdLFsxLDAsIlRfTCJdLFsyLDAsIlRfWHxfTCJdLFszLDAsIk5fe0wvWH0iXSxbNCwwLCIwIl0sWzAsMSwiMCJdLFsxLDEsIk5ee1xcdmVlfV97TC9YfSJdLFsyLDEsIlxcT21lZ2FfWHxfTCJdLFszLDEsIlxcT21lZ2FfTCJdLFs0LDEsIjAiXSxbMiw3LCJcXG9tZWdhIl0sWzAsMV0sWzEsMl0sWzIsM10sWzMsNF0sWzMsOCwicyJdLFs4LDldLFs3LDhdLFs2LDddLFs1LDZdXQ==
\[\begin{tikzcd}
	0 & {\cT_L} & {\cT_X|_L} & {\cN_{L/X}} & 0 \\
	0 & {\cN^{\vee}_{L/X}} & {\Omega_X|_L} & {\Omega_L} & 0
	\arrow["\omega", from=1-3, to=2-3]
	\arrow[from=1-1, to=1-2]
	\arrow[from=1-2, to=1-3]
	\arrow[from=1-3, to=1-4]
	\arrow[from=1-4, to=1-5]
	\arrow["s", from=1-4, to=2-4]
	\arrow[from=2-4, to=2-5]
	\arrow[from=2-3, to=2-4]
	\arrow[from=2-2, to=2-3]
	\arrow[from=2-1, to=2-2]
\end{tikzcd}\]
because the composition $\cT_L\to \cT_X|_L\cong \Omega_X|_L\to \Omega_L$ is zero. In particular, $s$ is surjective. Since $\cN_{L/X}$ and $\Omega_L$ are locally free sheaves of the same rank on $L$, we deduce that $s$ is an isomorphism and the result follows.
\begin{comment}

To this end, by \cite[Theorem (1.2)]{rydh:unramified} we can find a smooth scheme and a morphism $\pi\colon U\to X$ such that $\pi$ is \'etale and finite type and $i_U\colon L\times_X U\to U$ is a closed embedding. Then we have a Cartesian diagram
% https://q.uiver.app/#q=WzAsNCxbMSwxLCJYIl0sWzAsMSwiTCJdLFswLDAsIlVfTCJdLFsxLDAsIlUiXSxbMywwLCJcXHBpIl0sWzEsMCwiaSJdLFsyLDMsImlfVSJdLFsyLDEsIlxccGlfVSIsMl1d
\[\begin{tikzcd}
	{U_L} & U \\
	L & X
	\arrow["\pi", from=1-2, to=2-2]
	\arrow["i", from=2-1, to=2-2]
	\arrow["{i_U}", from=1-1, to=1-2]
	\arrow["{\pi_U}"', from=1-1, to=2-1]
\end{tikzcd}\]
As $\pi$ is \'etale, $\pi^*\omega$ remains a holomorphic $2$-form on $U$ and $U_L$ is a smooth Lagrangian submanifold with respect to $\pi^*\omega$. Then we have natural isomorphism $s_U\colon N_{U_L/U}\xra{\cong} \Omega_{U_L}$ since in this case $i_U$ is a closed embedding. As $\pi_U$ is \'etale, we have natural identifications $\pi_U^*\Omega_{L}\cong \Omega_{U_L}$ and $\pi^*_U N_{L/X}\cong N_{U_L/U}$. Then it is straightforward to check that $\pi^*_U s\cong s_U$, which implies $s$ is an isomorphism by \'etale descent.
    
\end{comment}
\end{proof}

% \begin{lemma}
% Let $X$ be a projective hyper-K\"ahler manifold of dimension $2n$ and $L$ be a connected smooth projective variety with a finite unramified morphism $i\colon L\to X$ such that the image of $i$ is a Lagrangian subvariety. Then the morphism $$\rH^1(X,\mathcal{T}_{X})\rightarrow \rH^{*}(X,\CC),\quad \mu\mapsto\mu\lrcorner i_*([L])$$ is non-trivial.
% \end{lemma}

%\zy{converse direction seems also true....}

\section{Atomic sheaves on Bridgeland moduli spaces}\label{sec-atomic-sheaves}

In this section, we aim to construct atomic sheaves on $M^X_{\sigma_X}(a,b)$ for a GM fourfold or sixfold $X$ via two different methods. The main results in this section are Theorem \ref{thm-fixed-loci-1-obs} and Theorem \ref{thm-pushforward-atomic}. At the end of this section, we prove Theorem \ref{thm-bundle}.

\subsection{Restriction of degree $2$ cohomology}

We begin with a result on the restriction of degree $2$ cohomology.

\begin{proposition}\label{prop-H2-moduli-space}
Let $X$ be a non-Hodge-special GM fourfold or sixfold and $\sigma_X$ be a stability condition on $\Ku(X)$. Given a pair of coprime integers $a,b$ and a finite morphism $i\colon L\to M^X_{\sigma_X}(a,b)$ such that $L$ is a projective variety and the image of $i$ is Lagrangian, then the map
\[i^*\colon \rH^2(M^X_{\sigma_X}(a,b), \QQ) \to \rH^2(L, \QQ)\]
is of rank one.
\end{proposition}

\begin{proof}
Since the image of $i$ is Lagrangian, $\rH^{p,q}_{\QQ}(M^X_{\sigma_X}(a,b))\subset \ker(i^*)$ for $(p,q)=(2,0)$ and $(0,2)$. As $i^*$ is a non-trivial map by the finiteness, it is of rank one by Remark \ref{rmk-non-hodeg-special}.
\end{proof}

Applying Proposition \ref{prop-H2-moduli-space} to the morphisms in Theorem \ref{thm-FGLZ-moduli}, we obtain the following two corollaries.

\begin{corollary}\label{cor-rk1-H2-moduli}
Let $X$ be a non-Hodge-special GM fourfold or sixfold and $j\colon Y\hookrightarrow X$ be a smooth hyperplane section. Let $\sigma_X$ be a stability condition on $\Ku(X)$ and $\sigma_Y$ be a Serre-invariant stability condition on $\Ku(Y)$. Then, for any pair of coprime integers $a,b$, the map
\[r^*\colon \rH^2(M^X_{\sigma_X}(a,b), \QQ) \to \rH^2(M^Y_{\sigma_Y}(a,b), \QQ)\]
is of rank one, where
\[r\colon M^Y_{\sigma_Y}(a,b)\to M^X_{\sigma_X}(a,b)\] 
is induced by the functor $\pr_X\circ j_*$ as in Theorem \ref{thm-FGLZ-moduli}. Furthermore, if $\sigma_X$ is rational, then $(\ell_{\sigma_X})^{\perp}= \ker(r^*)$.
\end{corollary}

\begin{proof}
Using Theorem \ref{thm-FGLZ-moduli} and Proposition \ref{prop-H2-moduli-space}, we know that $r^*$ is of rank one. By Proposition~\ref{prop-restrict-BM-divisor}, we have $r^*\ell_{\sigma_X}\neq 0$. Then the statement $(\ell_{\sigma_X})^{\perp}= \ker(r^*)$, which is also equivalent to $\ell_{\sigma_X}\in \ker(r^*)^{\perp}$, follows from Remark \ref{rmk-bm-divisor-4fold} and the fact that $\ker(r^*)^{\perp}$ is the image of $(A_1^{\oplus 2})\cap (a\Lambda_1+b\Lambda_2)^{\perp}$.
\end{proof}

%\zy{How about Hodge-special cases? We need to analyze the Mukai vector of the universal family carefully...}

\begin{corollary}\label{cor-rk1-H2-moduli-5fold}
Let $X$ be a non-Hodge-special GM fourfold and $j\colon X\hookrightarrow Y$ be an embedding such that $Y$ is a GM fivefold. Let $\sigma_X$ be a stability condition on $\Ku(X)$ and $\sigma_Y$ be a Serre-invariant stability condition on $\Ku(Y)$. Then, for any pair of coprime integers $a,b$, the map
\[r^*\colon \rH^2(M^X_{\sigma_X}(a,b), \QQ) \to \rH^2(M^Y_{\sigma_Y}(a,b), \QQ)\]
is of rank one, where
\[r\colon M^Y_{\sigma_Y}(a,b)\to M^X_{\sigma_X}(a,b)\] 
is induced by the functor $j^*$ as in Theorem \ref{thm-FGLZ-moduli}. Furthermore, if $\sigma_X$ is rational, then we have $(\ell_{\sigma_X})^{\perp}= \ker(r^*)$.
\end{corollary}

\begin{proof}
This follows from the proof analogous to Corollary \ref{cor-rk1-H2-moduli}.
\end{proof}

\subsection{Fixed loci of involutions are $1$-obstructed}

Using Proposition \ref{prop-H2-moduli-space} and Corollary \ref{cor-rk1-H2-moduli}, we are ready to prove that the fixed locus of the natural involution on $M^X_{\sigma_X}(a,b)$ is an $1$-obstructed Lagrangian submanifold. It relies on the following result of Markman:

\begin{lemma}[{\cite[Remark 3.10(3)]{markman:rank-1-obstruction}}]\label{lem-markman-fixed}
Let $X$ be a projective hyper-K\"ahler manifold with a finite group $G$ acting on $X$. If $L\subset X$ is a union of Lagrangian connected components of the fixed locus of $G$ such that the restriction map $\rH^2(X,\mathbb{Q})\rightarrow \rH^2(L,\mathbb{Q})$ is of rank one, then $L$ is $1$-obstructed.
\end{lemma}

Recall that we have an anti-symplectic involution $\tau_X$ on $M^X_{\sigma_X}(a,b)$ (see Section \ref{subsec-lag-moduli}).

%\begin{comment}
\begin{theorem}\label{thm-fixed-loci-1-obs}
Let $X$ be a GM fourfold or sixfold and $a,b$ be a pair of coprime integers. Then for a generic stability condition $\sigma_X\in \Stab^{\circ}(\Ku(X))$, each union of connected components of the fixed locus of the natural involution
\[\mathrm{Fix}(\tau_X)\subset M^X_{\sigma_X}(a,b)\]
is an $1$-obstructed Lagrangian submanifold.
\end{theorem}
%\end{comment}

%\begin{theorem}\label{thm-fixed-loci-1-obs}
%Let $X$ be a non-Hodge-special GM fourfold and $a,b$ be a pair of coprime integers. Then for any stability condition $\sigma_X$ on $\Ku(X)$, each connected component of the fixed locus of the natural involution
%\[\mathrm{Fix}(\tau_X)\subset M^X_{\sigma_X}(a,b)\]
%is an $1$-obstructed Lagrangian submanifold.
%\end{theorem}

\begin{proof}
By Lemma \ref{lem-union-component-1obs}, we only need to show that $\mathrm{Fix}(\tau_X)\subset M^X_{\sigma_X}(a,b)$ is an $1$-obstructed Lagrangian submanifold. 

First, we assume that $X$ is non-Hodge-special. As in \cite[Section 8.5]{ppzEnriques2023}, up to an equivalence of Kuznetsov components, it is enough to prove the statement when $X$ is a non-Hodge-special special GM fourfold. Let $j\colon Y\hookrightarrow X$ be the branch divisor of $X$, which is an ordinary GM threefold. By \cite[Proposition 5.3]{FGLZ}, for any Serre-invariant stability condition $\sigma_Y$ on~$\Ku(Y)$, the functor $\pr_X\circ j_*$ induces a morphism
\[r\colon M^{Y}_{\sigma_Y}(a,b)\to M^{X}_{\sigma_X}(a,b),\]
where $r$ is \'etale finite onto $\mathrm{Fix}(\tau_X)$ of degree $2$. Consequently, the pull-back map $$\rH^2(\mathrm{Fix}(\tau_X), \QQ)\to \rH^2(M^{Y}_{\sigma_Y}(a,b),\QQ)$$ is injective. Since the composition
\[r^*\colon \rH^2(M^X_{\sigma_X}(a,b), \QQ)\to \rH^2(\mathrm{Fix}(\tau_X), \QQ)\to \rH^2(M^{Y}_{\sigma_Y}(a,b),\QQ)\]
is of rank one by Corollary \ref{cor-rk1-H2-moduli}, the restriction map
\[\rH^2(M^X_{\sigma_X}(a,b), \QQ)\to \rH^2(\mathrm{Fix}(\tau_X), \QQ)\]
also has rank one. Therefore, $\mathrm{Fix}(\tau_X)$ is an $1$-obstructed Lagrangian submanifold by Lemma \ref{lem-markman-fixed}.

%Not that if $L$ is a union of connected components of $\mathrm{Fix}(\tau_X)$, then the restriction $\rH^2(M^X_{\sigma_X}(a,b), \QQ)\to \rH^2(L, \QQ)$ is nontrivial as its image contains the restriction of ample divisors, hence is of rank one as well. Therefore, $L$ is an $1$-obstructed Lagrangian submanifold by Lemma \ref{lem-markman-fixed}.

Next, we assume that $X$ is an arbitrary smooth GM fourfold. By \cite[Theorem 4.18]{perry2019stability}, we may assume that $X$ is general in the sense that the canonical quadric surface in $X$ is smooth (cf.~\cite[Remark 2.2]{perry2019stability}). Therefore, using Theorem \ref{thm-relative-stability-fourfold}, we can find a smooth connected quasi-projective curve $C$ and a family $\cX\to C$ of GM fourfolds with a full triangulated subcategory $\Ku(\cX)\subset \Db(\cX)$ such that $\Ku(\cX)_c\simeq \Ku(\cX_c)$ for each $c\in C$. Additionally, 

\begin{enumerate}
    \item there is a point $0\in C$ such that $\cX_0\cong X$,

    \item there is a point $1\in C$ such that $\cX_1$ is non-Hodge-special,

    \item there exists a relative stability condition $\underline{\sigma}$ on $\Ku(\cX)$ over $C$ such that $\underline{\sigma}_c\in \Stab^{\circ}(\Ku(\cX_c))$ is generic with respect to $a\Lambda_1+b\Lambda_2$ for every $c\in X$ and $M^X_{\sigma_X}(a,b)=M^X_{\underline{\sigma}_0}(a,b)$, and

    \item the relative moduli space $M^{\cX}_{\underline{\sigma}}(a,b)$ is smooth and projective over $C$.%\footnote{Only the properness is stated in \cite[Proposition 5.6]{perry2019stability}. However, using the smoothness, \cite[Theorem 21.15]{BLMNPS21} and \cite[Corollary 3.4]{david:projective-criteria}, this relative moduli space is indeed a variety projective over $C$.}
\end{enumerate}
From \cite[Lemma 5.9]{bayer2022kuznetsov}, there is an auto-equivalence $T_{\cX}$ on $\Ku(\cX)$ such that $(T_{\cX})|_c\simeq T_{\cX_c}$ for each $c\in C$. By \cite[Proposition 5.7]{bayer2022kuznetsov} and the proof of \cite[Theorem 8.2]{ppzEnriques2023}, such equivalence induces an automorphism $\tau_{\cX}\colon M^{\cX}_{\underline{\sigma}}(a,b)\to M^{\cX}_{\underline{\sigma}}(a,b)$, which is an involution since it is an involution at each fiber over $C$ and $M^{\cX}_{\underline{\sigma}}(a,b)$ is smooth. Therefore, the fixed locus $\mathrm{Fix}(\tau_{\cX})$ is smooth projective over $C$ and satisfies $(\mathrm{Fix}(\tau_{\cX}))|_c\cong \mathrm{Fix}(\tau_{\cX_c})$ for each $c$. By Proposition \ref{prop-H2-moduli-space}, the restriction map
$$\rH^2(M^{\cX_1}_{\underline{\sigma}_1},\QQ)\to \mathrm{H}^2(\mathrm{Fix}(\tau_{\cX_1}),\QQ)$$
is of rank one, then the restriction map 
\[\rH^2(M^{X}_{\sigma_X},\QQ)\to \mathrm{H}^2(\mathrm{Fix}(\tau_{X}),\QQ)\]
is also of rank one. Thus, $\mathrm{Fix}(\tau_{X})$ is $1$-obstructed by Lemma \ref{lem-markman-fixed} and the result follows.
\end{proof}

%Moreover, using a deformation argument, we can extend the above theorem to general GM fourfolds. We first state a classical lemma that will be used later.

\begin{comment}

\begin{corollary}\label{cor-fixed-loci-general}
Let $X$ be a GM fourfold and $a,b$ be a pair of coprime integers. Then for any stability condition $\sigma_X\in \Stab^{\dagger}(\Ku(X))$ generic with respect to $a\Lambda_1+b\lambda_2$, each connected component of the fixed locus of the natural involution
\[\mathrm{Fix}(\tau_X)\subset M^X_{\sigma_X}(a,b)\]
is an $1$-obstructed Lagrangian submanifold.
\end{corollary}

\begin{proof}
By \cite[Theorem 4.18]{perry2019stability}, we can assume that $X$ is a general GM fourfold in the sense that the canonical quadric surface in $X$ is smooth (cf.~\cite[Remark 2.2]{perry2019stability}).
\end{proof}

\end{comment}

Now we deduce an application of Theorem \ref{thm-fixed-loci-1-obs}, which describes the first Chern class of a Bridgeland moduli space on a GM threefold or fivefold in terms of its Bayer--Macr\`i divisor. Later on, we will see that this implies one of the assumptions in Theorem \ref{atomic-criterion-thm}.

Recall that a stability condition $\sigma=(\cA_{\sigma}, Z_{\sigma})$ is rational if the image of $Z_{\sigma}$ is contained in $\QQ[\mathfrak{i}]$.

\begin{corollary}\label{cor-c1-moduli}
Let $Y$ be an ordinary GM threefold or fivefold. For any pair of coprime integers $a,b$ and Serre-invariant stability condition $\sigma_Y$ on $\Ku(Y)$, if $M^Y_{\sigma_Y}(a,b)$ is smooth (e.g.~$Y$ is general), we have
\[c_1(M^Y_{\sigma_Y}(a,b))\in \mathbb{R}\cdot \ell_{\sigma_Y} \subset \NS(M^Y_{\sigma_Y}(a,b))_{\mathbb{R}}.\]
Moreover, if $\sigma_Y$ is rational, we have 
\[c_1(M^Y_{\sigma_Y}(a,b))\in \QQ\cdot \ell_{\sigma_Y} \subset \NS(M^Y_{\sigma_Y}(a,b))_{\mathbb{Q}}.\]
\end{corollary}

\begin{proof}
If $Y$ is an ordinary GM fivefold, by \cite[Corollary 6.5]{kuznetsov2019categorical}, we can find an ordinary GM threefold $Y'$ with an equivalence $\Ku(Y)\simeq \Ku(Y')$. Thus, we only need to prove the statement when $Y$ is an ordinary GM threefold. As $Y$ is ordinary, it is the branch divisor of a special GM fourfold $X$. %Let $M$ be a connected component of $M^Y_{\sigma_Y}(a,b)$. It suffices to prove the statement after restricting to $M$. 

Let $\sigma_X\in \Stab^{\circ}(\Ku(X))$ be a generic stability condition on $\Ku(X)$ with $Z_{\sigma_X}(\Lambda_i)=Z_{\sigma_Y}(\lambda_i)$ for each $i=1,2$. By \cite[Proposition 5.3]{FGLZ}, the functor $\pr_X\circ j_*$ induces a morphism
\[r\colon M^{Y}_{\sigma_Y}(a,b)\to M^{X}_{\sigma_X}(a,b)\]
such that $r$ is an \'etale finite morphism onto its image $\mathrm{Fix}(\tau_X)$. We denote this covering map by $\pi\colon M^Y_{\sigma_Y}(a,b)\to \mathrm{Fix}(\tau_X)$ and denote the embedding by $i\colon \mathrm{Fix}(\tau_X)\hookrightarrow M^{X}_{\sigma_X}(a,b)$. Since $\mathrm{Fix}(\tau_X)$ is $1$-obstructed by Theorem~\ref{thm-fixed-loci-1-obs}, according to Theorem~\ref{1-obstructed-atomic} and Theorem \ref{atomic-criterion-thm}, we see $$c_1(\mathrm{Fix}(\tau_X))\in \im(i^*)\subset \rH^2(\mathrm{Fix}(\tau_X), \QQ),$$ and the restriction map 
\[i^*\colon \rH^2(M^{X}_{\sigma_X}(a,b), \QQ)\to \rH^2(\mathrm{Fix}(\tau_X), \QQ)\]
is of rank one. Therefore, $\pi^*c_1(\mathrm{Fix}(\tau_X))=c_1(M^Y_{\sigma_Y}(a,b))\in \im((i\circ \pi)^*)$. By Proposition \ref{prop-restrict-BM-divisor}, $(i\circ \pi)^*\ell_{\sigma_X}=\ell_{\sigma_Y}$ is ample, then $\ell_{\sigma_Y}\neq 0\in \im((i\circ \pi)^*)\otimes \mathbb{R}$. Thus, the first result follows from the fact that $\im((i\circ \pi)^*)\otimes \mathbb{R}=\im(r^*)\otimes \mathbb{R}$ is one-dimensional. When $\sigma_Y$ is rational, $\sigma_X$ is also rational, and we observe that 
$$(i\circ \pi)^*\ell_{\sigma_X}=\ell_{\sigma_Y}\neq 0\in \im((i\circ \pi)^*) \subset \rH^2(M, \QQ).$$
The result now follows from $c_1(M^Y_{\sigma_Y}(a,b))\in \im((i\circ \pi)^*)$ and $\dim_{\QQ} (\im((i\circ \pi)^*))=1$.
\end{proof}

\begin{remark}
In the preceding proof, we established that $c_1(\mathrm{Fix}(\tau_X))$ is proportional to $\ell_{\sigma_X}|_{\mathrm{Fix}(\tau_X)}$ for any GM fourfold or sixfold $X$. It is intriguing to determine the coefficient $t$ such that
\[c_1(\mathrm{Fix}(\tau_X))=t\cdot\ell_{\sigma_X}|_{\mathrm{Fix}(\tau_X)}\]
and to compare this result with the results in \cite{macri:involution-I,macri:involution-II}. We expect that $t$ depends solely on the integers $a$ and $b$.
\end{remark}

\subsection{Families of immersed atomic Lagrangians}
% As $1$-obstructed Lagrangian submanifolds constructed in Theorem \ref{thm-fixed-loci-1-obs} are fixed loci of involutions, they are likely rigid.
We have demonstrated that the Lagrangian submanifolds constructed in Theorem~\ref{thm-ppz} are 1-obstructed. Next, using Theorem \ref{thm-FGLZ-moduli}, we can construct a family of immersed atomic Lagrangians, which can deform over an open subset of the linear series of hyperplane sections.

\begin{comment}

\begin{theorem}\label{thm-pushforward-atomic}
Let $X$ be a non-Hodge-special GM fourfold and $j\colon Y\hookrightarrow X$ be a general hyperplane section. Let $\sigma_X$ be a stability condition on $\Ku(X)$ and $\sigma_Y$ be a Serre-invariant stability condition $\Ku(Y)$. Then for any pair of coprime integer $a,b$, the functor $\pr_X\circ j_*$ induces a finite unramified morphism
\[r\colon M^Y_{\sigma_Y}(a,b)\to M^X_{\sigma_X}(a,b),\]
such that $r_*\oh_{M}$ is an atomic sheaf on $M^X_{\sigma_X}(a,b)$ for each connected component $M$ of $M^Y_{\sigma_Y}(a,b)$.
\end{theorem}

\begin{proof}
The existence of a finite unramified morphism $r\colon M^Y_{\sigma_Y}(a,b)\to M^X_{\sigma_X}(a,b)$ follows from Theorem \ref{thm-FGLZ-moduli}. From Corollary \ref{cor-rk1-H2-moduli}, we see the restriction map
\[r|_M^*\colon \rH^2(M^X_{\sigma_X}(a,b), \QQ)\to \rH^2(M, \QQ)\]
has rank one. Then by Theorem \ref{atomic-criterion-thm}, to show $r_*\oh_{M}$ is an atomic sheaf, it is enough to verify that $c_1(M)\in \im(r|_M^*)$. To this end, up to $\GL$-action, we can assume that $\sigma_Y$ and $\sigma_X$ are rational and $Z_{\sigma_X}(\Lambda_i)=Z_{\sigma_Y}(\lambda_i)$ for each $i=1,2$ without changing the isomorphism classes of $M^Y_{\sigma_Y}(a,b)$ and $M^X_{\sigma_X}(a,b)$. Then by Proposition \ref{prop-restrict-BM-divisor}, we have
\[r|_M^*\ell_{\sigma_X}=\ell_{\sigma_Y}|_M\neq 0\in \rH^2(M, \QQ)\]
and the result follows from Corollary \ref{cor-c1-moduli}.
\end{proof}
 
\end{comment}

\begin{theorem}\label{thm-pushforward-atomic}
Given a pair of coprime integers $a,b$, and let $X$ be a GM fourfold or sixfold and $j\colon Y\hookrightarrow X$ be a smooth hyperplane section. We assume that

\begin{itemize}
    \item $X$ is general when $\dim X=4$, or

    \item $X$ is very general and $Y\subset X$ is a general hyperplane section when $\dim X=6$. 
\end{itemize}

Let $\sigma_X\in \Stab^{\circ}(\Ku(X))$ be a generic stability condition and $\sigma_Y$ be a Serre-invariant stability condition on $\Ku(Y)$. Then, there is a finite morphism
\[r\colon M^Y_{\sigma_Y}(a,b)\to M^X_{\sigma_X}(a,b),\]
such that $r_*\oh_{M^Y_{\sigma_Y}(a,b)}$ is an atomic sheaf on $M^X_{\sigma_X}(a,b)$. In particular, $r$ is an immersed atomic Lagrangian submanifold when $Y$ is general. 
\end{theorem}

\begin{proof}
The existence of a finite morphism $r\colon M^Y_{\sigma_Y}(a,b)\to M^X_{\sigma_X}(a,b)$ follows from Theorem \ref{thm-FGLZ-moduli}.

First, we assume that $X$ is very general and $Y\subset X$ is a general hyperplane section. Then $M^Y_{\sigma_Y}(a,b)$ is smooth by Lemma \ref{lem-structure-3fold-moduli}, and $r$ is unramified by Theorem \ref{thm-FGLZ-moduli}. According to Corollary~\ref{cor-rk1-H2-moduli}, the restriction map
\[r^*\colon \rH^2(M^X_{\sigma_X}(a,b), \QQ)\to \rH^2(M^Y_{\sigma_Y}(a,b), \QQ)\]
is of rank one. Then, by Theorem \ref{atomic-criterion-thm}, to show that $r_*\oh_{M^Y_{\sigma_Y}(a,b)}$ is an atomic sheaf, it suffices to verify $c_1(M^Y_{\sigma_Y}(a,b))\in \im(r^*)$. To this end, up to $\GL$-action, we can assume that $\sigma_Y$ and $\sigma_X$ are rational and $Z_{\sigma_X}(\Lambda_i)=Z_{\sigma_Y}(\lambda_i)$ for each $i=1,2$ without changing the isomorphism classes of $M^Y_{\sigma_Y}(a,b)$ and $M^X_{\sigma_X}(a,b)$. By Proposition \ref{prop-restrict-BM-divisor}, we have
\[r^*\ell_{\sigma_X}=\ell_{\sigma_Y}\neq 0\in \rH^2(M^Y_{\sigma_Y}(a,b), \QQ)\]
and the atomicity of $r_*\oh_{M^Y_{\sigma_Y}(a,b)}$ follows from Corollary \ref{cor-c1-moduli}.

Next, we assume that $X$ is a general GM fourfold and $Y\subset X$ is a smooth hyperplane section. Let $X'$ be a very general GM fourfold with the smooth canonical quadric and $Y'\subset X'$ is a general hyperplane section. Following this, we can find a smooth connected curve $C$ and a smooth projective family of GM fourfolds $\cX\to C$ with a closed subscheme $\cY\subset \cX$, which is smooth over $C$ such that $$\cY_0\cong Y, \quad\cY_1\cong Y', \quad\cX_0\cong X, \quad\text{and}\quad\cX_1\cong X'$$ for two points $0,1\in C$. After replacing $C$ with an open dense subscheme containing $0,1\in C$, we can assume that each $\cX_c$ satisfies the statement in Theorem \ref{thm-FGLZ-moduli} and the family of the canonical quadrics is smooth. Now, using Theorem \ref{thm-relative-stability-fourfold} as in the proof of Theorem \ref{thm-fixed-loci-1-obs}, there exists a relative stability condition $\underline{\sigma}$ on~$\Ku(\cX)$ over $C$ such that $\underline{\sigma}_c\in \Stab^{\circ}(\Ku(\cX_c))$ and $M^X_{\underline{\sigma}_0}(a,b)=M^X_{\sigma_X}(a,b)$. Similarly, by Theorem \ref{thm-relative-stability-threefold}, there is a relative stability condition $\underline{\sigma'}$ on $\Ku(\cY)$ over $C$ such that $\underline{\sigma'}_c$ is Serre-invariant for each $c\in C$.

Let $p\colon M^{\cY}_{\underline{\sigma'}}(a,b)\to C$ and $q\colon M^{\cX}_{\underline{\sigma}}(a,b)\to C$ be the relative moduli spaces, which are proper. We denote by $r_c\colon M^{\cY_c}_{\underline{\sigma'}_c}(a,b)\to M^{\cX_c}_{\underline{\sigma}_c}(a,b)$ the morphism in Theorem \ref{thm-FGLZ-moduli} for each $c\in C$. By assumption, the push-forward of $\cY\to \cX$ and the relative projection functor $\Db(\cX)\to \Ku(\cX)$ induce a morphism $r'\colon M^{\cY}_{\underline{\sigma'}}(a,b)\to M^{\cX}_{\underline{\sigma}}(a,b)$ such that $r'_c=r_c$. Therefore, $r'$ is proper and quasi-finite, hence finite. Moreover, by \cite[Proposition 5.1(1)]{FGLZ}, the fibers of $p$ are equidimensional of the same dimension and $M^{\cY}_{\underline{\sigma'}}(a,b)$ is also equidimensional. By Proposition~\ref{prop-immersed-deform}, this implies that $r_{0*}\oh_{M^Y_{\sigma_Y}(a,b)}$ is atomic because  $r_{1*}\oh_{M_{\underline{\sigma}_1}^{Y'}(a,b)}$ is atomic. %Also, note that $p$ is flat since each fiber of $p$ has constant pure dimension and $C$ is a smooth curve (cf.~\cite[Proposition 5.1(1)]{FGLZ}). Thus, from \cite[Corollary 2.27]{ku:hyperplane-section}, the sheaf $F:=r'_*\oh_{M^{\cY}_{\underline{\sigma'}}(a,b)}\in \Coh(M^{\cX}_{\underline{\sigma}}(a,b))$ satisfying
%\[F|_{q^{-1}(0)}\cong r_{0*}\oh_{M^Y_{\sigma_Y}(a,b)},\quad F|_{q^{-1}(1)}\cong r_{1*}\oh_{M_{\underline{\sigma'}_1}^{Y'}(a,b)}.\]
%This implies that $r_{0*}\oh_{M^Y_{\sigma_Y}(a,b)}$ is atomic by Proposition \ref{prop-atomic-lagrangian-deform}, since $r_{1*}\oh_{M_{\underline{\sigma}_1}^{Y'}(a,b)}$ is.
\end{proof}

\begin{remark}
By Lemma \ref{lem-union-component-1obs}, we obtain that $r_*\oh_{M}$ is an atomic sheaf on $M^X_{\sigma_X}(a,b)$ for each union of connected components of $M^Y_{\sigma_Y}(a,b)$ when $X$ is general ($\dim X=4$) or very general ($\dim X=6$), and $M^Y_{\sigma_Y}(a,b)$ is smooth (e.g.~$Y\subset X$ is a general hyperplane section). We expect that $M^Y_{\sigma_Y}(a,b)$ is always connected, which is true in all known explicit examples.
\end{remark}

Similarly, using Proposition \ref{prop-restrict-BM-divisor-5fold} and Corollary \ref{cor-rk1-H2-moduli-5fold}, the same argument as in the proof of Theorem \ref{thm-pushforward-atomic} yields:

\begin{theorem}\label{thm-push-atomic-5fold}
Let $X$ be a very general GM fourfold and $j\colon X\hookrightarrow Y$ be an embedding such that $Y$ is a general GM fivefold. Let $\sigma_X\in \Stab^{\circ}(\Ku(X))$ and $\sigma_Y$ be a Serre-invariant stability condition on $\Ku(Y)$. 

Then, there is a morphism
\[r\colon M^Y_{\sigma_Y}(a,b)\to M^X_{\sigma_X}(a,b)\]
such that $r$ is an immersed atomic Lagrangian submanifold.
\end{theorem}

\begin{remark}
Once a construction of relative stability conditions of families of GM fivefolds is known, we can easily extend Theorem \ref{thm-push-atomic-5fold} to the case when $X$ is a general GM fourfold and $Y$ is a smooth GM fivefold containing $X$.
\end{remark}

\subsection{Twisted atomic bundles}

In this subsection, we are going to prove Theorem \ref{thm-bundle}. 

Given a hyper-K\"ahler manifold $X$ and a Brauer class $\alpha\in \rH^2(X, \oh_X^*)$, we consider the twisted derived category $\Db(X, \alpha)$. In \cite[Section 4.4]{bottini23thesis} and \cite[Section 2]{bottini2024grady}, the author extends the theory of atomicity to the case of twisted objects in $\Db(X, \alpha)$ and proves some properties of twisted atomic sheaves in a specific setting. 

On the other hand, for a twisted vector bundle on $X$, the notion of projectively hyperholomorphic bundles is introduced in \cite{Ver96hyperholomorphic}, which has well-behaved deformation properties. We refer to \cite[Section 6]{markman:char-class} for more details.

\begin{proof}[{Proof of Theorem \ref{thm-bundle}}]
We fix a pair of coprime integers $a,b$ and set $n:=a^2+b^2+1$. Let $X$ be a general GM fourfold and $j\colon Y\hookrightarrow X$ be a general hyperplane section. By Theorem \ref{thm-pushforward-atomic}, we have an immersed atomic Lagrangian submanifold $r\colon M^Y_{\sigma_Y}(a,b)\to M^X_{\sigma_X}(a,b)$. Up to the $\GL$-action, we may assume that $\sigma_X$ and $\sigma_Y$ are rational with the same central charge. From Theorem \ref{thm-FGLZ-moduli} and Lemma \ref{lem-birational-comp}, we can choose a connected component $L\subset M^Y_{\sigma_Y}(a,b)$ such that $r$ maps $L$ birationally to its image $Z$ and hence $r(L)=Z$ is not contained in the fixed locus of the natural involution on $M^X_{\sigma_X}(a,b)$. The restriction of $r$ on $L$, denoted also by $r$, is a connected immersed atomic Lagrangian submanifold by Lemma \ref{lem-union-component-1obs}.

We claim that $\rH^1(L,\oh_L)\neq 0$, i.e.~$L$ is non-rigid. Assume for the contradiction that $\rH^1(L,\oh_L)$ vanishes. Up to equivalences of Kuznetsov components, we may assume that $X$ is special and $Y$ is not the branch divisor of $X$ by \cite[Theorem 1.6]{kuznetsov2019categorical} and \cite[Lemma 3.8]{kuznetsov2018derived}. Since $\rH^1(L,\oh_L)=0$, the Albanese variety of $L$ is trivial. In particular, the Abel--Jacobi map
\[L\to \mathrm{J}(Y)\]
defined by using the second Chern class is trivial, where $\mathrm{J}(Y)$ is the intermediate Jacobian of $Y$. Then \cite[Theorem 4.1(ii)]{liefu:cycle-gushel-mukai} implies that $c_2(E)=c_2(E')\in \CH^2(Y)$ for any two points $[E],[E']\in L$. Using \cite[Theorem 4.1(i)]{liefu:cycle-gushel-mukai}, this further implies $$\ch(E)=\ch(E')\in \CH^*(Y)_{\QQ}\quad\text{and}\quad\ch(\pr_X(j_*E))=\ch(\pr_X(j_*E'))\in \CH^*(X)_{\QQ}.$$ Therefore, from \cite[Proposition 6.6]{FGLZ}, we know that $r(L)\subset M^X_{\sigma_X}(a,b)$ is a Lagrangian constant cycle in the sense of \cite{voisin:remark-coisotropic}, i.e.~the set of rational equivalence classes of points of $r(L)$ in $\CH_0(M^X_{\sigma_X}(a,b))$ is a singleton. Now, we consider the diagram
\[\begin{tikzcd}
	{L(a,b)} & {M^X_{\sigma_X}(a,b)} \\
	V\subset |\oh_X(H)|
	\arrow["q", from=1-1, to=1-2]
	\arrow["p"', from=1-1, to=2-1]
\end{tikzcd}\]
as in the proof of \cite[Theorem 5.8]{FGLZ}, where $V\subset |\oh_X(H)|$ is the open subset parameterizing hyperplane sections of $X$ such that the moduli space of class $a\lambda_1+b\lambda_2$ is smooth and $p$ is the relative moduli space which is smooth projective. Moreover, $q|_{p^{-1}(s)}$ is the morphism constructed in Theorem \ref{thm-FGLZ-moduli}(1) for each $s\in V$. We denote by $0\in V$ the point corresponding to $Y$. Since $X$ is general, there is also a point $1\in V$ corresponding to the branch divisor of $X$. Up to taking the Stein factorization, we can assume that $p$ has connected fibers. Moreover, by replacing $L(a,b)$ with the component containing $L\subset p^{-1}(0)$, we can assume that $L(a,b)$ is irreducible. Since $p$ is a smooth proper family, we have $\rH^1(p^{-1}(s), \oh_{p^{-1}(s)})=\rH^1(L,\oh_L)=0$ for each $s\in V$. In particular, by the same argument as above, we see that $q(p^{-1}(s))$ is a Lagrangian constant cycle in $M^X_{\sigma_X}(a,b)$ for each $s\in V$. Therefore, applying \cite[Theorem 1.3(i)]{voisin:remark-coisotropic}, we conclude that $q(L(a,b))$ has dimension $n=\frac{1}{2}\dim M^X_{\sigma_X}(a,b)$. Since $L(a,b)$ is irreducible, $q(L(a,b))$ is also irreducible. Therefore, we have $q(L(a,b))=q(p^{-1}(s))$ for each $s\in V$ as $q(p^{-1}(s))\subset q(L(a,b))$ and they have the same dimension. However, this leads to a contradiction, since $q(p^{-1}(1))$ lies in the fixed locus of the natural involution on $M^X_{\sigma_X}(a,b)$, while $q(p^{-1}(0))=q(L)$ does not, by our choice of $L$. This proves the claim.

Now, back to the proof of our theorem. Recall that $M^X_{\sigma_X}(a,b)$ forms an open dense subset of the moduli space of projective hyper-K\"ahler manifolds of degree $2n-2$ and divisibility $n-1$ when $X$ varies in an open dense subset of the moduli space of GM fourfolds (cf.~\cite[Theorem 1.7]{perry2019stability}). By the Torelli Theorem \cite[Theorem 8.1]{Huy99}, we can take $X$ such that $\rH^{1,1}_{\ZZ}(M^X_{\sigma_X}(a,b))=\langle h, f \rangle$, where $h$ is the natural polarization and $f$ induces a Lagrangian fibration $M^X_{\sigma_X}(a,b)\to \PP^n$ such that $q(h,f)=d$ for a positive integer $d$. When $d\gg 0$, the movable cone of $M^X_{\sigma_X}(a,b)$ equals to the nef cone of $M^X_{\sigma_X}(a,b)$ by \cite[Corollary 2.7]{mongardi:mori-cone} and \cite[Lemma 4.3]{ogrady:modular}. 

Assume that $d$ is sufficiently large and coprime with $n-1=a^2+b^2$. Then, according to \cite[Proposition 5.12]{ogrady:rigid-modular} or \cite[Theorem 7.13]{Mar14lag}, we can take $X$ such that the fibration $M^X_{\sigma_X}(a,b)\to \PP^n$ is isomorphic to a Tate--Shafarevich twist of $M_S(0,\cL,1-n)\to |\cL|$, where $S$ is a K3 surface with  $\Pic(S)=\ZZ\cL$. Then, using \cite[Theorem 3.3]{bottini2024grady}, there is a Brauer class $\alpha\in \rH^2(M^X_{\sigma_X}(a,b),\oh^*_{M^X_{\sigma_X}(a,b)})$ and a twisted Poincar\'e sheaf on $M^X_{\sigma_X}(a,b)\times_{\PP^n} M_S(0,\cL,1-n)$ such that the corresponding Fourier--Mukai functor $\Phi$ is an equivalence
\[\Phi\colon \Db(M^X_{\sigma_X}(a,b))\xra{\simeq} \Db(M_S(0,\cL,1-n), \alpha).\]
In particular, for any $\cF\in \Pic^0(L)$, its image $\Phi(r_*\cF)\in \Db(M_S(0,\cL,1-n), \alpha)$ is a twisted atomic object by \cite[Corollary 3.10]{bottini2024grady}.

Now, we claim $f^n.r_*[L]\neq 0$. Indeed, by Corollary~\ref{cor-rk1-H2-moduli}, we have $h^{\perp}=\ker(r^*)$. As $q(h,f)=d\neq 0$, we see $f\notin \ker(r^*)$. By the atomicity, $r^*h$ spans $\im(r^*)$ and  $r^*f=t\cdot r^*h\in \rH^2(L,\QQ)$ for a non-zero $t\in \QQ$. Then, we have $f^n.r_*[L]=t^n\cdot h^n.r_*[L]\neq 0$. Since $r^*\colon \rH^2(M^X_{\sigma_X}(a,b),\QQ) \to \rH^2(L,\QQ)$ is of rank one, we see that the pullback $\cL'$ of $\oh_{\PP^n}(1)$ to $L$ is a non-zero multiple of an ample class, hence is also ample since $\cL'$ has global sections. This implies the finiteness of $L\xra{r} M^X_{\sigma_X}(a,b)\to \PP^n$ and $Z\to \PP^n$.

As $r$ is finite, $r_*\cF$ is a rank one Cohen--Macaulay sheaf on $Z$ for any $\cF\in \Pic^0(L)$ by \cite[Corollary 2.73]{Kollar13}. Then, from the fact that $Z\to \PP^n$ is finite, $\Phi(r_*\cF)$ is a twisted atomic bundle by \cite[Proposition 6.2]{bottini2024grady}. So \cite[Theorem 6.6]{bottini2024grady} implies that $\Phi(r_*\cF)$ is stable with respect to any suitable polarization. 

Up to now, we get a stable, atomic, twisted bundle $\Phi(r_*\cF)$ which is non-rigid by $\rH^1(L,\oh_L)\neq 0$ and Lemma \ref{lem-immerse-ext1} below. The rest of the argument is similar to the proof of \cite[Theorem 1.1]{bottini:towards-OG10}. Note that \cite[Lemma 5.4]{beckmann:atomic-object} also holds for the twisted case, then \cite[Corollary 6.12, Lemma 7.2]{markman:char-class} imply that~$\Phi(r_*\cF)$ is projectively hyperholomorphic with respect to a suitable polarization. Thus, according to \cite[Section 6]{markman:char-class}, for any hyper-K\"ahler manifold $M'$ of $\mathrm{K3^{[n]}}$-type, there is a stable, atomic, projectively hyperholomorphic twisted bundle $F'$ on $M'$. Finally, $F'$ remains non-rigid by \cite[Proposition 6.3]{verbitsky:generic-k3-tori}.

%Therefore, the proof of \cite[Theorem 3.4]{markman:rank-1-obstruction} shows that for any hyper-K\"ahler manifold $M$ of $\mathrm{K3^{[n]}}$-type, there is a flat deformation of $\Phi(r_*\cF)$ to a twisted bundle $G$ on $M$ such that the deformation is given by a generic twistor path. In particular, $G$ is also twisted atomic, and it is stable and projectively hyperholomorphic by \cite[Lemma 6.15, Corollary 6.12]{markman:char-class}. Finally, the deformation space of $G$ is isomorphic to that of $\Phi(r_*\cF)$ by \cite[Proposition 6.3]{verbitsky:generic-k3-tori} and the non-rigidity follows.
\end{proof}

We expect that the bundle constructed in Theorem \ref{thm-bundle} has a $20$-dimensional deformation space. This can be probably deduced from the connectedness of $M^Y_{\sigma_Y}(a,b)$ and the fact that $b_1(M^Y_{\sigma_Y}(a,b))=b_3(Y)=20$, together with Lemma \ref{lem-immerse-ext1}. We anticipate that the arguments used in \cite{LLPZ:higher-dim-moduli} for cubic threefolds can be applied in this case as well.

\begin{lemma}\label{lem-immerse-ext1}
Let $M$ be a projective hyper-K\"ahler manifold and $r\colon L\to M$ be an immersed connected Lagrangian submanifold. Then for any line bundle $\cL$ on $L$, we have
\[\rH^1(L,\oh_L)\subset \Ext_M^1(r_*\cL, r_*\cL).\]
Moreover, if $r$ is a closed embedding away from a closed subset of codimension $\geq 2$ in $L$, then we have
\begin{equation}\label{eq-ineq}
    \dim \Ext_M^1(r_*\cL, r_*\cL)\leq \dim \rH^1(L,\CC)
\end{equation}
and the equality holds if there exists a smooth projective family $L_V\to V$ over an irreducible variety $V$ with $\dim V=\dim \rH^1(L,\oh_L)$ and a morphism $r_V\colon L_V\to M$ such that

\begin{enumerate}
    \item there exists a point $0\in V$ such that $r_0=r\colon L=L_0\to M$,

    \item $r_s\colon L_s\to M$ is birational onto its image and is an immersed Lagrangian for each $s\in V$, and

    \item for any general point $s_0\in V$, the set of points $s\in V$ such that $r_s(L_s)=r_{s_0}(L_{s_0})\subset M$ is finite.
\end{enumerate}

\end{lemma}

\begin{proof}
By the adjunction of functors, we have
\[\Ext_M^1(r_*\cL, r_*\cL)=\Ext_L^1(r^*r_*\cL,\cL).\]
Since $\cH^{i}(r^*r_*\cL)=0$ for any $i>0$, we have an exact sequence
\[0\to \Ext^1_L(\cL, \cL)\to \Ext_L^1(r^*r_*\cL,\cL)\to \Hom_L(\cH^{-1}(r^*r_*\cL),\cL)\to \Ext^2_L(\cL, \cL).\]
This proves the first statement and gives an inclusion
\begin{equation}\label{eq-inclusion}
    \Ext_L^1(r^*r_*\cL,\cL)\subset \Ext^1_L(\cL,\cL)\oplus \Hom_L(\cH^{-1}(r^*r_*\cL),\cL).
\end{equation}

Since $\cL$ is locally free and $L$ is normal, we get
\begin{equation}\label{eq-double-dual}
 \Hom_L(\cH^{-1}(r^*r_*\cL),\cL)=\Hom_L((\cH^{-1}(r^*r_*\cL))^{\vee \vee},\cL),
\end{equation}
where $(\cH^{-1}(r^*r_*\cL))^{\vee \vee}$ denotes the double dual (reflexive hull) of the sheaf $\cH^{-1}(r^*r_*\cL)$.

Now, let $W\subset L$ be an open subset with $\mathrm{codim}_L(L\setminus W)\geq 2$ such that $r|_{W}$ is a closed embedding. Then by \cite[Proposition 11.8]{huybrechts2006fourier}, we have
\[(\cH^{-1}(r^*r_*\cL))|_{W}\cong \cH^{-1}((r|_W)^*(r|_W)_*\cL|_W)\cong(\cN^{\vee}_{L/M}\otimes \cL)|_{W}.\]
Therefore, using the fact that  $\mathrm{codim}_L(L\setminus W)\geq 2$, the locally freeness of $\cN^{\vee}_{L/M}\otimes \cL$, and \cite[\href{https://stacks.math.columbia.edu/tag/0E9I}{Tag 0E9I}]{stacks-project}, we have
\[(\cH^{-1}(r^*r_*\cL))^{\vee \vee}\cong \cN^{\vee}_{L/M}\otimes \cL,\]
which together with $\cN_{L/M}\cong \Omega_L$, \eqref{eq-double-dual}, and \eqref{eq-inclusion} implies
\[\Ext_M^1(r_*\cL,r_*\cL)=\Ext_L^1(r^*r_*\cL,\cL)\subset \rH^1(L,\CC)\]
and we get \eqref{eq-ineq}.

Finally, let $P\to V$ be the component of the relative Picard scheme of $L_V\to V$ containing the point corresponding to $\cL$. Then we have $$\dim P=\dim V+\dim \rH^1(L,\oh_L)=\dim \rH^1(L,\CC)$$ by our assumption. Since $r_s$ is birational onto its image, we know that $r_{s_*}\cL_s$ is a stable sheaf on $M$ for each $s\in V$ and line bundle $\cL_s$ on $L_s$. Hence, up to shrinking $V$ if necessary, $r_V$ induces a morphism $P\to N$, where $N$ is an irreducible component of the moduli space of stable sheaves on $M$ such that $[r_*\cL]\in N$. By the assumption (3), we know that $P\to N$ has finite general fibers, which gives $\dim N\geq \dim P= \dim \rH^1(L,\CC)$. Therefore, we obtain
\[\dim \Ext_M^1(r_*\cL, r_*\cL)=\dim T_{[r_*\cL]} N\geq \dim N\geq \dim \rH^1(L,\CC),\]
which implies the equality in \eqref{eq-ineq}.
\end{proof}

\section{A new example of non-rigid atomic Lagrangian submanifolds}\label{sec-new-example}

In this section, we focus on an explicit locally complete family of projective hyper-K\"ahler fourfolds associated with GM varieties, called \emph{double dual EPW sextics}. The goal is to construct non-rigid immersed atomic Lagrangian submanifolds in these hyper-K\"ahler fourfolds (cf.~Theorem~\ref{thm-double-epw} and Theorem~\ref{thm-double-epw-5fold}).

%In Section \ref{subsec-llsvs}, we explore a family of projective hyper-K\"ahler eightfolds associated with cubic fourfolds, called \emph{LLSvS eightfolds}. We show that the natural Lagrangian submanifolds constructed from cubic threefolds are \emph{not} atomic, even though the restriction map on degree $2$ cohomology has rank one.

%First, for an $20$-dimensional family of hyper-K\"ahler fourfolds, we prove:

%Here, $\wt{\mathsf{Y}}_{A(X)^{\perp}}$ is a projective hyper-K\"ahler manifold of $\mathrm{K3}^{[\mathrm{2}]}$-type, called the double dual EPW sextic. See Section \ref{subsec-double-epw} for the definition of $\wt{\mathsf{Y}}_{A(Y)^{\perp}}^{\geq 2}$ and $\wt{\mathsf{Y}}_{A(X)^{\perp}}$.

\begin{comment}

Next, we have:

\begin{theorem}\label{thm-llsvs}
Let $X$ be a general cubic fourfold and $Y\hookrightarrow X$ be a general hyperplane section. Then $\mathsf{Z}_Y\hookrightarrow \mathsf{Z}_X$ is a connected atomic Lagrangian submanifold with
\[\rH^1(\mathsf{Z}_Y, \CC)=\CC^{10}.\]
\end{theorem}

The variety $\mathsf{Z}_X$ above is a projective hyper-K\"ahler manifold of $\mathrm{K3}^{[\mathrm{4}]}$-type, called the LLSvS eightfold. See Section \ref{subsec-llsvs} for a review of $\mathsf{Z}_X$ and $\mathsf{Z}_Y$.
    
\end{comment}

%\zy{formality of RHom, skew-symmetric}

%\zy{deform to general case}

\subsection{Double dual EPW sextics}\label{subsec-double-epw}

Recall that by \cite[Theorem 3.10]{debarre2015gushel}, each smooth GM variety $X$ of dimension $\dim X\geq 3$ is uniquely determined by its Lagrangian data $$(V_6(X), V_5(X), A(X))$$ up to isomorphism, where $V_6(X)$ is a $6$-dimensional vector space, $V_5(X)\subset V_6(X)$ is a hyperplane, and $A(X)\subset \bigwedge^3 V_6(X)$ is a Lagrangian subspace with respect to the symplectic form given by the wedge product. For simplicity, we set $V_6:=V_6(X)$. For each integer $k$, we define subschemes 
\[\mathsf{Y}^{\geq k}_{A(X)}:=\{[v]\in\mathbb{P}(V_6) \mid \mathrm{dim}(A(X)\cap(v\wedge\bigwedge^2V_6))\geq k\}\subset \mathbb{P}(V_6)\]
and
\[\mathsf{Y}^{\geq k}_{A(X)^{\perp}}=\{[V_5]\in\mathrm{Gr}(5,V_6) \mid \mathrm{dim}(A(X)\cap\bigwedge^3V_5)\geq k\}\subset \mathbb{P}(V_6^{\vee})\]
as in \cite[Section 2]{o2006irreducible}. See also \cite[Appendix B]{debarre2015gushel} for an overview. We set $\mathsf{Y}_{A(X)}:=\mathsf{Y}_{A(X)}^{\geq 1}$ and $\mathsf{Y}_{A(X)^{\perp}}:=\mathsf{Y}_{A(X)^{\perp}}^{\geq 1}$, which are called the \emph{Eisenbud--Popescu--Walter (EPW) sextic} and the \emph{dual EPW sextic} associated with $X$, respectively.

Moreover, by \cite[Section 1.2]{o2010double} and \cite[Theorem 5.2]{debarre2020double}, as $A(X)$ does not contain decomposable vectors\footnote{This means $\PP(A(X))\cap \Gr(3,V_6)=\varnothing$, which holds for $A(X)$ when $X$ is smooth and $\dim X\geq 3$ by \cite[Theorem 3.16]{debarre2015gushel}.}, we have natural double coverings
\[\wt{\mathsf{Y}}_{A(X)}\to \mathsf{Y}_{A(X)}, \quad \wt{\mathsf{Y}}_{A(X)^{\perp}}\to \mathsf{Y}_{A(X)^{\perp}}.\]
We call $\wt{\mathsf{Y}}_{A(X)}$ and $\wt{\mathsf{Y}}_{A(X)^{\perp}}$ the  \emph{double EPW sextic} and the \emph{double dual EPW sextic} associated with $X$, respectively. According to \cite{o2006irreducible}, when $X$ is general, $\wt{\mathsf{Y}}_{A(X)}$ and $\wt{\mathsf{Y}}_{A(X)^{\perp}}$ are projective hyper-K\"ahler manifolds of $\mathrm{K3}^{[\mathrm{2}]}$-type.

Similarly, we have natural double coverings
\[\wt{\mathsf{Y}}^{\geq 2}_{A(X)}\to \mathsf{Y}^{\geq 2}_{A(X)}, \quad \wt{\mathsf{Y}}^{\geq 2}_{A(X)^{\perp}}\to \mathsf{Y}^{\geq 2}_{A(X)^{\perp}}\]
and we call $\wt{\mathsf{Y}}^{\geq 2}_{A(X)}$ and $\wt{\mathsf{Y}}^{\geq 2}_{A(X)^{\perp}}$ the  \emph{double EPW surface} and the \emph{double dual EPW surface} associated with $X$, respectively. They are \'etale coverings between smooth connected surfaces when $X$ is general.

In the rest of this section, we will mainly focus on (double) dual EPW varieties, as they exhibit a clearer connection to GM varieties.

We start with some classical results, which are known to experts. While we state it for EPW varieties or dual EPW varieties, analogous results hold for both cases.

\begin{lemma}\label{lem-isolated-sing}
Let $(V_6, V_5, A)$ and $(V_6, V_5, A')$ be two Lagrangian data such that $\dim(A\cap A')=9$ and both $A$ and $A'$ contain no decomposable vectors. If $A'$ is general, then $\mathsf{Y}_{A}^{\geq 2}\cap \mathsf{Y}_{A'}^{\geq 2}\subset \mathsf{Y}_{A}$ is zero-dimensional and
\[\overline{\mathsf{Y}}^{\geq 2}_{A'}:=\wt{\mathsf{Y}}_{A}\times_{\mathsf{Y}_{A}} \mathsf{Y}_{A'}^{\geq 2}\]
has isolated singularities.
\end{lemma}

\begin{lemma}\label{lem-compatible-double-cover}
Let $X$ be an ordinary GM fourfold. Let $Y$ be an ordinary GM variety such that either $Y\hookrightarrow X$ or $X\hookrightarrow Y$ is a smooth hyperplane section.

\begin{enumerate}
    \item For each integer $k$, there exists a natural embedding $\mathsf{Y}^{\geq k+1}_{A(Y)^{\perp}}\hookrightarrow \mathsf{Y}^{\geq k}_{A(X)^{\perp}}$.

    %\item $\mathsf{Y}^{\geq 2}_{A(Y)}\neq \mathsf{Y}^{\geq 2}_{A(Y')}$ in $\mathsf{Y}_{A(X)}$......(look at H1)

    \item Let $\overline{\mathsf{Y}}^{\geq 2}_{A(Y)^{\perp}}:=\wt{\mathsf{Y}}_{A(X)^{\perp}}\times_{\mathsf{Y}_{A(X)^{\perp}}} \mathsf{Y}_{A(Y)^{\perp}}^{\geq 2}$. If $Y$ is general, then $\overline{\mathsf{Y}}^{\geq 2}_{A(Y)^{\perp}}$ has isolated singularities and the natural double covering $\wt{\mathsf{Y}}^{\geq 2}_{A(Y)^{\perp}}\to \mathsf{Y}^{\geq 2}_{A(Y)^{\perp}}$ can be decomposed as the composition $\wt{\mathsf{Y}}^{\geq 2}_{A(Y)^{\perp}}\xra{\iota} \overline{\mathsf{Y}}^{\geq 2}_{A(Y)^{\perp}} \to \mathsf{Y}^{\geq 2}_{A(Y)^{\perp}}$ such that $\iota$ is the normalization map. Moreover, we have a commutative diagram
% https://q.uiver.app/#q=WzAsNSxbMSwxLCJcXHd0e1xcbWF0aHNme1l9fV97QShYKV57XFxwZXJwfX0iXSxbMSwwLCJcXG92ZXJsaW5le1xcbWF0aHNme1l9fV57XFxnZXEgMn1fe0EoWSlee1xccGVycH19Il0sWzIsMSwiXFxtYXRoc2Z7WX1fe0EoWClee1xccGVycH19Il0sWzIsMCwiXFxtYXRoc2Z7WX1ee1xcZ2VxIDJ9X3tBKFkpXntcXHBlcnB9fSJdLFswLDAsIlxcd3R7XFxtYXRoc2Z7WX19XntcXGdlcSAyfV97QShZKV57XFxwZXJwfX0iXSxbMSwwLCIiLDIseyJzdHlsZSI6eyJ0YWlsIjp7Im5hbWUiOiJob29rIiwic2lkZSI6InRvcCJ9fX1dLFszLDIsIiIsMCx7InN0eWxlIjp7InRhaWwiOnsibmFtZSI6Imhvb2siLCJzaWRlIjoidG9wIn19fV0sWzAsMl0sWzQsMCwiaSIsMl0sWzQsMSwiXFxpb3RhIl0sWzEsM11d
\[\begin{tikzcd}
	{\wt{\mathsf{Y}}^{\geq 2}_{A(Y)^{\perp}}} & {\overline{\mathsf{Y}}^{\geq 2}_{A(Y)^{\perp}}} & {\mathsf{Y}^{\geq 2}_{A(Y)^{\perp}}} \\
	& {\wt{\mathsf{Y}}_{A(X)^{\perp}}} & {\mathsf{Y}_{A(X)^{\perp}}}
	\arrow["\iota", from=1-1, to=1-2]
	\arrow["i"', from=1-1, to=2-2]
	\arrow[from=1-2, to=1-3]
	\arrow[hook, from=1-2, to=2-2]
	\arrow[hook, from=1-3, to=2-3]
	\arrow[from=2-2, to=2-3]
\end{tikzcd}\]
such that $i$ is finite and unramified.% The same holds for $\wt{\mathsf{Y}}^{\geq 2}_{A(Y)}$.
\end{enumerate}
\end{lemma}

\begin{proof}
Since $X$ and $Y$ are both ordinary, by \cite[Propoition 3.14]{debarre2015gushel}, the Lagrangian data of $X$ and $Y$ are $(V_6, V_5, A(X))$ and $(V_6, V_5, A(Y))$ respectively, with $\dim (A(X)\cap A(Y))= 9$. Therefore, part (1) follows directly from the above definition of EPW varieties. 

Let $\iota\colon W \to \overline{\mathsf{Y}}^{\geq 2}_{A(Y)^{\perp}}$ be the normalization map. Since $Y$ is general, we know that $\wt{\mathsf{Y}}^{\geq 2}_{A(Y)^{\perp}}$ and $\mathsf{Y}_{A(Y)^{\perp}}^{\geq 2}$ are smooth. Therefore, by Lemma \ref{lem-isolated-sing}, $\overline{\mathsf{Y}}^{\geq 2}_{A(Y)^{\perp}}$ is only singular along the zero-dimensional locus 
\[\overline{\mathsf{Y}}^{\geq 2}_{A(Y)^{\perp}}\cap \mathsf{Y}_{A(X)^{\perp}}^{\geq 2}\subset \wt{\mathsf{Y}}_{A(X)^{\perp}}.\]
Consequently, $\iota$ is finite and unramified in codimension one. Thus, the natural map $W\to \mathsf{Y}_{A(Y)^{\perp}}^{\geq 2}$ is also unramified in codimension one. Then, by the smoothness of $\mathsf{Y}_{A(Y)^{\perp}}^{\geq 2}$ and \cite[Theorem 1.6]{griffith:normal-extension}, it follows that $W\to \mathsf{Y}_{A(Y)^{\perp}}^{\geq 2}$ is unramified, and so is $\iota$.

Moreover, since $W\to \mathsf{Y}_{A(Y)^{\perp}}^{\geq 2}$ is unramified and finite of degree two, by \cite[\href{https://stacks.math.columbia.edu/tag/0B3D}{Tag 0B3D}]{stacks-project} and \cite[\href{https://stacks.math.columbia.edu/tag/00R4}{Tag 00R4}]{stacks-project}, it is flat and hence an \'etale double covering. In particular, $W$ is smooth.

Finally, to show $W\cong \wt{\mathsf{Y}}_{A(Y)^{\perp}}^{\geq 2}$ and the diagram in (2) commutes, we use the uniqueness part of \cite[Theorem 5.2(2)]{debarre2020double}. We first recall some notions in \cite[Section 5.1]{debarre2020double}. Let $(V_6, V_5, A)$ and $(V_6, V_5, A')$ be two Lagrangian data such that $\dim(A\cap A')=9$ and both $A$ and $A'$ are general. We define three bundles by $\mathscr{A}_1:=A\otimes \oh_{\PP(V_6)}$, $\mathscr{A}_1':=A'\otimes \oh_{\PP(V_6)}$, and $\mathscr{A}_2:=\wedge^2 T_{\PP(V_6)}(-3)$. These bundles are Lagrangian subbundles of the trivial symplectic bundle $\mathscr{V}:=\wedge^3V_6\otimes \oh_{\PP(V_6)}$ on $\PP(V_6)$. As in \cite[Section 4]{debarre2020double}, the symplectic structure induces injective maps
\[\omega_{\mathscr{A}_1, \mathscr{A}_2}\colon \mathscr{A}_1\to \mathscr{A}_2^{\vee}\]
and
\[\omega_{\mathscr{A}_1', \mathscr{A}_2}\colon \mathscr{A}_1'\to \mathscr{A}_2^{\vee}\]
We set $\mathscr{C}:=\mathrm{cok}(\omega_{\mathscr{A}_1, \mathscr{A}_2})$ and $\mathscr{C}':=\mathrm{cok}(\omega_{\mathscr{A}_1', \mathscr{A}_2})$. Then as in \cite[Section 5]{debarre2020double}, $\mathsf{Y}^{\geq k}_{A}$ and $\mathsf{Y}^{\geq k}_{A'}$ are exactly the rank $\geq k$ loci of sheaves $\mathscr{C}$ and $\mathscr{C}'$, respectively. 

Let $U:=\mathsf{Y}^{ \geq 2}_{A'}\setminus \mathsf{Y}^{\geq 2}_{A}$. We claim that $\mathscr{R}_1|_{U}\cong \mathscr{R}'_2|_U$, where $\mathscr{R}_1:=(\mathscr{C}|_{\mathsf{Y}_{A}})^{\vee \vee}$ and $\mathscr{R}'_2:=(\wedge^2 (\mathscr{C}'|_{\mathsf{Y}^{\geq 2}_{A'}}))^{\vee \vee}$. Note that $\mathscr{C}|_U\cong\mathscr{R}_1|_U$ is a line bundle on $U$ by $U\subset \mathsf{Y}^1_{A}$ and $\mathscr{R}'_2$ is also a line bundle since $\mathsf{Y}^{ \geq 3}_{A'}=\varnothing$. To prove the claim, we consider $$\mathscr{F}:=\mathrm{cok}(A\cap A'\otimes \oh_{\PP(V_6)}\to \mathscr{A}_2^{\vee}),$$ which fits into two exact sequences
\begin{equation}\label{eq-X}
    0\to \oh_{\PP(V_6)}\to \mathscr{F}\to \mathscr{C}\to 0
\end{equation}
and
\begin{equation}\label{eq-Y}
    0\to \oh_{\PP(V_6)}\to \mathscr{F}\to \mathscr{C}'\to 0.
\end{equation}
From the definition, we know that the rank $\geq 3$ locus of $\mathscr{F}$ can be described as 
\[\{[v]\in\mathbb{P}(V_6) \mid \mathrm{dim}(A\cap A'\cap(v\wedge\bigwedge^2V_6))\geq 2\},\]
which is contained in $\mathsf{Y}^{\geq 2}_{A'}\cap \mathsf{Y}^{\geq 2}_{A}$. Therefore, the fiber of $\mathscr{F}$ has rank $\leq 2$ at each point of $U$ by the definition of $U$. As $\mathscr{C}'|_{\mathsf{Y}^{\geq 2}_{A'}}$ is locally free of rank $2$ by definition, then by restricting the exact sequence \eqref{eq-Y} to $U$, we get an isomorphism $\mathscr{F}|_{U}\cong \mathscr{C}'|_U$. In particular, $\mathscr{F}|_{U}$ is a bundle of rank $2$ on $U$. On the other hand, as $\mathsf{Y}^{\geq 2}_{A'}\neq \mathsf{Y}^{\geq 2}_{A}$, we know that $\mathscr{C}|_{\mathsf{Y}^{\geq 2}_{A'}}$ has generic rank $1$ on $\mathsf{Y}^{\geq 2}_{A'}$ as $\mathsf{Y}^{\geq 2}_{A'}$ is integral by \cite[Theorem 5.1]{debarre2020double}. Therefore, the kernel $\mathscr{F}|_U\twoheadrightarrow \mathscr{C}|_U$ has generic rank $1$ on $U$ and the pull-back of \eqref{eq-X} gives an exact sequence
\[ 0\to \oh_{U}\to \mathscr{F}|_U\to \mathscr{C}|_U\to 0.\]
Combing this with $\mathscr{F}|_{U}\cong \mathscr{C}'|_U$ implies $$\mathscr{R}'_2|_{U}=\wedge^2(\mathscr{C}'|_U)\cong \mathscr{C}|_U=\mathscr{R}_1|_U.$$ This proves the claim.

Now, back to the proof of (2). We set $A:=A(X)^{\perp}$ and $A':=A(Y)^{\perp}$ as above. By the uniqueness part of \cite[Theorem 5.2(2)]{debarre2020double}, we need to show the \'etale covering $\gamma\colon W\to \mathsf{Y}_{A'}^{\geq 2}$ satisfies
\[\mathrm{cok}(\oh_{\mathsf{Y}_{A'}^{\geq 2}}\to \gamma_*\oh_W)\cong \mathscr{R}'_2(-3),\]
where $\mathscr{R}'_2:=(\wedge^2 (\mathscr{C}'|_{\mathsf{Y}^{\geq 2}_{A'}}))^{\vee \vee}$. Since $\gamma$ coincides with the base change of
\[\widetilde{\mathsf{Y}}_{A}\to \mathsf{Y}_{A}\]
along $\mathsf{Y}^{\geq 2}_{A'}\hookrightarrow \mathsf{Y}_{A}$ outside $\mathsf{Y}^{\geq 2}_{A}$, by \cite[Theorem 5.2(1)]{debarre2020double} and the base change theorem, we have
\[(\mathscr{R}_1(-3))|_U=((\mathscr{C}|_{\mathsf{Y}_{A}}(-3))^{\vee \vee})|_U\cong \mathrm{cok}(\oh_{\mathsf{Y}_{A'}^{\geq 2}}\to \gamma_*\oh_W)|_U\]
for $U:=\mathsf{Y}^{ \geq 2}_{A'}\setminus \mathsf{Y}^{\geq 2}_{A}$. Then, from the claim we just proved, this implies
\[\mathscr{R}'_2(-3)|_U\cong\mathrm{cok}(\oh_{\mathsf{Y}_{A'}^{\geq 2}}\to \gamma_*\oh_W)|_U.\]
From $\dim(\mathsf{Y}_{A'}^{\geq 2}\setminus U)=0$ by Lemma \ref{lem-isolated-sing}, the normality of $\mathsf{Y}_{A'}^{\geq 2}$, and \cite[\href{https://stacks.math.columbia.edu/tag/0AY6}{Tag 0AY6}]{stacks-project}, we obtain 
\[\mathscr{R}'_2(-3)\cong\mathrm{cok}(\oh_{\mathsf{Y}_{A'}^{\geq 2}}\to \gamma_*\oh_W).\]
This finishes the proof of (2).
\end{proof}

\begin{remark}
It is suggested in \cite[Section 5.1]{iliev2011fano} that $i$ is injective, implying that $i$ is an embedding. Nevertheless, this is equivalent to $\varnothing=\mathsf{Y}^{\geq 2}_{A(Y)^{\perp}}\cap \mathsf{Y}_{A(X)^{\perp}}^{\geq 2}\subset \mathsf{Y}_{A(X)^{\perp}}$, which is impossible as explained in \cite[Lemma 3.2]{Ferretti12}. See Remark \ref{rmk-non-inj} for an alternative proof of the non-injectivity of $i$. %since $[\overline{\mathsf{Y}}^{\geq 2}_{A(Y)^{\perp}}]=2[\mathsf{Y}_{A(X)^{\perp}}^{\geq 2}]\in \rH^4(\wt{\mathsf{Y}}_{A(X)^{\perp}}, \QQ)$ and 
%\[[\overline{\mathsf{Y}}^{\geq 2}_{A(Y)^{\perp}}].[\mathsf{Y}_{A(X)^{\perp}}^{\geq 2}]=2[\mathsf{Y}_{A(X)^{\perp}}^{\geq 2}]^2=384\neq 0\] as computed in \cite{o2010double}. See Remark \ref{rmk-non-inj} for an alternative proof of the non-injectivity of $i$.
\end{remark}

\subsubsection{GM threefolds in a fixed GM fourfold}

The above constructions are related to Fano varieties of conics in GM threefolds and fourfolds as follows. For a GM fourfold $X$ and its smooth hyperplane section $Y$, we denote by $F(X)$ and $F(Y)$ the Fano varieties of conics in $X$ and $Y$, respectively. For a general GM fourfold $X$, $F(X)$ is a smooth fivefold and there is a contraction
\[\pi_X\colon F(X)\twoheadrightarrow \wt{\mathsf{Y}}_{A(X)^{\perp}}\]
constructed in \cite{iliev2011fano}. Moreover, if $Y\subset X$ is a general hyperplane section, then $F(Y)$ is a smooth surface and there is a contraction
\[\pi_Y\colon F(Y)\twoheadrightarrow \wt{\mathsf{Y}}^{\geq 2}_{A(Y)^{\perp}}\]
constructed in \cite{Log12} and \cite{debarre2012period}. As explained in \cite[Section 5.1]{iliev2011fano} (see also \cite{Debarre2021quadrics}), we have a commutative diagram
% https://q.uiver.app/#q=WzAsNCxbMCwwLCJGKFkpIl0sWzAsMSwiRihYKSJdLFsxLDEsIlxcd3R7XFxtYXRoc2Z7WX19X3tBKFgpXntcXHBlcnB9fSJdLFsxLDAsIlxcd3R7XFxtYXRoc2Z7WX19XntcXGdlcSAyfV97QShZKV57XFxwZXJwfX0iXSxbMCwxLCIiLDAseyJzdHlsZSI6eyJ0YWlsIjp7Im5hbWUiOiJob29rIiwic2lkZSI6InRvcCJ9fX1dLFsxLDIsIlxccGlfWCJdLFszLDIsIiIsMix7InN0eWxlIjp7InRhaWwiOnsibmFtZSI6Imhvb2siLCJzaWRlIjoidG9wIn19fV0sWzAsMywiXFxwaV9ZIl1d
\[\begin{tikzcd}
	{F(Y)} & {\wt{\mathsf{Y}}^{\geq 2}_{A(Y)^{\perp}}} \\
	{F(X)} & {\wt{\mathsf{Y}}_{A(X)^{\perp}}}
	\arrow[hook, from=1-1, to=2-1]
	\arrow["{\pi_X}", from=2-1, to=2-2]
	\arrow["{i}", from=1-2, to=2-2]
	\arrow["{\pi_Y}", from=1-1, to=1-2]
\end{tikzcd}\]
where the left column is the natural inclusion and the right column is the  finite and unramified morphism in Lemma \ref{lem-compatible-double-cover}(2). Using this relation, the following result is obtained in \cite[Proposition 5.2]{iliev2011fano}.

\begin{proposition}\label{prop-lag-4fold}
Let $X$ be a general GM fourfold and $Y\hookrightarrow X$ be a general hyperplane section. Then $i\colon \wt{\mathsf{Y}}_{A(Y)^{\perp}}^{\geq 2}\to \wt{\mathsf{Y}}_{A(X)^{\perp}}$ is a connected immersed Lagrangian submanifold.
\end{proposition}

\begin{remark}\label{rmk-covering-4fold}
The locus swept out by the image of $\wt{\mathsf{Y}}_{A(Y)^{\perp}}^{\geq 2}$ in $\wt{\mathsf{Y}}_{A(X)^{\perp}}$ dominates $\wt{\mathsf{Y}}_{A(X)^{\perp}}$ as $Y\subset X$ varies in the linear series. This follows from the fact that a general conic on $X$ is contained in a general hyperplane section $Y\subset X$. For the case when $Y$ is a GM fivefold containing $X$, see Remark \ref{rmk-covering-5fold}.
\end{remark}

Therefore, combined with the modular descriptions in \cite{JLLZ2021gushelmukai} and \cite{GLZ2021conics}, we obtain:

\begin{lemma}\label{lem-EPW-atomic-very-general}
Let $X$ be a very general GM fourfold and $Y\hookrightarrow X$ be a general hyperplane section. Then $i\colon \wt{\mathsf{Y}}_{A(Y)^{\perp}}^{\geq 2}\to \wt{\mathsf{Y}}_{A(X)^{\perp}}$ is a connected immersed atomic Lagrangian submanifold. 
\end{lemma}

\begin{proof}
Since $X$ is very general, by \cite[Corollary B.9]{FGLZ}, there is a commutative diagram
\[\begin{tikzcd}
	{\wt{\mathsf{Y}}_{A(Y)^{\perp}}^{\geq 2}} & {M_{\sigma_Y}^Y(1,0)} \\
	{\wt{\mathsf{Y}}_{A(X)^{\perp}}} & {M_{\sigma_X}^X(1,0)}
	\arrow["{r}", from=1-2, to=2-2]
	\arrow["{i}"', from=1-1, to=2-1]
	\arrow["\cong"', from=2-1, to=2-2]
	\arrow["\cong"', from=1-1, to=1-2]
\end{tikzcd}\]
where $i$ is the natural morphism and $r$ is the induced morphism in Theorem \ref{thm-FGLZ-moduli}. Then the result follows from Theorem \ref{thm-pushforward-atomic}.
\end{proof}

%\zy{GM varieties from Lagrangian data, EPW, double EPW, and family version...Fano variety of conics...}

\begin{comment}

\begin{lemma}
Let $X$ be an ordinary GM fourfold and $Y\hookrightarrow X$ be a smooth hyperplane section. Then there are natural inclusions
\[\wt{\mathsf{Y}}_{A(Y)}^{\geq 2}\hookrightarrow \wt{\mathsf{Y}}_{A(X)},\quad \wt{\mathsf{Y}}_{A(Y)^{\perp}}^{\geq 2}\hookrightarrow \wt{\mathsf{Y}}_{A(X)^{\perp}}.\]
\end{lemma}

\begin{proof}
Let $(V_6, V_5, A(X))$ be the Lagrangian data of $X$. By \cite[Proposition 3.14]{debarre2015gushel}, the Lagrangian data of $Y$ is $(V_6, V_5, A(Y))$ such that $\dim( A(X)\cap A(Y))=9$. Hence we have $\dim( A(X)^{\perp}\cap A(Y)^{\perp})=9$ as well, which gives the inclusions 
\[\mathsf{Y}_{A(Y)}^{\geq 2}\hookrightarrow \mathsf{Y}_{A(X)},\quad \mathsf{Y}_{A(Y)^{\perp}}^{\geq 2}\hookrightarrow \mathsf{Y}_{A(X)^{\perp}}\]
by definition.
\end{proof}

\end{comment}

Now we aim to extend Lemma \ref{lem-EPW-atomic-very-general} to the general case. Since the results of \cite{GLZ2021conics} are established only for very general GM fourfolds, we employ a deformation argument through relative constructions of GM varieties and double EPW varieties.

\begin{theorem}\label{thm-double-epw}
Let $X$ be a general GM fourfold and $Y\hookrightarrow X$ be a general hyperplane section. Then $i\colon \wt{\mathsf{Y}}_{A(Y)^{\perp}}^{\geq 2}\to \wt{\mathsf{Y}}_{A(X)^{\perp}}$ is a connected immersed atomic Lagrangian submanifold with
\[\rH^1(\wt{\mathsf{Y}}_{A(Y)^{\perp}}^{\geq 2}, \CC)=\CC^{20}.\]
Moreover, we have $$(h)^{\perp}=\ker\big(\rH^2(\wt{\mathsf{Y}}_{A(X)^{\perp}}, \QQ)\to \rH^2(\wt{\mathsf{Y}}_{A(Y)^{\perp}}^{\geq 2}, \QQ)\big),$$
where $h$ is the pull-back of the ample class on $\mathsf{Y}_{A(X)^{\perp}}\subset \PP(V_6^{\vee})$.
\end{theorem}

\begin{proof}
As $X$ and $Y$ are general, we know that $F(X)$ and $F(Y)$ are both smooth by \cite[Theorem 3.2]{iliev2011fano} and \cite{Log12}. Moreover, from \cite[Theorem B.2]{debarre2015gushel}, we may assume that $$\mathsf{Y}^3_{A(Y)^{\perp}}=\mathsf{Y}^3_{A(X)^{\perp}}=\varnothing,$$ then $\wt{\mathsf{Y}}_{A(Y)^{\perp}}^{\geq 2}$ and $\wt{\mathsf{Y}}_{A(X)^{\perp}}$ are smooth by \cite[Theorem 5.2]{debarre2020double}. Consequently, according to Proposition~\ref{prop-lag-4fold}, $i\colon\wt{\mathsf{Y}}_{A(Y)^{\perp}}^{\geq 2}\to \wt{\mathsf{Y}}_{A(X)^{\perp}}$ is an immersed Lagrangian submanifold.

Recall that the moduli stacks of ordinary GM varieties of dimension $n=3$ or $4$ are smooth irreducible separated Deligne--Mumford stacks of finite type over $\CC$ (cf.~\cite[Corollary 5.12]{debarre:GM-moduli} and \cite[Proposition A.2]{kuznetsov2018derived}). Since the locus of Hodge-special GM fourfolds is the union of countably many divisors by \cite[Corollary 4.6]{debarre2015special}, using Lemma \ref{lem-curve-intersect-coutably-divisor}, we can construct a smooth family of ordinary GM fourfolds $\cX\to C$ over a smooth connected curve $C$ and a smooth divisor $\cY\subset \cX$ such that the fibers over a point $0\in C$ are $\cX_0\cong X$ and $\cY_0\cong Y$. Moreover, there exists an uncountable set $F\subset C$ such that $\cX_c$ is very general in the sense of Lemma \ref{lem-EPW-atomic-very-general} for each $c\in F$. Indeed, we start by considering a curve $C'$ along with a family $\cX'\to C'$ that satisfies the above assumptions. Next, we lift $C'$ to a curve $C$ in the relative Hilbert scheme of hyperplane sections of $\cX'\to C'$, obtaining the family $\cY\subset \cX\to C$. %Moreover, by further shrinking $C$ if necessary, we can assume that $C$ is affine. %Hence, there exists a smooth divisor $\cY\subset \cX$ such that $\cY_0\cong Y$, as in Lemma \ref{lem-lift-divisor-over-curve}. %By replacing $C$ with an affine open neighborhood of $0\in C$ again, we can assume that $\cY\to C$ is a smooth family of ordinary GM threefolds.

From the existence of relative Hilbert schemes, after replacing $C$ with an open neighborhood of $0\in C$, the Fano varieties of conics $F(\cY_c)$ and $F(\cX_c)$ are smooth for each $c\in C$. Moreover, we can assume that $\mathsf{Y}^3_{A(\cY_c)^{\perp}}=\mathsf{Y}^3_{A(\cX_c)^{\perp}}=\varnothing$ for each $c\in C$ as well.

Applying \cite[Theorem 3.7, Proposition 4.1]{debarre:GM-moduli} to $\cX\to C$ and $\cY\to C$, we obtain two families of Lagrangian data $(C, \mathscr{V}_6,\mathscr{V}_5, \mathscr{A}(\cX))$ and $(C, \mathscr{V}_6,\mathscr{V}_5, \mathscr{A}(\cY))$ over $C$\footnote{See \cite[Definition 3.15]{debarre:GM-moduli} for the precise definition.}, where $\mathscr{V}_6$, $\mathscr{V}_5$, $\mathscr{A}(\cX)$, and $\mathscr{A}(\cY)$ are all locally free sheaves on $C$ such that
\[(\mathscr{V}_6|_c,\mathscr{V}_5|_c, \mathscr{A}(\cX)|_c)=(V_6(\cX_c), V_5(\cX_c), A(\cX_c))~\text{and}~ (\mathscr{V}_6|_c,\mathscr{V}_5|_c, \mathscr{A}(\cY)|_c)=(V_6(\cY_c), V_5(\cY_c), A(\cY_c))\]
for each $c\in C$. After replacing $C$ with an open dense neighborhood of $0\in C$, we can assume that these bundles are all trivial. Therefore, as in the proofs of \cite[Theorem 5.2]{debarre2020double} and \cite[Proposition 5.27]{debarre2019gushel}, we get a family $p\colon \wt{\mathscr{Y}}_{\cX}\to C$ of double dual EPW sextics and a family $q\colon \wt{\mathscr{Y}}^{\geq 2}_{\cY}\to C$ of double dual EPW surfaces such that $p^{-1}(c)\cong \wt{\mathsf{Y}}_{A(\cX_c)^{\perp}}$ and $q^{-1}(c)\cong \wt{\mathsf{Y}}_{A(\cY_c)^{\perp}}^{\geq 2}$ for each $c\in C$ by replacing $C$ with an open dense neighborhood of $0\in C$. Note that by our assumptions of families $\cX$ and $\cY$, both $p$ and $q$ are smooth. We denote by $\mathscr{Y}_{\cX}\to C$, $\mathscr{Y}_{\cY}^{\geq 2}\to C$, and $\mathscr{Y}^{\geq 2}_{\cX}\to C$ the corresponding families of dual EPW sextics and dual EPW surfaces, respectively. Then $\mathscr{Y}_{\cX}, \mathscr{Y}^{\geq 2}_{\cX}$ and $\mathscr{Y}^{\geq 2}_{\cY}$ are smooth. Note that they are all closed subschemes of $\PP(\mathscr{V}^{\vee}_6)\cong \PP(V^{\vee}_6)\times C$ and $\mathscr{Y}^{\geq 2}_{\cY_c}\subset \mathscr{Y}_{\cX_c}$ for each $c\in C$, we see $\mathscr{Y}^{\geq 2}_{\cY}\subset \mathscr{Y}_{\cX}$. Therefore, by passing to double coverings, taking simultaneous normalization (cf.~\cite[Theorem 12]{kollar:normalization}), and using a relative version of Lemma \ref{lem-compatible-double-cover}, we have a commutative diagram
% https://q.uiver.app/#q=WzAsMyxbMCwwLCJcXHd0e1xcbWF0aHNjcntZfX1ee1xcZ2VxIDJ9Il0sWzEsMCwiXFx3dHtcXG1hdGhzY3J7WX19Il0sWzEsMSwiQyJdLFsxLDIsInAiXSxbMCwyLCJxIiwyXSxbMCwxLCJpIiwwLHsic3R5bGUiOnsidGFpbCI6eyJuYW1lIjoiaG9vayIsInNpZGUiOiJ0b3AifX19XV0=
\[\begin{tikzcd}
	{\wt{\mathscr{Y}}_{\cY}^{\geq 2}} & {\wt{\mathscr{Y}}_{\cX}} \\
	& C
	\arrow["p", from=1-2, to=2-2]
	\arrow["q"', from=1-1, to=2-2]
	\arrow["j", from=1-1, to=1-2]
\end{tikzcd}\]
such that $j_c\colon \wt{\mathsf{Y}}_{A(\cY_c)^{\perp}}^{\geq 2}\to \wt{\mathsf{Y}}_{A(\cX_c)^{\perp}}$ is the natural morphism in Lemma \ref{lem-compatible-double-cover}(2) for each $c\in C$. Hence, $j$ is also finite and unramified.

Up to now, we have shrunk $C$ near $0\in C$ finitely many times. As $F$ is an infinite set, there is a point $t\in C\cap F$ such that Lemma \ref{lem-EPW-atomic-very-general} can be applied to $\cY_{t}\hookrightarrow \cX_t$, i.e.~$j_t\colon \wt{\mathsf{Y}}_{A(\cY_t)^{\perp}}^{\geq 2}\to \wt{\mathsf{Y}}_{A(\cX_t)^{\perp}}$ is an immersed atomic Lagrangian submanifold. Therefore, using Proposition \ref{prop-immersed-deform}, we know that the Lagrangian submanifold $i_c\colon \wt{\mathsf{Y}}_{A(\cY_c)^{\perp}}^{\geq 2}\to \wt{\mathsf{Y}}_{A(\cX_c)^{\perp}}$ is atomic for each $c\in C$, and in particular, $j_0=i\colon \wt{\mathsf{Y}}_{A(Y)^{\perp}}^{\geq 2}\to \wt{\mathsf{Y}}_{A(X)^{\perp}}$ is atomic. Finally, the computation of the cohomology group is done in \cite[Proposition 2.5]{debarre:GM-jacobian} and \cite[Proposition 0.5]{Log12}, and the orthogonal complement of the kernel of the restriction map is computed in Corollary \ref{cor-rk1-H2-moduli} and Remark \ref{rmk-bm-divisor-4fold}.
\end{proof}

\begin{remark}
One can also use relative Hilbert schemes of conics and the constructions in \cite{iliev2011fano} to employ the deformation argument in Theorem \ref{thm-double-epw}.
\end{remark}

\begin{lemma}\label{lem-curve-intersect-coutably-divisor}
Let $X$ be an irreducible quasi-projective variety of $\dim X\geq 2$ and $\{D_i\}_{i\in \NN}$ be a set of irreducible Weil divisors in $X$. Then for any point $x\in X$, there exists an irreducible curve $C\subset X$ passing through $x$ such that $C$ intersects properly with each divisor in $\{D_i\}_{i\in \NN}$ and $C\cap D_i\neq C\cap D_j$ for $i\neq j$.
\end{lemma}

\begin{proof}
Let $p\colon X'\to X$ be the blow-up at $x$ and $E\subset X'$ be the exceptional divisor. We only need to show that there is an irreducible curve $C'\subset X'$ such that $C'$ intersects with each divisor in the set $\{E\}\cup\{p^{-1}(D_i)\}_{i\in \NN}$ and $C'\cap p^{-1}(D_i)\neq C'\cap p^{-1}(D_j)$ for $i\neq j$. By replacing $X$ with~$X'$, relabeling divisors, and using induction on $\dim X$, it suffices to show that for any irreducible quasi-projective variety $X$ with $\dim X\geq 2$ and a set $\{D_i\}_{i\in \NN}$ of irreducible Weil divisors in $X$, there exists an irreducible Weil divisor $H\subset X$ such that each divisor $D_i$ intersects $H$ properly and $H\cap D_i$ is irreducible with $\dim H\cap D_i\cap D_j\leq \dim X-3$ for $i\neq j$.

%there exists an irreducible Weil divisor $H\subset X'$ such that each divisor $p^{-1}(D_i)$ intersects $H$ properly and $H\cap p^{-1}(D_i)$ is irreducible with $\dim H\cap p^{-1}(D_i)\cap p^{-1}(D_j)\leq \dim X-3$ for $i\neq j$. %By adding irreducible components of each $D_i$ into $\{D_i\}_{i\in \NN}$ and relabeling, we can assume that $D_i$ is irreducible for each $i\in \NN$.

To this end, we first assume that $X$ is projective. Let $l\colon X\hookrightarrow \PP^n$ be an embedding, then each hyperplane $H\subset \PP^n$ intersects with each divisor in $\{D_i\}_{i\in \NN}$. By \cite[Theorem 3.3.1]{lazar:positivity-I}, for each $t\in \NN$, there exists an open dense subset $U_t\subset |\oh_{\PP^n}(1)|$ such that when $[H]\in U_t$, $H\cap X$ and $H\cap D_i$ are irreducible and $H$ intersects with $D_i$ and $D_i\cap D_j$ properly for each $i,j\leq t, i\neq j$. Thus, $\dim H\cap D_i\cap D_j\leq \dim X-3$ for $i,j\leq t, i\neq j$. We define $U:=\cap_{t\in \NN} U_t$. Note that~$U$ contains a closed point in $|\oh_{\PP^n}(1)|$ as it is the complement of the union of countably many closed subvarieties and we work over $\CC$. Then any $[H]\in U$ satisfies the required conditions. For the general case, let $\overline{X}$ be a projective variety containing $X$ as an open dense subset. Then $\overline{X}\setminus X$ has codimension at least one, and we can assume that there exist irreducible divisors $E_j\subset \overline{X}$ for $1\leq j\leq m$ such that $\overline{X}\setminus X\subset \cup_{1\leq j\leq m} E_j$. Thus, applying the above result to $\overline{X}$ and $\{E_j\}_{1\leq j\leq m}\cup \{D_i\}_{i\in \NN}$, we get an irreducible curve $\overline{C}\subset \overline{X}$. Then $C:=\overline{C}\cap X$ satisfies all desired properties in the statement.
\end{proof}

\begin{comment}

\begin{lemma}\label{lem-lift-divisor-over-curve}
Let $\pi\colon \cX\to C$ be a flat proper morphism between algebraic varieties such that $C$ is a connected affine smooth curve. Assume that there is a line bundle $\oh_{\cX}(H)$ on $\cX$ such that is very ample when restricted to each fiber. Then for any point $0\in C$ and effective divisor $Y\in |\oh_{\cX_0}(H)|$, there exists a divisor $\cY\subset \cX$ such that $\cY_0=Y$.
\end{lemma}

\begin{proof}

We have an exact sequence
\[0\to \oh_{\cX}(H-\cX_0)\to \oh_{\cX}(H)\to \oh_{\cX_0}(H)\to 0.\]
As $\oh_{\cX}(\cX_0)=\pi^*\oh_C(0)$, we see $\pi_*\oh_{\cX}(H-\cX_0)\cong \oh_C(-0)\otimes \pi_*\oh_{\cX}(H)$ by the projection formula. Moreover, by our assumption, we see $R^q\pi_*(\oh_{\cX}(H))=0$ for $q\neq 0$. Therefore, the Leray spectral sequence (cf.~\cite[Equation (3.3)]{huybrechts2006fourier}) $$E^{p,q}_2:=\rH^p\big(C, R^q\pi_*(\oh_{\cX}(H-\cX_0))\big)\Rightarrow \rH^{p+q}(\cX, \oh_{\cX}(H-\cX_0))$$ satisfies $E^{p,q}_2=0$ when $(p,q)\neq (0,0)$ since $C$ is affine. Thus, we get $\rH^1(\cX, \oh_{\cX}(H-\cX_0))=0$. In other words, the restriction map
\[\rH^0(\cX, \oh_{\cX}(H))\to \rH^0(\cX, \oh_{\cX_0}(H))\]
is surjective and the result follows.
\end{proof}
\end{comment}

\subsubsection{GM fivefolds containing a fixed GM fourfold}

Now, we consider a general GM fourfold $X$ and a general GM fivefold $Y$ containing $X$. By Lemma \ref{lem-compatible-double-cover}, we also have a commutative diagram:
% https://q.uiver.app/#q=WzAsNCxbMCwxLCJcXHd0e1xcbWF0aHNme1l9fV97QShYKV57XFxwZXJwfX0iXSxbMCwwLCJcXHd0e1xcbWF0aHNme1l9fV57XFxnZXEgMn1fe0EoWSlee1xccGVycH19Il0sWzEsMSwiXFxtYXRoc2Z7WX1fe0EoWClee1xccGVycH19Il0sWzEsMCwiXFxtYXRoc2Z7WX1ee1xcZ2VxIDJ9X3tBKFkpXntcXHBlcnB9fSJdLFsxLDAsIiIsMix7InN0eWxlIjp7InRhaWwiOnsibmFtZSI6Imhvb2siLCJzaWRlIjoidG9wIn19fV0sWzMsMiwiIiwwLHsic3R5bGUiOnsidGFpbCI6eyJuYW1lIjoiaG9vayIsInNpZGUiOiJ0b3AifX19XSxbMSwzXSxbMCwyXV0=
\[\begin{tikzcd}
	{\wt{\mathsf{Y}}^{\geq 2}_{A(Y)^{\perp}}} & {\mathsf{Y}^{\geq 2}_{A(Y)^{\perp}}} \\
	{\wt{\mathsf{Y}}_{A(X)^{\perp}}} & {\mathsf{Y}_{A(X)^{\perp}}}
	\arrow[from=1-1, to=1-2]
	\arrow["{i}", from=1-1, to=2-1]
	\arrow[hook, from=1-2, to=2-2]
	\arrow[from=2-1, to=2-2]
\end{tikzcd}\]

According to \cite[Section 5.2]{iliev2011fano} and \cite{Debarre2021quadrics}, there is a $\PP^1$-fibration $\pi_Y\colon F_{qs}(Y)\twoheadrightarrow \wt{\mathsf{Y}}^{\geq 2}_{A(Y)^{\perp}}$ and we have a commutative diagram
\[\begin{tikzcd}
	{F_{qs}(Y)} & {\wt{\mathsf{Y}}^{\geq 2}_{A(Y)^{\perp}}} \\
	{F(X)} & {\wt{\mathsf{Y}}_{A(X)^{\perp}}}
	\arrow[hook, from=1-1, to=2-1]
	\arrow["{\pi_X}", from=2-1, to=2-2]
	\arrow["{i}", from=1-2, to=2-2]
	\arrow["{\pi_Y}", from=1-1, to=1-2]
\end{tikzcd}\]
where $F_{qs}(Y)$ is a connected component of the Hilbert scheme of quadric surfaces in $Y$ such that $S\cap X$ is a conic for each $[S]\in F_{qs}(Y)$\footnote{A quadric surface $S\subset Y$ is in $F_{qs}(Y)$ if and only if $S\subset \Gr(2, V_4)$ for a $4$-dimensional subspace $V_4\subset V_5$. It is called a $\tau$-quadric in \cite[Section 7.3]{debarre2015special}.}. Using this relation, the following result is obtained in \cite[Proposition 5.6]{iliev2011fano}, which is an analog of Proposition \ref{prop-lag-4fold}.

\begin{proposition}\label{prop-lag-5fold}
Let $X$ be a general GM fourfold and $X\hookrightarrow Y$ be an embedding such that~$Y$ is a general GM fivefold. Then $i\colon\wt{\mathsf{Y}}_{A(Y)^{\perp}}^{\geq 2}\to \wt{\mathsf{Y}}_{A(X)^{\perp}}$ is a connected immersed Lagrangian submanifold.
\end{proposition}

Therefore, we can prove:

\begin{lemma}\label{lem-5fold-epw}
Let $X$ be a very general GM fourfold and $X\hookrightarrow Y$ be an embedding such that $Y$ is a general GM fivefold. Then $i\colon \wt{\mathsf{Y}}_{A(Y)^{\perp}}^{\geq 2}\to \wt{\mathsf{Y}}_{A(X)^{\perp}}$ is a connected immersed atomic Lagrangian submanifold.
\end{lemma}

\begin{proof}
By \cite{GLZ2021conics}, we have an isomorphism $M^X_{\sigma_X}(1,0)\cong \wt{\mathsf{Y}}_{A(X)^{\perp}}$. Therefore, Proposition \ref{prop-H2-moduli-space} implies that
\[i^*\colon \rH^2(\wt{\mathsf{Y}}_{A(X)^{\perp}}, \QQ)\to \rH^2(\wt{\mathsf{Y}}_{A(Y)^{\perp}}^{\geq 2}, \QQ)\]
is of rank one. And from \cite[Theorem 5.2(2)]{debarre2020double}, we know that
\[c_1(\wt{\mathsf{Y}}_{A(Y)^{\perp}}^{\geq 2})=-3h\in \rH^2(\wt{\mathsf{Y}}_{A(Y)^{\perp}}^{\geq 2}, \QQ),\]
where $h$ is the canonical polarization on $\wt{\mathsf{Y}}_{A(X)^{\perp}}$. Then, by Theorem \ref{atomic-criterion-thm} and Proposition \ref{prop-lag-5fold}, we conclude that  $\wt{\mathsf{Y}}_{A(Y)^{\perp}}^{\geq 2}$ is an immersed atomic Lagrangian in $\wt{\mathsf{Y}}_{A(X)^{\perp}}$.
\end{proof}

\begin{remark}\label{rmk-5fold-epw}
Using similar computations as in \cite{JLLZ2021gushelmukai,GLZ2021conics}, it is not hard to prove that the projection functor induces a $\PP^1$-fibration $F_{qs}(Y)\twoheadrightarrow M^Y_{\sigma_Y}(1,0)$ for a GM fivefold $Y$. Then we have $M^Y_{\sigma_Y}(1,0)\cong \wt{\mathsf{Y}}_{A(Y)^{\perp}}^{\geq 2}$. Thus, Lemma \ref{lem-5fold-epw} also follows from Theorem \ref{thm-push-atomic-5fold}.
\end{remark}

Finally, employing the same proof as in Theorem \ref{thm-double-epw}, but replacing Lemma \ref{lem-EPW-atomic-very-general} with Lemma~\ref{lem-5fold-epw}, we extend Lemma \ref{lem-5fold-epw} to the general case:

\begin{theorem}\label{thm-double-epw-5fold}
Let $X$ be a general GM fourfold and $X\hookrightarrow Y$ be an embedding such that $Y$ is a general GM fivefold. Then $i\colon \wt{\mathsf{Y}}_{A(Y)^{\perp}}^{\geq 2}\to \wt{\mathsf{Y}}_{A(X)^{\perp}}$ is a connected immersed atomic Lagrangian submanifold with
\[\rH^1(\wt{\mathsf{Y}}_{A(Y)^{\perp}}^{\geq 2}, \CC)=\CC^{20}.\]
Moreover, we have $$(h)^{\perp}=\ker\big(\rH^2(\wt{\mathsf{Y}}_{A(X)^{\perp}}, \QQ)\to \rH^2(\wt{\mathsf{Y}}_{A(Y)^{\perp}}^{\geq 2}, \QQ)\big),$$
where $h$ is the pull-back of the ample class on $\mathsf{Y}_{A(X)^{\perp}}\subset \PP(V_6^{\vee})$.
\end{theorem}

\begin{remark}\label{rmk-covering-5fold}
As in Remark \ref{rmk-covering-4fold}, the locus swept out by the image of  $\wt{\mathsf{Y}}_{A(Y)^{\perp}}^{\geq 2}$ in $ \wt{\mathsf{Y}}_{A(X)^{\perp}}$ dominates $\wt{\mathsf{Y}}_{A(X)^{\perp}}$ as $Y\supset X$ varies. This follows from the fact that for a general conic $C$ on~$X$, one can find a general GM fivefold $Y$ containing $X$ and a quadric surface $S\subset Y$ such that $X\cap S=C$.
\end{remark}

\begin{remark}\label{rmk-non-inj}
We can also demonstrate that $i\colon \wt{\mathsf{Y}}_{A(Y)^{\perp}}^{\geq 2}\to \wt{\mathsf{Y}}_{A(X)^{\perp}}$ is not injective as follows. If $i$ is an embedding, then by Theorem \ref{thm-double-epw} and \ref{thm-double-epw-5fold}, $\wt{\mathsf{Y}}_{A(Y)^{\perp}}^{\geq 2}\subset \wt{\mathsf{Y}}_{A(X)^{\perp}}$ is an atomic embedded Lagrangian. However, by \cite[Equation (7.11)]{markman:rank-1-obstruction}, we have $$\sqrt{\frac{\chi(\wt{\mathsf{Y}}_{A(Y)^{\perp}}^{\geq 2})}{3}}\in \mathbb{Z},$$ which contradicts $\chi(\wt{\mathsf{Y}}_{A(Y)^{\perp}}^{\geq 2})=2\chi(\mathsf{Y}_{A(Y)^{\perp}}^{\geq 2})=384$ by \cite[Proposition 1.10]{Ferretti12}.
\end{remark}

\subsection{Double EPW cubes}\label{subsec-cube}

Similar to double EPW varieties, there is another construction of hyper-K\"ahler manifolds using the projective geometry of Grassmannians.

Given a GM variety $X$ of dimension $n\geq 3$ with the Lagrangian data $(V_6, V_5, A(X))$, analogous to EPW varieties discussed in Section \ref{subsec-double-epw}, one can define a subscheme of $\Gr(3, V_6)$ for each $k$ as
\[\mathsf{Z}^{\geq k}_{A(X)}:=\{[U]\in\Gr(3,V_6) \mid \mathrm{dim}(A(X)\cap(V_6\wedge\bigwedge^2U))\geq k\}\subset \Gr(3, V_6).\]
In \cite{IKKR19}, it is shown that there exists a double covering
\[\wt{\mathsf{Z}}^{\geq 2}_{A(X)} \to \mathsf{Z}^{\geq 2}_{A(X)}\]
such that $\wt{\mathsf{Z}}^{\geq 2}_{A(X)}$ is a hyper-K\"ahler sixfold of $\mathrm{K3^{[3]}}$-type when $X$ is general.

More generally, for each $0\leq k\leq 3$, a double covering 
\[\wt{\mathsf{Z}}^{\geq k}_{A(X)} \to \mathsf{Z}^{\geq k}_{A(X)}\]
is constructed in \cite[Theorem 5.7]{debarre2020double}, such that $\wt{\mathsf{Z}}^{\geq 3}_{A(X)}$ is a smooth connected projective threefold when $X$ is general.

As in Lemma \ref{lem-compatible-double-cover}, for a smooth hyperplane section $Y\subset X$, we have a commutative diagram
% https://q.uiver.app/#q=WzAsNCxbMSwxLCJcXG1hdGhzZntafV57XFxnZXEgMn1fe0EoWCl9Il0sWzAsMSwiXFxtYXRoc2Z7Wn1ee1xcZ2VxIDN9X3tBKFkpfSJdLFswLDAsIlxcd3R7XFxtYXRoc2Z7Wn19XntcXGdlcSAzfV97QShZKX0iXSxbMSwwLCJcXHd0e1xcbWF0aHNme1p9fV57XFxnZXEgMn1fe0EoWCl9Il0sWzEsMCwiIiwwLHsic3R5bGUiOnsidGFpbCI6eyJuYW1lIjoiaG9vayIsInNpZGUiOiJ0b3AifX19XSxbMiwzLCIiLDAseyJzdHlsZSI6eyJ0YWlsIjp7Im5hbWUiOiJob29rIiwic2lkZSI6InRvcCJ9fX1dLFszLDBdLFsyLDFdXQ==
\[\begin{tikzcd}
	{\wt{\mathsf{Z}}^{\geq 3}_{A(Y)}} & {\wt{\mathsf{Z}}^{\geq 2}_{A(X)}} \\
	{\mathsf{Z}^{\geq 3}_{A(Y)}} & {\mathsf{Z}^{\geq 2}_{A(X)}}
	\arrow[from=1-1, to=1-2]
	\arrow[from=1-1, to=2-1]
	\arrow[from=1-2, to=2-2]
	\arrow[hook, from=2-1, to=2-2]
\end{tikzcd}\]
Although in this case, we do not know whether $\wt{\mathsf{Z}}^{\geq 3}_{A(Y)}\to \wt{\mathsf{Z}}^{\geq 2}_{A(X)}$ is Lagrangian, we can still construct atomic Lagrangians in $\wt{\mathsf{Z}}^{\geq 2}_{A(X)}$ as follows.

According to \cite{kapustka2022epw}, for a very general GM fourfold $X$, there is an isomorphism
\[\wt{\mathsf{Z}}^{\geq 2}_{A(X)}\cong M^X_{\sigma_X}(1,-1).\]
Moreover, in \cite[Section 8]{FGLZ}, the authors construct a rational map
\[\pi_X\colon \Hilb^{3t+1}_X\dashrightarrow M^X_{\sigma_X}(1,-1)\cong \wt{\mathsf{Z}}^{\geq 2}_{A(X)}\]
induced by the projection functor, where $\PP^1$ is the general fiber. By \cite{FGLZ}, for a general hyperplane section $Y\subset X$, we have a commutative diagram
% https://q.uiver.app/#q=WzAsNCxbMCwxLCJcXEhpbGJeezN0KzF9X1giXSxbMCwwLCJcXEhpbGJeezN0KzF9X1kiXSxbMSwxLCJNX3tcXHNpZ21hX1h9XlgoMSwtMSkiXSxbMSwwLCJNX3tcXHNpZ21hX1l9XlkoMSwtMSkiXSxbMSwwLCIiLDAseyJzdHlsZSI6eyJ0YWlsIjp7Im5hbWUiOiJob29rIiwic2lkZSI6InRvcCJ9fX1dLFswLDIsIlxccGlfWCIsMCx7InN0eWxlIjp7ImJvZHkiOnsibmFtZSI6ImRhc2hlZCJ9fX1dLFsxLDMsIlxccGlfWSJdLFszLDIsInIiXV0=
\[\begin{tikzcd}
	{\Hilb^{3t+1}_Y} & {M_{\sigma_Y}^Y(1,-1)} \\
	{\Hilb^{3t+1}_X} & {M_{\sigma_X}^X(1,-1)}
	\arrow["{\pi_Y}", from=1-1, to=1-2]
	\arrow[hook, from=1-1, to=2-1]
	\arrow["r", from=1-2, to=2-2]
	\arrow["{\pi_X}", dashed, from=2-1, to=2-2]
\end{tikzcd}\]
where $\pi_Y$ is induced by the projection functor, and $r$ is the morphism in Theorem~\ref{thm-FGLZ-moduli}. Moreover, $\pi_Y$ maps birationally onto a connected component of $M_{\sigma_Y}^Y(1,-1)$ by \cite[Theorem C.13]{FGLZ}, and $r$ is an embedding when restricted to this component by \cite[Lemma 8.20]{FGLZ}.

In particular, denoting by $L_Y$ the image of $r\circ \pi_Y$, it follows that $L_X$ is isomorphic to a connected component of $M_{\sigma_Y}^Y(1,-1)$ and  $L_Y\subset M^X_{\sigma_X}(1,-1)\cong \wt{\mathsf{Z}}^{\geq 2}_{A(X)}$ is a Lagrangian subvariety. Moreover, $L_Y$ is smooth by \cite[Theorem C.13, Corollary C.14]{FGLZ} and $\Hilb^{3t+1}_Y\twoheadrightarrow L_Y$ is a divisorial contraction. Thus, Theorem \ref{thm-pushforward-atomic} can be applied in this case,  yielding:

\begin{corollary}\label{cor-double-epw-cube}
Let $X$ be a very general GM fourfold and $Y\subset X$ be a general hyperplane section. Then 
\[L_Y\subset \wt{\mathsf{Z}}^{\geq 2}_{A(X)}\]
is a connected atomic Lagrangian submanifold.
\end{corollary}

We expect that $L_Y\hookrightarrow \wt{\mathsf{Z}}^{\geq 2}_{A(X)}$ coincides with $\wt{\mathsf{Z}}^{\geq 3}_{A(Y)}\to \wt{\mathsf{Z}}^{\geq 2}_{A(X)}$. This assertion can be verified once~$\wt{\mathsf{Z}}^{\geq 3}_{A(Y)}$ also has a description in terms of $M^Y_{\sigma_Y}(1,-1)$ and $\Hilb^{3t+1}_Y$.

\section{Further discussions and examples}\label{sec-discussion}

In this section, we further explore the example in Section \ref{subsec-double-epw}. Then we provide two examples of atomic sheaves supported on non-atomic Lagrangians, constructed from cubic hypersurfaces.

\subsection{The relation to the relative Jacobian fibration of GM fivefolds}\label{subsec-jacobian}
%\zy{TBC...}

We fix $X$ to be a general GM fourfold. Let $\cB_X$ be the space that parametrizes general GM fivefolds containing $X$, which is an open dense subscheme of $|\cI_{X/\Gr(2,V_5)}(2)|\cong \PP^{10}$. Then Theorem \ref{thm-double-epw-5fold} gives a smooth projective family of immersed atomic Lagrangians
% https://q.uiver.app/#q=WzAsMyxbMCwwLCJcXHd0e1xcbWF0aHNjcntZfX1ee1xcZ2VxIDJ9Il0sWzAsMSwiXFxjQl9YIl0sWzEsMCwiXFx3dHtcXG1hdGhzZntZfX1fe0EoWClee1xccGVycH19Il0sWzAsMSwiaCIsMl0sWzAsMl1d
\[\begin{tikzcd}
	{\wt{\mathscr{Y}}^{\geq 2}} & {\wt{\mathsf{Y}}_{A(X)^{\perp}}} \\
	{\cB_X}
	\arrow[from=1-1, to=1-2]
	\arrow["\varphi"', from=1-1, to=2-1]
\end{tikzcd}\]
which also defines a morphism $s\colon\cB_X\to \cH$, where $\cH$ is the component of the Hilbert scheme of Lagrangian subvarieties in $\wt{\mathsf{Y}}_{A(X)^{\perp}}$ containing the image $\overline{\mathsf{Y}}^{\geq 2}_{A(Y)^{\perp}}$ of $\wt{\mathsf{Y}}^{\geq 2}_{A(Y)^{\perp}}$ for $[Y]\in \cB_X$.

\begin{lemma}\label{lem-open-embed}
The morphism $s\colon\cB_X\to \cH$ is an open embedding, hence $\cH$ is smooth along $s(\cB_X)$.
\end{lemma}

\begin{proof}
First, we show that $s$ is injective, which is equivalent to showing that for any two different points $[Y],[Y']\in \cB_X$, we have $\overline{\mathsf{Y}}^{\geq 2}_{A(Y)^{\perp}}\neq \overline{\mathsf{Y}}^{\geq 2}_{A(Y')^{\perp}}$. Indeed, if $\overline{\mathsf{Y}}^{\geq 2}_{A(Y)^{\perp}}= \overline{\mathsf{Y}}^{\geq 2}_{A(Y')^{\perp}}$, then we have $\mathsf{Y}^{\geq 2}_{A(Y)^{\perp}}= \mathsf{Y}^{\geq 2}_{A(Y')^{\perp}}\subset \mathsf{Y}_{A(X)^{\perp}}$. However, by \cite[Proposition 2.5]{debarre:GM-jacobian}, we have a canonical identification $A(Y)=A(Y')$, which implies that $Y=Y'$ by \cite[Propoition 3.14]{debarre2015gushel}.

Consequently, we have $\dim \cH\geq \dim \cB_X=10$. For any point $[Y]\in \cB_X$, we are going to compare $\rH^0(\overline{\mathsf{Y}}^{\geq 2}_{A(Y)^{\perp}}, \cN_1)$ and $\rH^0(\wt{\mathsf{Y}}^{\geq 2}_{A(Y)^{\perp}}, \cN_2)$, where
$$\cN_1:=\cN_{\overline{\mathsf{Y}}^{\geq 2}_{A(Y)^{\perp}}/\wt{\mathsf{Y}}_{A(X)^{\perp}}},\quad\cN_2:=\cN_{\wt{\mathsf{Y}}^{\geq 2}_{A(Y)^{\perp}}/\wt{\mathsf{Y}}_{A(X)^{\perp}}}.$$
Note that $\cN_1$ and $\cN_2$ are reflexive sheaves by \cite[Corollary 1.2]{har80}. Let $U\subset \overline{\mathsf{Y}}^{\geq 2}_{A(Y)^{\perp}}$ be the smooth locus and $V\subset \wt{\mathsf{Y}}^{\geq 2}_{A(Y)^{\perp}}$ be the pullback of $U$. Hence, $V\cong U$ and $\cN_1|_V\cong \cN_2|_U$. From the fact that $\dim (\wt{\mathsf{Y}}^{\geq 2}_{A(Y)^{\perp}}\setminus V)=0$ by Lemma \ref{lem-compatible-double-cover} and $\cN_2$ is reflexive, we have
\[\rH^0(\wt{\mathsf{Y}}^{\geq 2}_{A(Y)^{\perp}}, \cN_2)=\rH^0(V, \cN_2)=\CC^{10}.\]
Since $\cN_1$ is torsion-free, there is an embedding $\rH^0(\overline{\mathsf{Y}}^{\geq 2}_{A(Y)^{\perp}}, \cN_1)\hookrightarrow \rH^0(U, \cN_1)$. As $\cN_1|_V\cong \cN_2|_U$, we get an inclusion
\[\rH^0(\overline{\mathsf{Y}}^{\geq 2}_{A(Y)^{\perp}}, \cN_1)\hookrightarrow \rH^0(V, \cN_2)=\rH^0(\wt{\mathsf{Y}}^{\geq 2}_{A(Y)^{\perp}}, \cN_2)=\CC^{10}.\]
By $10\leq \dim \cH\leq \dim \rH^0(\overline{\mathsf{Y}}^{\geq 2}_{A(Y)^{\perp}}, \cN_1)$, we finally obtain $\dim \cH=10$ and
\[\rH^0(\overline{\mathsf{Y}}^{\geq 2}_{A(Y)^{\perp}}, \cN_1)=\rH^0(\wt{\mathsf{Y}}^{\geq 2}_{A(Y)^{\perp}}, \cN_2).\]
This implies that $s$ is \'etale. Then $s$ is an open embedding by its injectivity.
\end{proof}

\begin{proof}[{Proof of Corollary \ref{cor-epw-bundle}}]
This follows from Theorem \ref{thm-double-epw-5fold} and the same argument as in Theorem \ref{thm-bundle}. Note that the dimension of the deformation space is $20$ since the morphism $i\colon \wt{\mathsf{Y}}_{A(Y)^{\perp}}^{\geq 2}\to \wt{\mathsf{Y}}_{A(X)^{\perp}}$ is a closed embedding away from finitely many points by Lemma \ref{lem-compatible-double-cover}(2) and there exists a nice family of immersed Lagrangians of $\wt{\mathsf{Y}}_{A(X)^{\perp}}$ as required in Lemma \ref{lem-immerse-ext1}. Indeed, the family $L_V\to V$ in Lemma \ref{lem-immerse-ext1} is constructed by a relative version of Theorem \ref{thm-double-epw-5fold} as in the proof of Theorem \ref{thm-double-epw}, which satisfies (1) and (2) in Lemma \ref{lem-immerse-ext1} by Lemma \ref{lem-compatible-double-cover}. The assumption (3) in Lemma \ref{lem-immerse-ext1} follows from Lemma \ref{lem-open-embed}. Hence, we can apply Lemma \ref{lem-immerse-ext1} in this case to determine the dimension of the deformation space.
\end{proof}

Since $\cH$ is smooth along $s(\cB_X)$, according to \cite[Theorem 8.1]{markman:spectral-curve}, $\Pic^0(\wt{\mathscr{Y}}^{\geq 2}/\cB_X)$ is equipped with a holomorphic symplectic two-form and the relative Picard fibration $$\Pic^0(\wt{\mathscr{Y}}^{\geq 2}/\cB_X)\to \cB_X$$ is a Lagrangian fibration. In fact, while \cite[Theorem 8.1]{markman:spectral-curve} is stated only for embedded Lagrangians, it can also be applied in our case by Lemma \ref{lem-open-embed}. This is because the proof of \cite[Theorem 8.1]{markman:spectral-curve} involves a local study, as explained in \cite[Section 3.1]{IM08cubic}. 

We denote by $\cM$ the component of stable sheaves on $\wt{\mathsf{Y}}_{A(X)^{\perp}}$ containing $i_*\cL$, where $[Y]\in \cB_X$ and $\cL\in \Pic^0(\wt{\mathsf{Y}}^{\geq 2}_{A(Y)^{\perp}})$. Note that $i_*\cL$ is stable because it is a rank one torsion-free sheaf supported on the image of the normalization map $i$. Then $\cM$ has an open dense subscheme $\cM^0$ isomorphic to $\Pic^0(\wt{\mathscr{Y}}^{\geq 2}/\cB_X)$ such that the support morphism $\cM^0\to \cB_X$ is isomorphic to the relative Picard fibration via the identification of $\cB_X$ as an open subscheme of $\cH$ by Lemma \ref{lem-open-embed}. Thus, $\cM^0$ also carries a holomorphic symplectic structure. This strategy is explained in \cite[Section 7.4]{markman:rank-1-obstruction}.

On the other hand, according to \cite[Theorem 1.1]{debarre:GM-jacobian}, there is a canonical isomorphism of principally polarized abelian varieties
\[\mathrm{J}(Y)\cong \mathrm{Alb}(\wt{\mathsf{Y}}^{\geq 2}_{A(Y)^{\perp}})\]
for any $[Y]\in \cB_X$. Thus, $\mathrm{J}(Y)\cong \mathrm{Alb}(\wt{\mathsf{Y}}^{\geq 2}_{A(Y)^{\perp}})\cong \Pic^0(\wt{\mathsf{Y}}^{\geq 2}_{A(Y)^{\perp}})$ and the fiber of $\cM^0\to \cB_X$ over $[Y]\in \cB_X$ is canonically isomorphic to $\mathrm{J}(Y)$. This situation is very similar to the case of cubic fourfolds (cf.~\cite[Example 8.22]{markman:spectral-curve} and \cite{LSV,IM08cubic}), and we hope that the techniques in \cite{LSV} can be applied to the Gushel--Mukai case to construct a suitable compactification $\overline{\cM}$ of $\cM^0\to \cB_X$.

In practice, constructing a smooth hyper-K\"ahler compactification $\overline{\cM}$ can be challenging. However, it is natural to ask the following question: 

\begin{question}
Let $X$ be a general GM fourfold and $\cB_X\subset \PP^{10}$ be the space of general GM fivefolds containing $X$. Does the relative Picard fibration
\[\cM^0\to \cB_X\]
have a compactification $\overline{\cM}\to \PP^{10}$ such that $\overline{\cM}$ is a projective irreducible holomorphic symplectic variety with at worst $\QQ$-factorial terminal singularities?
\end{question}

A positive answer to this question will lead to a projective irreducible holomorphic symplectic variety of a new deformation type.

\subsection{Atomic sheaves from non-atomic Lagrangians}\label{sec-llsvs}

In this subsection, we study two families of hyper-K\"ahler manifolds associated with cubic fourfolds: LLSvS eightfolds \cite{LLSvS17} and Fano varieties of lines \cite{beauville:fano-variety-cubic-4fold}. %In Proposition \ref{prop-llsvs-not-atomic}, we show that although LLSvS eightfolds and their natural Lagrangian submanifolds have descriptions in terms of Bridgeland moduli spaces, these Lagrangians are not atomic. In Proposition~\ref{prop-fanoplane-not-atomic}, we establish that the Lagrangian surfaces constructed in \cite{IM08cubic} are also not atomic in the Fano variety of lines. 

For a cubic fourfold $X$, its semiorthogonal decomposition is given by $$\Db(X)=\langle\Ku(X),\oh_X,\oh_X(H),\oh_X(2H)\rangle=\langle\oh_X(-H),\Ku(X),\oh_X,\oh_X(H)\rangle,$$
where $H$ is the hyperplane class of $X$, satisfying $S_{\Ku(X)}=[2]$. In the numerical Grothendieck group $\Knum(\Ku(X))$, there is a rank two lattice generated by $\Lambda_1$ and $\Lambda_2$ with $$\ch(\Lambda_1)=3-H-\frac{1}{2}H^2+\frac{1}{6}H^3+\frac{1}{8}H^4,\quad \ch(\Lambda_2)=-3+2H-\frac{1}{3}H^3.$$ 
When $X$ is non-Hodge-special, we have $\Knum(\Ku(X))=\langle\Lambda_1, \Lambda_2 \rangle$.

For a cubic threefold $Y$, we have a semiorthogonal decomposition
\[\Db(Y)=\langle \Ku(Y), \oh_Y, \oh_Y(H)\rangle.\]
In this case, $S^3_{\Ku(Y)}=[5]$. Moreover, $\Knum(\Ku(Y))$ is a rank two lattice generated by $\lambda_1$ and $\lambda_2$ with 
$$\ch(\lambda_1)=2-H-\frac{1}{6}H^2+\frac{1}{6}H^3 ,\quad \ch(\lambda_2)=-1+H-\frac{1}{6}H^2-\frac{1}{6}H^3.$$

The stability conditions on $\Ku(X)$ and $\Ku(Y)$ are constructed in \cite{bayer2017stability}. We denote by $M^Y_{\sigma_Y}(a,b)$ (resp.~$M^X_{\sigma_X}(a,b)$) the moduli space of S-equivalence classes of $\sigma_Y$-semistable (resp.~$\sigma_X$-semistable) objects in $\Ku(Y)$ (resp.~$\Ku(X)$) with class $a\lambda_1+b\lambda_2$ (resp.~$a\Lambda_1+b\Lambda_2$).

We will utilize Theorem \ref{thm-relative-stability-threefold}, Theorem \ref{thm-relative-stability-fourfold}, Theorem \ref{thm-fourfold-moduli}, and Corollary \ref{cor-rk1-H2-moduli} for cubic fourfolds and threefolds, as the same argument applies in this context (cf.~\cite[Section 29]{BLMNPS21}).

\subsubsection{LLSvS eightfolds}

Let $X$ be a cubic fourfold not containing a plane and $j\colon Y\hookrightarrow X$ be a smooth hyperplane section. Let $M_3(X)$ be the Hilbert scheme of generalized twisted cubics on~$X$, which is a smooth projective tenfold by \cite[Theorem 4.7]{LLSvS17}.

In \cite{LLSvS17}, the authors construct a birational $\mathbb{P}^2$-fibration
$$u \colon M_3(X)\rightarrow \mathsf{Z}_X$$
such that $\mathsf{Z}_X$ is a hyper-K\"ahler eightfold, referred to as  \emph{the LLSvS eightfold associated with $X$}. If we define $\mathsf{Z}_Y:=u(M_3(Y))$, where $M_3(Y)$ is the Hilbert scheme of twisted cubics on $Y$, then~$\mathsf{Z}_Y$ is a connected smooth Lagrangian submanifold of $\mathsf{Z}_X$ according to \cite{shinder2017geometry}.

On the other hand, we have the moduli interpretation $\mathsf{Z}_X\cong M^X_{\sigma_X}(2,1)$ (see \cite{li2018twisted}) and $\mathsf{Z}_Y\cong M^Y_{\sigma_Y}(2,1)$ (see \cite{bayer2020desingularization} or \cite[Appendix B.2]{FGLZ}). In particular, if we denote by $i\colon \mathsf{Z}_Y\hookrightarrow \mathsf{Z}_X$ the Lagrangian embedding constructed in \cite{shinder2017geometry}, the moduli construction is shown to be compatible with $i$ by \cite[Corollary B.7]{FGLZ}, as the following commutative diagram

\[\begin{tikzcd}\label{diagram}
\mathsf{Z}_Y \arrow[d, "i", hook] \arrow[r, "\cong"] & { M^Y_{\sigma_Y}(2,1)} \arrow[d, "r", hook] \\
\mathsf{Z}_X \arrow[r, "\cong"]                      & { M^X_{\sigma_X}(2,1)}                     
\end{tikzcd}\]
where the morphism $r$ is induced by the functor $\pr_X\circ j_*$.

Therefore, by Corollary \ref{cor-rk1-H2-moduli}, the restriction map
\[i^*\colon \rH^2(\mathsf{Z}_X, \QQ) \to \rH^2(\mathsf{Z}_Y, \QQ)\]
is of rank one when $X$ is non-Hodge-special and $\ker(i^*)=h^{\perp}$, where $h$ is the natural polarization on $\mathsf{Z}_X\cong M^X_{\sigma_X}(2,1)$. Furthermore, a deformation argument using Theorem \ref{thm-relative-stability-fourfold} and \ref{thm-relative-stability-threefold} shows that this holds for a general $X$ as well. Subsequently, based on Theorem~\ref{atomic-criterion-thm} and Theorem \ref{thm-double-epw}, one might speculate that $\mathsf{Z}_Y\subset \mathsf{Z}_X$ is an atomic Lagrangian submanifold. However, the following result demonstrates that this is not the case.

\begin{proposition}\label{prop-llsvs-not-atomic}
Let $X$ be a cubic fourfold not containing a plane and $Y\subset X$ be a smooth hyperplane section. Then \begin{enumerate}
    \item $i\colon \mathsf{Z}_Y\hookrightarrow \mathsf{Z}_X$ is not an atomic Lagrangian submanifold, but

    \item $i_*\cL\in \Coh(\mathsf{Z}_X)$ is a stable atomic sheaf for any line bundle $\cL$ on $\mathsf{Z}_Y$ satisfying $$c_1(\cL)=D|_{\mathsf{Z}_Y}\in \rH^2(\mathsf{Z}_Y, \QQ),$$ where $D$ is the exceptional divisor of the blow-up of $\mathsf{Z}_X$ along $X\subset \mathsf{Z}_X$.
\end{enumerate}

\end{proposition}

\begin{proof}
We have already seen that $i^*$ is of rank one. Next, we proceed to compute $\omega_{\mathsf{Z}_Y}$. Assume that $X\subset \PP(V_6)$ for a $6$-dimensional vector space $V_6$ and $Y=X\cap \PP(V_5)$ for a $5$-dimensional vector space $V_5$. Recall that, by the construction of \cite{LLSvS17}, we have a natural $\PP^2$-fibration $$f_X\colon \mathsf{Z}'_X\rightarrow \Gr(4,V_6),$$ induced by the Stein factorization of the natural morphism 
$M_{3}(X)\rightarrow\Gr(4,V_6)$, which maps a twisted cubic to its generating $\PP^3$. Moreover, $\mathsf{Z}'_X$ is smooth by \cite[Theorem 4.7]{LLSvS17}. Similarly, there exists a smooth subvariety $\mathsf{Z}'_Y\subset \mathsf{Z}'_X$ induced by $M_3(Y)\subset M_3(X)$, and $f_X$ induces a morphism $f_Y\colon \mathsf{Z}'_Y\to \Gr(4,V_5)$, which maps a twisted cubic on $Y$ to its generating $\PP^3$ as well. Therefore, it is straightforward to check that there is a Cartesian diagram
% https://q.uiver.app/#q=WzAsNCxbMCwwLCJcXG1hdGhzZntafSdfWSJdLFsxLDAsIlxcbWF0aHNme1p9J19YIl0sWzAsMSwiXFxHcig0LFZfNSkiXSxbMSwxLCJcXEdyKDQsVl82KSJdLFsyLDMsIiIsMCx7InN0eWxlIjp7InRhaWwiOnsibmFtZSI6Imhvb2siLCJzaWRlIjoidG9wIn19fV0sWzAsMSwiIiwwLHsic3R5bGUiOnsidGFpbCI6eyJuYW1lIjoiaG9vayIsInNpZGUiOiJ0b3AifX19XSxbMSwzLCJmX1giLDJdLFswLDIsImZfWSIsMl1d
\[\begin{tikzcd}
	{\mathsf{Z}'_Y} & {\mathsf{Z}'_X} \\
	{\Gr(4,V_5)} & {\Gr(4,V_6)}
	\arrow[hook, from=2-1, to=2-2]
	\arrow[hook, from=1-1, to=1-2]
	\arrow["{f_X}"', from=1-2, to=2-2]
	\arrow["{f_Y}"', from=1-1, to=2-1]
\end{tikzcd}\]
In particular, $\mathsf{Z}'_Y$ is the zero locus of a regular section of $f_X^*\cU_{\Gr(4,V_6)}^\vee$, where $\cU_{\Gr(4,V_6)}$ is the tautological subbundle on $\Gr(4,V_6)$. Thus, we have 
\[\omega_{\mathsf{Z}'_Y}\cong \omega_{\mathsf{Z}'_X}|_{\mathsf{Z}'_Y}\otimes \det(\mathcal{N}_{\mathsf{Z}'_Y/\mathsf{Z}'_X})\cong (\omega_{\mathsf{Z}'_X}\otimes f^*_X\oh_{\Gr(4,V_6)}(1))|_{\mathsf{Z}'_Y},\]
where $\oh_{\Gr(4,V_6)}(1)$ is the very ample generator of $\Pic(\Gr(4,V_6))$.

According to \cite[Theorem 4.11]{LLSvS17}, there is an embedding $X\hookrightarrow \mathsf{Z}_X$ such that a point $[C]\in M_3(X)$ maps to $p\in X\subset \mathsf{Z}_X$ if and only if $C$ has $p$ as an embedded point. Moreover, the blow-up of $\mathsf{Z}_X$ along $X$ is $\mathsf{Z}_X'$. We denote by $D\subset \mathsf{Z}_X'$ the exceptional divisor, and let $\phi\colon \mathsf{Z}_X'\to \mathsf{Z}_X$ be the blow-up morphism. As any two twisted cubics with embedded points in $Y$ have the same embedded point if and only if they span the same $\PP^3$, we have an identification $\mathsf{Z}_Y'=\mathsf{Z}_Y$. Then we obtain a commutative diagram
% https://q.uiver.app/#q=WzAsNCxbMCwwLCJcXG1hdGhzZntafV9ZIl0sWzEsMCwiXFxtYXRoc2Z7Wl9ZfSJdLFsxLDEsIlxcbWF0aHNme1p9X1giXSxbMCwxLCJcXG1hdGhzZntafV9YJyJdLFszLDIsIlxccGhpIiwyXSxbMSwyLCJpIiwwLHsic3R5bGUiOnsidGFpbCI6eyJuYW1lIjoiaG9vayIsInNpZGUiOiJ0b3AifX19XSxbMCwzLCIiLDIseyJzdHlsZSI6eyJ0YWlsIjp7Im5hbWUiOiJob29rIiwic2lkZSI6InRvcCJ9fX1dLFswLDEsIiIsMCx7InN0eWxlIjp7ImhlYWQiOnsibmFtZSI6Im5vbmUifX19XSxbMCwxLCIiLDEseyJvZmZzZXQiOi0xLCJzdHlsZSI6eyJoZWFkIjp7Im5hbWUiOiJub25lIn19fV1d
\[\begin{tikzcd}
	{\mathsf{Z}_Y} & {\mathsf{Z}_Y} \\
	{\mathsf{Z}_X'} & {\mathsf{Z}_X}
	\arrow[no head, from=1-1, to=1-2]
	\arrow[shift left, no head, from=1-1, to=1-2]
	\arrow[hook, from=1-1, to=2-1]
	\arrow["i", hook, from=1-2, to=2-2]
	\arrow["\phi"', from=2-1, to=2-2]
\end{tikzcd}\]

By \cite[pp.38]{LLSvS17}, there exists an ample line bundle $L$ on $\mathsf{Z}_X$ such that we have an isomorphism $\oh_{\mathsf{Z}_X'}(D)\otimes f^*_X\oh_{\Gr(4,V_6)}(1)\cong \phi^*L$. Moreover, according to  \cite[pp.37]{LLSvS17}, we have $\omega_{\mathsf{Z}'_X}\cong \oh_{\mathsf{Z}_X'}(3D)$, which implies
\[\omega_{\mathsf{Z}_Y}\cong (\oh_{\mathsf{Z}_X'}(3D)\otimes f^*_X\oh_{\Gr(4,V_6)}(1)))|_{\mathsf{Z}_Y}\cong i^*L\otimes \oh_{\mathsf{Z}_X'}(2D)|_{\mathsf{Z}_Y}.\]
Given that the restriction map
\[i^*\colon \rH^2(\mathsf{Z}_X, \QQ) \to \rH^2(\mathsf{Z}_Y, \QQ)\]
has rank one and $L$ is ample, we know that $i^*c_1(L)\in \rH^2(\mathsf{Z}_Y, \QQ)$ generates $\im(i^*)$. Therefore, if $$c_1(\mathsf{Z}_Y)=-2D|_{\mathsf{Z}_Y}-i^*c_1(L)\in \im(i^*)\subset \rH^2(\mathsf{Z}_Y, \QQ),$$ then $D|_{\mathsf{Z}_Y}\in \im(i^*)$ is either ample, zero, or anti-ample. As $\mathcal{O}_Y(D|_{Y})\cong \oh_Y(-1)$ by \cite[Proposition 4.5]{LLSvS17}, we know that $-D|_{\mathsf{Z}_Y}$ is ample. In particular, we obtain $$\rH^0(\mathsf{Z}_Y,\oh_{\mathsf{Z}_Y}(D|_{\mathsf{Z}_Y}))=\rH^4(\mathsf{Z}_Y,\omega_{\mathsf{Z}_Y}(-D|_{\mathsf{Z}_Y}))=0$$ by Kodaira vanishing theorem, which leads to a contradiction since $Y=D\cap \mathsf{Z}_Y\in |\oh_{\mathsf{Z}_Y}(D|_{\mathsf{Z}_Y})|$. Therefore, we conclude that $c_1(\mathsf{Z}_Y)\notin \im(i^*)$ and $\mathsf{Z}_Y\subset \mathsf{Z}_X$ is not atomic by Theorem \ref{atomic-criterion-thm}. This completes the proof of (1).

Now part (2) directly follows from Proposition \ref{prop-general-criterion} and Remark \ref{rmk-prop}, using the description of~$c_1(\mathsf{Z}_Y)$ above.
\end{proof}

Applying the argument in Theorem \ref{thm-bundle} to the atomic sheaf in Proposition \ref{prop-llsvs-not-atomic}, we have the following result.

\begin{corollary}\label{cor-llsvs-bundle}
Let $M$ be a hyper-K\"ahler manifold of $\mathrm{K3^{[4]}}$-type. Then there exists a stable, atomic, projectively hyperholomorphic twisted bundle on $M$ with a $10$-dimensional deformation space.
\end{corollary}

\begin{proof}
Since $i\colon \mathsf{Z}_Y\hookrightarrow \mathsf{Z}_X$ is a closed embedding that realizes $\mathsf{Z}_Y$ as the zero locus of a regular section of a rank $4$ bundle on $\mathsf{Z}_X$ as in the proof of Proposition \ref{prop-llsvs-not-atomic}, we have 
\[\Ext^1_{\mathsf{Z}_X}(i_*\cL, i_*\cL)\cong \rH^1(\mathsf{Z}_Y, \CC)\cong \CC^{10},\]
where $\cL$ is defined in Proposition \ref{prop-llsvs-not-atomic}. Alternatively, we can also use Lemma \ref{lem-immerse-ext1}, where the family $L_V\to V$ is constructed from the relative Bridgeland moduli space of the smooth universal hyperplane in $X$.

We know that the locus of $\mathsf{Z}_X$ form an open dense subset of the moduli space of projective hyper-K\"ahler manifold of degree $3$ and divisibility $2$ (cf.~\cite{LLSvS17}). Then by the Torelli Theorem \cite[Theorem 8.1]{Huy99}, we can take $X$ such that $\rH^{1,1}_{\ZZ}(\mathsf{Z}_X)=\langle h, f \rangle$, where $h$ is the natural polarization and $f$ induces a Lagrangian fibration $\mathsf{Z}_X\to \PP^4$ such that $q(h,f)=d$ for a suitable positive integer $d$. Moreover, we have $f^4.i_*[\mathsf{Z}_Y]\neq 0$ by $\ker(i^*)=h^{\perp}$. Hence, $\mathsf{Z}_Y$ is finite over $\PP^4$ as in Theorem \ref{thm-bundle}, since $\im(i^*)$ is spanned by an ample class. Now, the rest of the argument in the proof of Theorem \ref{thm-bundle} works without any change.
\end{proof}

%\begin{remark}\label{rmk-llsvs}
%Although $i\colon \mathsf{Z}_Y\hookrightarrow \mathsf{Z}_X$ is not atomic, it is likely that $i_*\oh_{\mathsf{Z}_Y}(D|_{\mathsf{Z}_Y})\in \Coh(\mathsf{Z}_X)$ is an atomic sheaf. This can be verified once we have a higher-dimensional version of \cite[Theorem 4.6.6]{bottini23thesis}.
%\end{remark}

\subsubsection{Fano variety of lines}
Let $X$ be a cubic fourfold and $Y\subset X$ be a general hyperplane section. Recall that the Fano variety of lines of $X$, denoted by $F(X)$,  is a hyper-K\"ahler fourfold of $\mathrm{K3^{[2]}}$-type by \cite{beauville:fano-variety-cubic-4fold}. The classical result of \cite{voisin1992stabilite} shows that the Fano surface of lines $F(Y)\subset F(X)$ is Lagrangian. Using the explicit calculation of the Chern classes of $F(Y)$, it is proved in \cite{markman:rank-1-obstruction} and \cite{beckmann:atomic-object} that $F(Y)$ is atomic in $F(X)$. 

Besides this construction, the authors of \cite{IM08cubic} consider another Lagrangian surface of  $F(X)$ arising from a cubic fivefold. Let $Z$ be a cubic fivefold, where $X$ is a hyperplane section of $Z$. We consider the Fano variety of planes of $Z$, denoted by $F_2(Z)$. When $X$ and $Z$ are general, $F_2(Z)$ is a smooth surface and there is a morphism $i\colon F_2(Z)\to F(X)$ such that $F_2(Z)$ is a connected immersed Lagrangian submanifold of $F(X)$ (cf.~\cite[Proposition 4]{IM08cubic}). Using a standard deformation argument of relative Hilbert schemes, one can show that $i^*\colon \rH^2(F(X), \QQ)\to \rH^2(F_2(Z), \QQ)$ is of rank one and $\ker(i^*)=h^{\perp}$, where $h$ is the natural Pl\"ucker polarization on $ F(X)$. However, we have the following result.

\begin{proposition}\label{prop-fanoplane-not-atomic}
Let $X$ be a general cubic fourfold and $Z\supset X$ be a general cubic fivefold. Then

\begin{enumerate}
    \item the immersed Lagrangian submanifold $i\colon F_2(Z)\to F(X)$ is not atomic, but 

    \item $i_*\cL\in \Coh(F(X))$ is a stable atomic sheaf for any line bundle $\cL$ on $F_2(Z)$ satisfying $c_1(\cL)=7(h_2-h_1)\in\rH^2(F_2(Z), \QQ)$, where $h_1$ is the pull-back of the Pl\"ucker class via the natural morphism $F_2(Z)\to \Gr(3, 7)$ and $h_2$ is the pull-back of the Pl\"ucker class $h$ on $F(X)$ via $i$.
\end{enumerate}
\end{proposition}

\begin{proof}
First, we prove (1). Assume that $F_2(Z)$ is atomic. Then, by Theorem \ref{atomic-criterion-thm}, the restriction map $i^*\colon \rH^2(F_2(Z), \QQ)\to \rH^2(F(X), \QQ)$ is of rank one and $c_1(F_2(Z))\in \im(i^*)$. In particular, $c_1(F_2(Z))$ is proportional to $i^*\alpha$ for any $\alpha\in \rH^2(F(X), \QQ)$. 

We assume that $X=Z\cap \PP(V_6)\subset \PP(V_7)$. Let $\cT'$ be the pull-back of the rank two tautological subbundle on $\Gr(2,V_6)$ via the natural embedding $F(X)\hookrightarrow \Gr(2,V_6)$ and $\mathcal{W}:=V_6\otimes \oh_{F(X)}$. In \cite[Lemma 6]{IM08cubic}, it is shown that $i$ can be decomposed as
\[F_2(Z)\hookrightarrow \PP:=\PP_{F(X)}(\mathcal{W}/\cT')\to F(X),\]
where the first embedding, denoted by $t$, realizes $F_2(Z)$ as the zero locus of a regular section of the rank six bundle $$\mathcal{K}:=\ker(\Sym^3\cS^{\vee}\twoheadrightarrow \Sym^3\cT^{\vee}).$$ Here, $\cT$ is the pull-back of $\cT'$ to $\PP$, and $\mathcal{S}$ is the rank three bundle induced by the natural map $\PP\to \Gr(3, V_7)$ such that $\cS/\cT=\oh_{\PP}(-1)$. Therefore, we obtain $$c_1(F_2(Z))=t^*(c_1(\cK)+c_1(\PP)).$$
If we denote by $\alpha:=c_1(\oh_{\PP}(1))$ and $\beta:=c_1(\cT^{\vee})$, then we get $c_1(\cK)=10\alpha+3\beta$ and $c_1(\PP)=4\alpha+\beta$. Hence,  $c_1(F_2(Z))=14t^*\alpha+4t^*\beta$. However, as $h_2=t^*\beta$ is the pull-back of the class of the Pl\"ucker line bundle on $F(X)$, we have $0\neq t^*\beta\in \im(i^*)$ and $\im(i^*)$ is spanned by $t^*\beta$. Thus, as  $c_1(F_2(Z))\in \im(i^*)$, we have $t^*\alpha=0\in \rH^2(F_2(Z), \QQ)$. Since the fundamental class of $F_2(Z)$ in~$\PP$ is $63\alpha^4c_2(\cT^{\vee})$, we obtain $63\alpha^5c_2(\cT^{\vee})=\alpha.t_*[F_2(Z)]=t_*t^*\alpha=0.$ However, as computed in \cite[Lemma 6]{IM08cubic}, we have $\alpha^5=-\beta.\alpha^4$. This implies that $$0=\alpha.t_*[F_2(Z)]=-\beta.t_*[F_2(Z)],$$ which is a contradiction since $\beta$ is the class of an ample line bundle.

Next, we prove (2). Since $\cS/\cT=\oh_{\PP}(-1)$, we have $h_1-h_2=t^*\alpha$. Therefore, (2) follows from the description of $c_1(F_2(Z))$, Proposition \ref{prop-general-criterion}, and Remark \ref{rmk-prop}.

\end{proof}

Applying the argument in Theorem \ref{thm-bundle} to the atomic sheaf in Proposition \ref{prop-fanoplane-not-atomic}, we have the following result.

\begin{corollary}\label{cor-fano-bundle}
Let $M$ be a hyper-K\"ahler manifold of $\mathrm{K3^{[2]}}$-type. Then there exists a stable, atomic, projectively hyperholomorphic twisted bundle on $M$ with a $42$-dimensional deformation space.
\end{corollary}

\begin{proof}
Since $i\colon F_2(Z)\to F(X)$ is a closed embedding away from finitely many points and there exists a nice family of immersed Lagrangians of $F(X)$ as required in Lemma \ref{lem-immerse-ext1} (cf.~\cite{IM08cubic}), we have 
\[\Ext^1_{F(X)}(i_*\cL, i_*\cL)\cong \rH^1(F_2(Z), \CC)\cong \CC^{42}\]
by applying Lemma \ref{lem-immerse-ext1}, where $\cL$ is defined in Proposition \ref{prop-fanoplane-not-atomic}. 

We know that the locus of $F(X)$ form an open dense subset of the moduli space of projective hyper-K\"ahler manifold of degree $6$ and divisibility $2$ (cf.~\cite{beauville:fano-variety-cubic-4fold}). Then by the Torelli Theorem \cite[Theorem 8.1]{Huy99}, we can take $X$ such that $\rH^{1,1}_{\ZZ}(F(X))=\langle h, f \rangle$, where $h$ is the natural Pl\"ucker polarization and $f$ induces a Lagrangian fibration $F(X)\to \PP^2$ such that $q(h,f)=d$ for a suitable positive integer $d$. Moreover, we have $f^2.i_*[F_2(Z)]\neq 0$ by $\ker(i^*)=h^{\perp}$. Hence, $i(F_2(Z))$ is finite over $\PP^2$ as in Theorem \ref{thm-bundle}, since $\im(i^*)$ is spanned by an ample class. Now, the rest of the argument in the proof of Theorem \ref{thm-bundle} works without any change.
\end{proof}

\bibliography{atomic}

\bibliographystyle{alpha}

\end{document}